\documentclass[11pt]{amsart}

\newif\ifdebug
\debugfalse

\newif \iffig
\figtrue
\newif \iftable
\tabletrue
\usepackage{array}
\newcolumntype{C}[1]{>{\centering\arraybackslash}p{#1}}

\usepackage{upgreek}
\renewcommand{\tau}{\uptau}

\usepackage{sfilip_package_settings}
\usepackage{sfilip_abbreviations}
\usepackage{sfilip_thm_style_long}

\begin{document}
%====================================================
%							Title & Date
%====================================================
%	also: Add it in the hyperref setup
%
\title[Uniformizing VHS, Anosov reps., Lyapunov exponents]{Uniformization of some weight $3$ variations of Hodge structure, Anosov representations, and Lyapunov exponents}
%	change to \title[short title]{Actual title}
%	if what appears on headers is too long
%
%---------------------------------------
%				Date of revision
\thanks{Revised \textsc{\today} }

%				Date First posted
\date{September 2021}

%====================================================
%									Author info
%====================================================

\author{
	Simion Filip
}
%		Address
\address{
	\parbox{0.5\textwidth}{
		Department of Mathematics\\
		University of Chicago\\
		5734 S University Ave\\
		Chicago, IL 60637\\}
}
\email{{sfilip@math.uchicago.edu}}
%
%
%====================================================
%										Abstract
%====================================================
%
\begin{abstract}
We develop a class of uniformizations for certain weight $3$ variations of Hodge structure (VHS).
The analytic properties of the VHS are used to establish a conjecture of Eskin, Kontsevich, M\"oller, and Zorich on Lyapunov exponents.
Additionally, we prove that the monodromy representations are log-Anosov, a dynamical property that has a number of global consequences for the VHS.
We establish a strong Torelli theorem for the VHS and describe appropriate domains of discontinuity.

Additionally, we classify the hypergeometric differential equations that satisfy our assumptions.
We obtain several multi-parameter families of equations, which include the mirror quintic as well as the six other thin cases of Doran--Morgan and Brav--Thomas.
\end{abstract}
%====================================================

%====================================================
%									Title, ToC, FixMe
%
\maketitle
%
%			ToC
\noindent\hrulefill
\tableofcontents
% \nointerlineskip
\noindent \hrulefill
%---------------------------------------
%					Things to be fixed
\ifdebug	
	\listoffixmes
\fi
%====================================================

%%%%%%%%%%%%%%%%%%%%%%%%%%%%%%%%%%%%%%%%%%%%%%%%%%%%%%%%%%%%%%%%%%%%%%%%%%%%%%%
%%% 				Start of Section: Introduction
%%%%%%%%%%%%%%%%%%%%%%%%%%%%%%%%%%%%%%%%%%%%%%%%%%%%%%%%%%%%%%%%%%%%%%%%%%%%%%%

\section{Introduction}
	\label{sec:introduction}

The global geometry of higher weight variations of Hodge structure remains mysterious.
In this paper we investigate one of the simplest nonclassical situations, with Hodge numbers $(1,1,1,1)$, subject to a nonvanishing condition on a component of the second fundamental form.
This covers for instance the case of the mirror quintic, or Dwork, family as well as six other families among the fourteen identified by Doran and Morgan \cite{Doran_Morgan} that were later proved to have thin monodromy by Brav--Thomas \cite{BravThomas2014_Thin-monodromy-in-Sp4}.
As we show below, even in the class of hypergeometric local systems, a much larger class consisting of several multi-parameter families is covered by our analysis (a one-parameter subfamily was analyzed by completely different methods in \cite{FilipFougeron2021_A-cyclotomic-family-of-thin-hypergeometric-monodromy-groups-in-Sp4R}).

The interest in the geometry of these variations of Hodge structure (VHS for short) comes from many directions, including mirror symmetry, see \cite{CandelasOssaGreen1991_A-pair-of-Calabi-Yau-manifolds-as-an-exactly-soluble-superconformal-theory} for the stunning application of the mirror quintic VHS to count rational curves on the generic quintic $3$-fold.
The original motivation of the present work was a conjecture of Eskin, Kontsevich, M\"oller, and Zorich \cite{EskinKontsevichMoller2018_Lower-bounds-for-Lyapunov-exponents-of-flat-bundles-on-curves} predicting a formula for the sum of Lyapunov exponents in terms of degrees of Hodge bundles.
This formula was discovered to hold numerically by Kontsevich in seven of the fourteen families from \cite{Doran_Morgan}.
Around the same time it was discovered by Brav and Thomas that the same seven families had thin monodromy, while Singh and Venkataramana \cite{SinghVenkataramana2014_Arithmeticity-of-certain-symplectic-hypergeometric-groups} showed the monodromy in the remaining seven surjects onto a lattice in $\Sp_4(\bZ)$.

In this paper we establish the conjectured formula for Lyapunov exponents stated in \cite{EskinKontsevichMoller2018_Lower-bounds-for-Lyapunov-exponents-of-flat-bundles-on-curves} and in fact several other properties of a VHS which satisfies ``assumption A'' as per \autoref{def:assumption_a} below.
In short, it requires that the second fundamental form $\sigma_{2,1}\colon \cV^{2,1}\to \cV^{1,2}\otimes \Omega^1_X$ is an isomorphism (and at the cusps it is an isomorphism if we use the logarithmic cotangent bundle).
Geometrically, this condition can be viewed as a generalization of the ``tiling'' condition for a Fuchsian triangle group, which requires the angles to be of the form $\pi/N$.

We show that the monodromy representation is in fact an Anosov representation, a concept introduced by Labourie \cite{LabourieAnosov}, except that the image contains unipotent elements and this leads to what we call a ``log-Anosov representation''.
Basic facts on this class of representations were also recently developed in the greater generality of relatively hyperbolic groups by Zhu \cite{Zhu2019_Relatively-dominated-representations} and Kapovich--Leeb \cite{KapovichLeeb2018_Relativizing-characterizations-of-Anosov-subgroups-I} under the name relatively dominated, or relatively Anosov, representations.
In the case of Fuchsian groups, an equivalent notion was studied by Canary--Zhang--Zimmer \cite{CanaryZhangZimmer2021_Cusped-Hitchin-representations-and-Anosov-representations-of-geometrically-finite} under the name of cusped Anosov representations.
% The Anosov property amounts to a uniform separation of the singular values of the monodromy matrices, expressed in terms of the hyperbolic distance between the two points on the universal cover corresponding to the monodromy matrix.
% Note however that the Anosov property is a coarse-geometric property, independent of the choice of Riemann surface structure on the base.

Associated to Anosov, or log-Anosov, representations are domains of discontinuity in flag manifolds.
Our main results show that some of these domains of discontinuity have natural interpretations in terms of Hodge theory.

Two points are of interest.
First, the domains of discontinuity are ``nonclassical'' in the sense that they are not of the form $G/M$ where $M$ is a compact subgroup of the semisimple Lie group $G$.
In fact, lattices in $G$ will not have domains of discontinuity of the type constructed below.
Second, it is well-known that the Griffiths transversality condition imposes severe restrictions on the dimension of the image of the period map in the Griffiths period domain $\cD$ (which incidentally is of the form $G/M$ for $M$ compact).
In particular, most points in $\cD$ (a full measure set and with Baire-negligible complement) are not in the image of any period map.
It might therefore be of interest that linear-algebraic data in other kinds of domains, described below, can be put into bijective correspondence with data coming from the VHS.

As a consequence of our methods we obtain a strong Torelli theorem for the VHS under consideration, using a pseudo-hermitian period domain parametrizing the Griffiths intermediate Jacobians.
Further uniformization results are described below.

As for dynamical results on log-Anosov representations, we introduce the notion of ``stable point'' for a general closed subgroup of $\SL(V)$ acting on $\bP(V)$ and use this to deduce the standard proper discontinuity criteria.
We also obtain a minimality result for the action on a limit set which might be of interest because the analogous statement \emph{does not} hold for Hitchin representations.
Additionally, we can classify the points with rational coordinates that occur on the limit set.

\laternote{Include explicit nonvanishing power series}

%%=============================================================================
%%					start of subsec: Main results

\subsection*{Main results}
	\label{ssec:main_results}

Let $X$ be a complete finite volume Riemann surface and $\bV\to X$ a variation of Hodge structure with Hodge numbers $(1,1,1,1)$ satisfying assumption A from \autoref{def:assumption_a}.
Let also $\bW:=\Lambda^2_{\circ}\bV$ be the reduced second exterior power, with the line corresponding to the invariant symplectic form removed, so it is a VHS with Hodge numbers $(1,1,1,1,1)$.
Fix a basepoint $x_0\in X$ (omitted most of the time from the notation) and let $\rho\colon \pi_1(X)\to \Sp_4(V_\bR)$ be the monodromy representation, where $V_{\bR}$ is the fiber of $\bV$ over the basepoint, and similarly for $\SO_{2,3}(W_{\bR})$ and $\bW$.

Our first result settles the conjectures made by Eskin, Kontsevich, M\"oller, and Zorich in \cite{EskinKontsevichMoller2018_Lower-bounds-for-Lyapunov-exponents-of-flat-bundles-on-curves}.

\begin{theoremintro}[Formula for the sum of Lyapunov exponents]
	\label{thmintro:formula_for_the_sum_of_lyapunov_exponents}
	With the above assumptions let $\lambda_1\geq \lambda_2\geq 0$ be the nonnegative Lyapunov exponents of the geodesic flow on $\bV$.
	Then
	\[
		\lambda_1 + \lambda_2 = \frac{\deg \cV^{0,3}_{ext} + \deg \cV^{1,2}_{ext}}{\chi(X)}
	\]
	where $\chi(X)$ is the (orbifold) Euler characteristic of $X$, $\deg$ denotes the (orbifold) degree of a complex line bundle, and the subscript $ext$ denotes the Deligne extension of a bundle across punctures.

	Furthermore, for a vector $w\in W_{\bR}$, the ``bad locus'' (see \cite[\S4]{EskinKontsevichMoller2018_Lower-bounds-for-Lyapunov-exponents-of-flat-bundles-on-curves}) is empty if $[w]\in \LGr(V_{\bR})\setminus \Omega_L$, where $\Omega_L$ is the domain of discontinuity from \autoref{corintro:domains_of_discontinuity}(i).
	For a vector in $\Omega_L$ or $\bS^{1,3}(W_{\bR})$, the bad locus consists of precisely one point.

	In the hypergeometric examples, the MUM Lagrangian belongs to the complement of $\Omega_L$, so the nonvanishing Conjecture~6.4 of \cite{EskinKontsevichMoller2018_Lower-bounds-for-Lyapunov-exponents-of-flat-bundles-on-curves} holds.
\end{theoremintro}
\noindent The reader interested only in the proof of the formula for the sum of Lyapunov exponents from \autoref{thmintro:formula_for_the_sum_of_lyapunov_exponents} can read just \autoref{sec:hodge_theory_and_growth}.
The rest of the paper is devoted to deducing the Anosov-type properties of the monodromy, as well as analyzing further the domains of discontinuity.

Let $\wtilde{X}$ be the universal cover associated to the basepoint $x_0$, whose lift to $\wtilde{X}$ we will continue to denote by $x_0$.
Equipped with the complete hyperbolic metric, $\wtilde{X}$ can be identified with the hyperbolic plane $\bH^2$ and we will write $\partial \wtilde{X}$ for its boundary, which can be identified with $\bP^1(\bR)$.

\begin{theoremintro}[log-Anosov property]
	\label{thmintro:log_anosov_property_introduction}
	The monodromy representation $\rho$ is log-Anosov, in the sense of \autoref{def:log_anosov_representation}.
	Specifically, there exist constants $c,\ve>0$ such that
	\[
		\mu_1(\rho(\gamma))-\mu_2(\rho(\gamma))\geq \ve \cdot \dist(x_0,\gamma x_0) - c
	\]
	where $\mu_1\geq \mu_2 \geq 0$ are the logarithms of the singular values and $\dist$ refers to the hyperbolic distance.

	Furthermore, there exists a continuous $\rho$-equivariant map
	\[
		\xi \colon \partial \wtilde{X}\to \bP(V_{\bR})
	\]
	that is dynamics-preserving and transversal.
\end{theoremintro}
Transversality means that for two distinct boundary points $p_1\neq p_2$ the line $\xi(p_1)$ is disjoint from the symplectic-orthogonal of $\xi(p_2)$,
while dynamics-preserving is detailed in \autoref{sssec:transversality_and_dynamics}.
Informally, it means that points on the boundary fixed by $\gamma\in \pi_1(X)$ map to points in $\bP(V_{\bR})$ fixed by $\rho(\gamma)$.

\autoref{thmintro:log_anosov_property_introduction} is a concatenation of \autoref{thm:anosov_property_from_assumption_A}, which establishes the singular value gap using Hodge-theoretic techniques, and \autoref{thm:existence_of_boundary_maps} which gives the boundary map.
Recall also that the singular values of a matrix $g$ are the square roots of the eigenvalues of $g\cdot g^{\dag}$, and in the case of symplectic matrices the singular values have a symmetry under $\lambda\mapsto \tfrac 1\lambda$, so in the statement of the theorem we only consider the logarithms which yield positive numbers.

A standard consequence of the log-Anosov property is the existence of a generous supply of domains of discontinuity.
Recall that the real Lagrangian Grassmannian $\LGr(V_\bR)$ is the quadric of null vectors in $\bP(W_{\bR})$, were $W=\Lambda^2_{\circ}V$ is the reduced second exterior power.

\begin{corollaryintro}[Domains of Discontinuity]
	\label{corintro:domains_of_discontinuity}
	Let $\Gamma\subset \Sp_4(V_{\bR})$ be the image of the monodromy group.
	\begin{enumerate}
		\item In the real Lagrangian Grassmannian $\LGr(V_\bR)$ there exists an open nonempty set $\Omega_L$ on which $\Gamma$ acts properly discontinuously.
		\item The pseudosphere $\bS^{1,3}$ of unit vectors in $W_{\bR}$ is a domain of discontinuity for $\Gamma$.
		\item The complex Grassmannian of Lagrangians of signature $(1,1)$ for the indefinite hermitian metric on $V_{\bC}$, denoted $\LGr^{1,1}(V_{\bC})$, is also a domain of discontinuity for $\Gamma$.
		\item In the complex projective space $\bP(V_{\bC})$ there exists an open, nonempty set $\Omega_P$ on which $\Gamma$ acts properly discontinuously.
	\end{enumerate}
\end{corollaryintro}
\noindent These results are proved in \autoref{thm:real_uniformization_lagrangian}, \autoref{thm:real_uniformization_pseudo_hyperbolic}, \autoref{thm:proper_discontinuity_on_indefinite_lagrangians}, and \autoref{thm:proper_discontinuity_on_complex_projective_space} respectively.
Note that the log-Anosov property, as well as the existence of the above domains of discontinuity, is a topological, or coarse-geometric future of the monodromy group.
We can show, however, that some of the domains of discontinuity are uniformized using the variation of Hodge structure.

Domains of discontinuity of this flavor were constructed by Guichard--Wienhard \cite[Part 2]{GuichardWienhard2012_Anosov-representations:-domains-of-discontinuity-and-applications}, see also Benoist \cite{Benoist_Actions-propres-sur-les-espaces-homogenes-reductifs} who considered domains of discontinuity that are homogeneous under the action of the ambient Lie group.
We provide a self-contained treatment in \autoref{ssec:stability_and_the_numerical_criterion} in terms of the notion of ``stable point'', though the mechanism behind the proof of the proper discontinuity criterion is essentially the same.

\subsubsection*{Uniformization from Hodge structures}
	% \label{sssec:uniformization_from_hodge_structures}
To state the uniformization results, we need a linear-algebraic construction.
Let $F^2 = \cV^{3,0}\oplus \cV^{2,1}$ be the middle term of the Hodge filtration.
To such a complex Lagrangian associate the set of real Lagrangians $L$ such that their complexification is \emph{not} transverse to $F^2$, i.e. $L_{\bC}\cap F^2\neq 0$.
We denote this set by $\beta(F^2)\subset \LGr(V_{\bR})$ and call it an ``electron'' (see \autoref{sssec:on_electrons}), it is a real $1$-dimensional circle in $\LGr(V_{\bR})$ which is moreover timelike for the causal structure on $\LGr(V_{\bR})$.
This allows us to uniformize the domain of discontinuity $\Omega_L\subset \LGr(V_{\bR})$:

\begin{theoremintro}[Uniformization of real Lagrangians]
	\label{thmintro:uniformization_of_real_lagrangians}
	Let $Bad\to X$ be the circle bundle of electrons over $X$, and $\wtilde{Bad}$ its lift to the universal cover $\wtilde{X}$.
	Then the natural developing map $\wtilde{Bad}\xrightarrow{\Dev}\LGr(V_{\bR})$ is a diffeomorphism between $\wtilde{Bad}$ and the domain of discontinuity $\Omega_L$.
\end{theoremintro}
\noindent This is a combination of \autoref{thm:real_uniformization_lagrangian} and \autoref{thm:anosov_property_from_assumption_A}(ii).
For why the circle bundle is called ``bad'', see \autoref{sssec:on_electrons}.

There is an extension of the uniformization result to the real pseudosphere $\bS^{1,3}(W_{\bR})$.
Recall that $\bW=\left(\Lambda^2_{\circ}\bV\right)$ is the variation of Hodge structure associated to the reduced second exterior power of $\bV$, Tate-twisted from weight $6$ to weight $4$, with Hodge numbers $(1,1,1,1,1)$, and Hodge bundles denoted $\cW^{p,q}$.

\begin{theoremintro}[Uniformization of pseudosphere]
	\label{thmintro:uniformization_of_pseudo_sphere}
	Let $\cW^{1,3}_{\times}\to X$ be the $\bC^{\times}$-bundle of nonzero vectors in the Hodge bundle $\cW^{1,3}$, and let $\wtilde{\cW}^{1,3}_{\times}$ be its lift to the universal cover $\wtilde{X}$.

	There is a natural developing map $\wtilde{\cW}^{1,3}_{\times}\xrightarrow{\Dev}\bS^{1,3}(W_{\bR})$ which is a bijective diffeomorphism.
\end{theoremintro}
\noindent The developing map is constructed in \autoref{eqn:developing_map_pseudosphere}, using the middle Hodge bundle $\cW^{2,2}$ as well as $\cW^{1,3}$, to parametrize the unit vectors orthogonal to $\cW^{0,4}$.

Additionally, in \autoref{thm:boundary_values_in_adjoint_representation}, we show that if we consider the adjoint representation VHS, then the map from the universal cover $\wtilde{X}$ to the Grassmannian of $2$-planes $\Gr(2;W_{\bR})$ obtained from $\cW^{4,0}\oplus \cW^{0,4}$, has well-defined boundary values.
These coincide with the limit curve $\xi$ from \autoref{thmintro:log_anosov_property_introduction}.

\subsubsection*{Hypergeometric equations}
	\label{sssec:hypergeometric_equations}
A large class of examples of VHS satisfying assumption A comes from hypergeometric ordinary differential equations (see \autoref{sec:hypergeometric_local_systems} for the relevant notions).
A classification of the hypergeometric parameters which satisfy it is provided by \autoref{thm:good_hypergeometrics} and reads:

\begin{theoremintro}[Good hypergeometric parameters]
	\label{thmintro:good_hypergeometric_parameters}
	The parameters of hypergeometric differential equations that satisfy assumption A are listed in \autoref{table:good_hypgeom_real} and \autoref{table:good_hypgeom_integral}.
\end{theoremintro}
\noindent \autoref{table:good_hypgeom_integral} includes the seven thin cases of Brav--Thomas \cite{BravThomas2014_Thin-monodromy-in-Sp4}, as well as the infinite family from \cite{FilipFougeron2021_A-cyclotomic-family-of-thin-hypergeometric-monodromy-groups-in-Sp4R}.
In particular, assumption A distinguishes the seven thin cases of Brav--Thomas from the remaining seven cases in the Doran--Morgan list.
In the case of an orbifold point of order $N$, with local exponents $0<\alpha_1<\alpha_2<\alpha_3<\alpha_4<1$, the requirement of assumption A is that $N\cdot (\alpha_3-\alpha_2)=1$.
It means that on the finite cover removing the orbifold point, the corresponding second fundamental form is an isomorphism.
Using that $\alpha_3=1-\alpha_2$ it follows that $\alpha_2=\frac{N-1}{2N}$.
A glance at the table from \cite{BravThomas2014_Thin-monodromy-in-Sp4} reveals that this is precisely what distinguishes the thin cases with orbifold points from the remaining cases.
Note that four of the thin cases have cusps instead of orbifold points, where assumption A is analyzed in \autoref{prop:boundedness_of_second_fundamental_form}.

\begin{table}
	\begin{tabular}{ccc}
	\toprule
	$\alpha_{\bullet}$                                     & $\beta_{\bullet}$                            & \quad {\text{conditions}}\\ \hline
	% making table rows wider
	\renewcommand{\arraystretch}{1.5}
	$\left(\mu,\tfrac 12,\tfrac 12,1-\mu\right)$ & $(0,0,0,0)$  & $\mu\in\left(0,\tfrac12\right]$\\
	% $\left(
	% \tfrac 12,\tfrac 12,
	% \tfrac 12,\tfrac 12\right)$  & $\left(0,0,\mu,1-\mu\right)$                 & $\mu\in\left(0,\tfrac12\right)$\\
	 & $\left(0,0,\nu,1-\nu\right)$ & $0<\nu<\mu \leq\tfrac 12$\\
	 \bottomrule
	\end{tabular}
	\caption{\label{table:good_hypgeom_real}Parameters $\mu,\nu$ are real.}

	\renewcommand{\arraystretch}{1}
	\vskip 2em

	\begin{tabular}{cc}
	\toprule
	 $\beta_\bullet$ & \text{ conditions}  \\ \hline
 	\renewcommand{\arraystretch}{1.5}
	 $\left(
	 \tfrac 12,\tfrac 12,\tfrac 12,\tfrac 12
	 \right)$                         & \\
	 $(0,0,0,0)$                      & \\
	 $\left(0,0,\mu,1\!-\!\mu\right)$ & $0<\mu<\frac{2k-1}{2N}$\\
	 $\left(
	 \mu,\tfrac 12,\tfrac 12,1\!-\!\mu
	 \right)$                         & $\tfrac{N-1}{2N}<\mu<\tfrac 12$
	 \\
	$ \left(
	 	\tfrac{M-(2k_M+1)}{2M},
	 	\tfrac{M-1}{2M},
	 	\tfrac{M+1}{2M},
	 	\tfrac{M+(2k_M+1)}{2M}
	 \right)$                         & $\tfrac{2k_M+1}{M} < \tfrac 1N$\\
 	\bottomrule
	\end{tabular}
	\caption{\label{table:good_hypgeom_integral}
	Parameters $M,k_M$ are integers, $\mu$ is real.\\
	Throughout $\alpha_{\bullet}= \left(
	\tfrac{N-(2k+1)}{2N},
	\tfrac{N-1}{2N},
	\tfrac{N+1}{2N},
	\tfrac{N+(2k+1)}{2N}
	\right)$
	with arbitrary integers $k\geq 1, N>2k+1$.}
\end{table}

The tables are constructed from two restrictions.
One is local, and specifies what the exponents of the ODE have to be in order for the second fundamental form $\sigma_{2,1}$ to be an isomorphism (see \autoref{prop:assumption_a_in_the_local_case}).
The second restriction is global and comes from results of Fedorov \cite{Fedorov2018_Variations-of-Hodge-structures-for-hypergeometric-differential-operators-and-parabolic}, and requires the exponents $\alpha_\bullet,\beta_\bullet$ to be arranged in a specific way $\mod 1$, so that the Hodge numbers are $(1,1,1,1)$, see \autoref{sssec:variation_of_hodge_structure_fedorov}.

Let us note that the families above are of several types (in all cases subject to some inequalities on parameters): one type of family depends on four integers $M,N,k_M,k_N$, another type of family depends on two real numbers $\mu,\nu$, and finally there are mixed cases where the family depends on two integers $(N,k_N)$ and a real number $\mu$.

\subsubsection*{Torelli theorems}
	\label{sssec:torelli_theorems}
With the above results, we can gain some insight into the global structure of some moduli spaces of Calabi--Yau $3$-folds.
Specifically, the real uniformization result from \autoref{thmintro:uniformization_of_real_lagrangians} implies a strong Torelli theorem for VHS satisfying assumption A, in particular for the much-studied mirror quintic introduced in \cite{CandelasOssaGreen1991_A-pair-of-Calabi-Yau-manifolds-as-an-exactly-soluble-superconformal-theory} and the six other families from \cite{Doran_Morgan} satisfying our assumption:

\begin{theoremintro}[Strong Torelli using Lagrangians]
	\label{thmintro:strong_torelli_using_lagrangians}
	Suppose that $\bV\to X$ is a VHS satisfying assumption A, and lift it to the universal cover $\wtilde{X}$.
	Then:
	\begin{enumerate}
		\item The maps induced by the Hodge filtration $F^2=\cV^{3,0}\oplus \cV^{2,1}$:
		\begin{align*}
			\wtilde{X}&\xrightarrow{F^2} \LGr^{1,1}(V_{\bC})\\
			\intertext{and taking the quotient by $\pi_1$:}
			{X}&\xrightarrow{F^2} \leftquot{\Gamma}{\LGr^{1,1}(V_{\bC})}\\
		\end{align*}
		are injective.
		\item Furthermore, for $x,y\in \wtilde{X}$, if $L$ is a real Lagrangian such that $\left(F^2(x)\cap L_{\bC}\right)\neq 0 \neq \left(F^2(y)\cap L_{\bC}\right)$ then $x=y$.
	\end{enumerate}
\end{theoremintro}
\noindent A further dichotomy, related to rational Lagrangian subspaces, is contained in \autoref{corintro:dichotomy_for_rational_lagrangian_subspaces}.
\begin{proof}
	The claims are an immediate consequence of \autoref{thmintro:uniformization_of_real_lagrangians}, which associates to $F^2$ the electron $\beta(F^2)$ and establishes injectivity of the corresponding developing map.
\end{proof}
For the mirror quintic family, a \emph{generic} Torelli theorem using the full Griffiths period domain was proved by Usui \cite{Usui2008_Generic-Torelli-theorem-for-quintic-mirror-family}.

\begin{remarkintro}[On Griffiths intermediate Jacobians]
	\label{rmkintro:on_griffiths_jacobians}
	When the real weight $3$ Hodge structure has an underlying integral structure, one can associate to it the Griffiths intermediate Jacobian $V_{\bZ}\backslash V_{\bC}/F^2$.
	In general this is not an abelian variety, and the period domain of such objects is $\LGr^{1,1}(V_{\bC})$, a pseudo-hermitian homogeneous space, in contrast to Siegel spaces parametrizing marked abelian varieties.

	It is not possible to take a Hausdorff quotient of $\LGr^{1,1}(V_{\bC})$ by $\Sp_4(\bZ)$, or any lattice in $\Sp_4(\bR)$.
	However, \autoref{corintro:domains_of_discontinuity} implies that for the monodromy $\Gamma$ of a VHS satisfying assumption A, \emph{it is} possible to take the quotient.
	\autoref{thmintro:strong_torelli_using_lagrangians} then says that the period map to this quotient is injective.
	So it can be viewed a Torelli theorem in the classical sense.
\end{remarkintro}

\subsubsection*{Integral vectors}
	\label{sssec:integral_vectors}
The ``limit set'' curves provided by \autoref{thmintro:log_anosov_property_introduction} are fractal curves (in fact, it is possible and not hard to prove that in hypergeometric examples, because of the rank $1$ unipotent, the limit curve cannot be rectifiable).
\laternote{expand on lack of rectifiability after showing minimality}
Nonetheless, we can classify the rational points on the curve, and similarly for the limit set in the Lagrangian Grassmannian.
This is possible when the monodromy is contained in integer matrices, for instance in the mirror quintic example as well as the six other families of Doran--Morgan \cite{Doran_Morgan} that satisfy assumption A.

\begin{theoremintro}[Rational directions on limit curve]
	\label{thmintro:rational_directions_on_limit_curve}
	Suppose that the image of the monodromy group $\Gamma$ is contained in $\Sp_4(\bZ)$.
	Let $\xi\colon \partial\wtilde{X}\to \bP(V_\bR)$ be the boundary map, and let $\Lambda_L\subset \LGr(V_{\bR})$ be defined as the complement of the domain of discontinuity $\Omega_L\subset \LGr(V_{\bR})$ from \autoref{thmintro:uniformization_of_real_lagrangians}.
	\begin{enumerate}
		\item A line $[v]\in \xi(\partial\wtilde{X})\subset \bP(V_\bR)$ has rational coordinates if and only if there exists a unipotent transformation $\gamma\in \Gamma$ such that $v$ is both in the kernel and image of $\gamma-\id$.
		In particular, the rational vectors on $\xi(\partial\wtilde{X})$ fall into finitely many orbits under the action of $\Gamma$, corresponding to the cusps of $X$.
		\item A line $[w]\in \LGr\left(V_\bQ\right)$ representing a Lagrangian $L_{[w]}\subset V_{\bQ}$ is contained in $\Lambda_L$ if and only if there exists a unipotent $\gamma\in \Gamma$ and $[v]$ as in part (i), fixed by $\gamma$, such that $[v]\subset L_{[w]}$.
	\end{enumerate}
\end{theoremintro}
\noindent This result follows from \autoref{thm:integral_vectors_must_come_from_cusps}, which establishes the result in slightly greater generality.
This result should be useful to build a compactification of the quotients of some of the domains of discontinuity appearing in \autoref{corintro:domains_of_discontinuity}, particularly $\Omega_L\subset \LGr(V_\bR)$.

\autoref{thmintro:rational_directions_on_limit_curve} above, combined with the strong Torelli theorem provides an interesting property of rational Lagrangian subspaces:
\begin{corollaryintro}[Dichotomy for rational Lagrangian subspaces]
	\label{corintro:dichotomy_for_rational_lagrangian_subspaces}
	With assumptions as in \autoref{thmintro:rational_directions_on_limit_curve}, for every rational Lagrangian subspace $L\subset V_{\bQ}$, precisely one of the following holds:
	\begin{itemize}
		\item either there exists a unipotent transformation $\gamma\in \Gamma$ and vector $v\in L$ such that $v$ is both in the kernel and in the image of $\gamma-\id$,
		\item or there exists a unique $x\in \wtilde{X}$ such that $L_{\bC}\cap F^2(x)\neq \{0\}$.
	\end{itemize}
\end{corollaryintro}
\noindent Note that the property of being both in the kernel, and in the image, of $\gamma-\id$ is equivalent (in $\Sp_4$) to $v$ belonging to the deepest part of the monodromy weight filtration, see \autoref{sssec:weight_filtration}.

% \laternote{Donaldson conjecture on Lagrangian spheres? Ask Ivan Smith about applications}

\subsubsection*{Minimality of the action}
	\label{sssec:minimality_of_the_action}
In the case of Hitchin representations (see \cite{Hitchin1992_Lie-groups-and-Teichmuller-space} were they were introduced) to $\Sp_{4}(\bR)$, there is a limit ``curve'' $\Lambda_P\subset \bP(V_{\bR})$ and a limit surface $\Lambda_L\subset \LGr(V_{\bR})$, but furthermore there is a continuous equivariant map $\Lambda_P\into \Lambda_L$, in other words there is an equivariant choice of full flag.
The action of the group on the limit surface $\Lambda_L$ preserves the limit curve, and on the complement the action is conjugated to that of a Fuchsian lattice in $\SL_{2}(\bR)$ acting on $\SL_{2}(\bR)/A$, where $A$ is the diagonal group.
In particular, the orbits of the action are ``wild'' and there are orbit closures of arbitrary Hausdorff dimension below $2$.

It is therefore of interest that the following holds:
\begin{theoremintro}[Minimality on limit set]
	\label{thmintro:minimality_on_limit_set}
	Suppose $\rho\colon \pi_1(X)\to \Sp_4(\bR)$ is an $\alpha_1$-log-Anosov representation as above.
	Suppose also that there is at least one parabolic boundary point where the monodromy is a rank $1$ unipotent.

	Then the action on the limit set $\Lambda_L\subset \LGr(V_{\bR})$ is minimal.
\end{theoremintro}
\noindent This result is established in \autoref{thm:minimality_of_the_action}.
Note that all hypergeometric examples of \autoref{thmintro:good_hypergeometric_parameters} satisfy the assumptions.
In particular, the equivariant \emph{measurable} section provided by the Oseledets theorem has \emph{dense} image in $\Lambda_L$ (see \autoref{sssec:oseledets_lagrangians_dense}).

\subsubsection*{Schwarz reflection in higher rank}
Hypergeometric equations with real parameters have an additional complex-conjugation symmetry.
This allows one to embed with index $2$ the monodromy group in a group generated by three involutions, corresponding to the three segments on the real projective line formed by eliminating $0,1,\infty$.
In fact, explicit formulas \`a la Levelt are possible, see \autoref{sssec:matrices_of_anti_involutions}.

This information can be used to ``Schwarz-reflect'' the variation of Hodge structure, in analogy with the case of triangle reflection groups.
First, recall that in the case of triangle reflection groups, the image of the period map on the real axis is necessarily a subsegment of a hyperbolic geodesic, and the geodesic itself is determined by the ``monodromy'' reflection which fixes it.

Two notable features arise in this higher weight situation.
First, the image of the period map is not uniquely determined by the monodromy reflection.
However, the Griffiths transversality conditions pleasantly decouple and the image is essentially determined by a spacelike curve in the pseudo-Riemannian space $\bS^{1,1}$.
This puts rather stringent conditions on its global behavior, and in particular makes it the graph of a globally Lipschitz function in appropriate coordinates.

The second feature is the construction of a tiling of the domain of discontinuity $\bS^{1,3}(W_\bR)$ from \autoref{corintro:domains_of_discontinuity}, using polyhedra obtained from the variation of Hodge structure.
We refer to \autoref{ssec:polyhedra_and_reflections} for more details on this construction and remark that this tiling can be used as an alternative route to obtaining the log-Anosov property and its consequences.
Our assumption A can be understood geometrically as a tiling condition, analogous to the one arising in the \Poincare polyhedron theorem.
It seems likely that this approach can be generalized to higher rank hypergeometric variations.

%%					end of subsec: Main results
%%=============================================================================

%%=============================================================================
%%					start of subsec: Further techniques and results

\subsection*{Further techniques and results}
	\label{ssec:further_techniques_and_results}

\subsubsection*{GIT and proper discontinuity}
	\label{sssec:git_and_proper_discontinuity}
There is a large literature on criteria for proper discontinuity of an action of a closed subgroup $\Gamma$ in a semisimple Lie group $G$ on a $G$-space $M$.
For example Benoist \cite{Benoist_Actions-propres-sur-les-espaces-homogenes-reductifs} considers $M=G/H$ for noncompact subgroups $H$, and related criteria for domains in flag manifolds appear in Guichard--Wienhard \cite{GuichardWienhard2012_Anosov-representations:-domains-of-discontinuity-and-applications}, Kapovich--Leeb--Porti \cite{KapovichLeebPorti2018_Dynamics-on-flag-manifolds:-domains-of-proper-discontinuity-and-cocompactness}, and references therein.
To make the text self-contained, and to provide some tools for more general situations, we consider a general $G$-representation $V$ and show that $\Gamma$ acts properly discontinuously on the set of stable points in $\bP(V)$, see \autoref{thm:proper_discontinuity_criterion} for a more general statement.
That there is a relation to Mumford's constructions in Geometric Invariant Theory \cite{MumfordFogartyKirwan_Geometric-invariant-theory} was first suggested by Kapovich--Leeb--Porti \cite{KapovichLeebPorti2018_Dynamics-on-flag-manifolds:-domains-of-proper-discontinuity-and-cocompactness}, who used this point of view to construct domains of discontinuity in flag manifolds.

We take this point of view a bit further and define stable points in an arbitrary $G$-representation $V$ so that the assertion of the Hilbert--Mumford numerical criterion holds (see \autoref{def:stable_point}).
Semistable points are defined analogously.
The class of $1$-parameter semigroups that we use for the ``numerical criterion'' are given by the limit set of the group $\Gamma$.
We refer to \autoref{ssec:stability_and_the_numerical_criterion} for the relevant notions.

\subsubsection*{Index estimates, variational methods}
	\label{sssec:index_estimates_variational_methods}
The proofs of the main analytic properties of the VHS, contained in \autoref{sec:hodge_theory_and_growth}, are based on some inequalities which provide control on the growth of vectors under parallel transport.
Some of the arguments have a variational flavor, and indeed it is quite convenient to invoke at certain steps the Palais--Smale condition and the Ekeland variational principle, even though our situation is essentially $2$-dimensional.

\subsubsection*{Dominated cocycles and posets}
	\label{sssec:dominated_cocycles_and_posets}
While the notion of Anosov, or log-Anosov, representation, is intrinsic to the group $G$, it is almost always necessary to look at particular $G$-representations.
To each $G$-representation $\phi\colon G\to \GL(V)$ there is an associated local system $\bV$, and if there is a VHS on one such local system, it can be propagated to other $G$-representations.

The dynamical properties of the cocycle induced by $\bV$ under parallel transport along the geodesic flow are encoded by the notion of \emph{dominated bundle, or cocycle}.
In the context of Anosov representations, this was observed by Bochi, Potrie, and Sambarino \cite{BochiPotrieSambarino2019_Anosov-representations-and-dominated-splittings}, where most of the analysis was concerned with cocycles with only two pieces, one strictly dominating the other.
The natural structure, in the context of Anosov representations, is a splitting indexed by a \emph{poset}, see \autoref{def:dominated_cocycle}.
This allows us to do a finer analysis of the dynamical behavior.
We illustrate the flexibility of this notion with some examples in \autoref{ssec:dominated_cocycles_and_representations}.
The necessity to consider poset-indexed bundles in domination arises when considering representations other than the ``standard ones''; the adjoint representation of $\Sp_4(\bR)$ for instance exhibits a nontrivial poset structure, see \autoref{fig:adjoint_rep_poset} and \autoref{thm:domination_in_log_anosov_representations} for the general version.
Let us note that poset structures on Lyapunov bundles have also recently appeared in the work of Ledrappier and Lessa \cite{LessaLedrappier2021_Exact-dimension-of-Furstenberg-measures}.

In most situations in the literature, the representation of $G$ is used as an auxiliary tool and can be picked at will, c.f. for instance the ``Tits'', or ``\Plucker'' representation used in \cite[Prop.~8.3]{BochiPotrieSambarino2019_Anosov-representations-and-dominated-splittings}.
In the present paper, the representations to consider are dictated by Hodge theory and so we need to extract further dynamical information, encoded by the poset of dominated bundles.

\subsubsection*{Diffeomorphism types of real domains of discontinuity}
	% \label{sssec:diffeomorphism_types}
The real $3$-manifold obtained as a quotient of the domain of discontinuity in the Lagrangian Grassmannian, constructed in \autoref{corintro:domains_of_discontinuity}, has a natural structure of (orbifold) circle bundle over a Riemann surface, by \autoref{thmintro:uniformization_of_real_lagrangians}.
This makes it a Seifert fibered $3$-manifold, for which we refer to the lecture notes \cite{JankinsNeumann1983_Lectures-on-Seifert-manifolds} for more information.
Furthermore, \autoref{thmintro:uniformization_of_pseudo_sphere} provides a uniformization of the corresponding $\bC^{\times}$-bundle by the pseudo-Riemannian sphere $\bS^{1,3}$.
This ties in with Neumann's appendix to \cite{JankinsNeumann1983_Lectures-on-Seifert-manifolds}, which identifies holomorphic $\bC^\times$-bundles (with negative Euler class) with quasi-homogeneous surface singularities and endows the associated circle bundle with a geometric structure of Thurston type (see also \cite{KulkarniRaymond1985_3-dimensional-Lorentz-space-forms-and-Seifert-fiber-spaces}).
It would be interesting to understand if there is a relation between the uniformization constructed here using Hodge theory, and those constructed by Neumann.

\subsubsection*{Thin hypergeometric groups}
	% \label{sssec:thin_groups}
No lattice in $\Sp_4(\bR)$ can have a domain of discontinuity as in \autoref{corintro:domains_of_discontinuity}.
It follows that the monodromy group $\Gamma$, when it has an integral structure, is necessarily a ``thin group'' in the sense of Sarnak \cite{Sarnak2014_Notes-on-thin-matrix-groups}.
Let us note that the proofs in the present paper offer an alternate route to the results from \cite{FilipFougeron2021_A-cyclotomic-family-of-thin-hypergeometric-monodromy-groups-in-Sp4R} and \cite{BravThomas2014_Thin-monodromy-in-Sp4}, where thinness is established using ping-pong and an explicit construction of cones.
These explicit cones have, nonetheless, other applications to a more detailed understanding of the monodromy groups, and which will be explored in a separate paper \cite{FilipCones}.

Note also that if the parameters $\mu,\nu$ appearing in \autoref{table:good_hypgeom_real}, \ref{table:good_hypgeom_integral}, and \ref{table:maximal_hypgeom} are taken to be \emph{rational}, then the entries of the monodromy groups will be algebraic integers.
We can therefore apply Galois conjugation to the entries and in fact the Galois group of $\bQ$ will act via its cyclotomic quotient.
The new group will be the monodromy of another hypergeometric equation, where the parameters are obtained by multiplying the original ones by a specific integer and reducing modulo $1$.
The two monodromy groups will be abstractly isomorphic, via the isomorphism given by Galois conjugation, and since the first one was thin, so will the second one.
Note that thinness of the first group follows also because it is commensurable to a free or a surface group, a property preserved under abstract isomorphism.

To put into a broader perspective the hypergeometric examples analyzed in \autoref{sec:hypergeometric_local_systems}, let us note that Beukers--Heckman \cite{Beukers_Heckman} were concerned with the classification of finite monodromy groups, i.e. whose image is compact and integral.
This situation can be regarded as the ``rank $0$'' case, where rank refers to the real rank of a real semisimple Lie group.
Fuchs, Meiri, and Sarnak \cite{FuchsMeiriSarnak2014_Hyperbolic-monodromy-groups-for-the-hypergeometric-equation-and-Cartan-involutions} classified the monodromy groups acting on hyperbolic space, i.e. a ``rank $1$'' situation.
The results of this paper can then be placed in a ``rank $2$'' situation, which as we show leads to interesting geometric structures and brings in the period map and the variation of Hodge structure (which do not seem to play a direct role in the previously mentioned works).

Because of the rich geometric structures that result from assumption A (and its variant ``assumption B'' described below) it follows that any VHS satisfying one of these assumptions is necessarily thin.
It would be interesting to understand if there are conditions under which the converse holds.

\begin{figure*}[htbp!]
	\centering
	\includegraphics[width=0.49\linewidth]{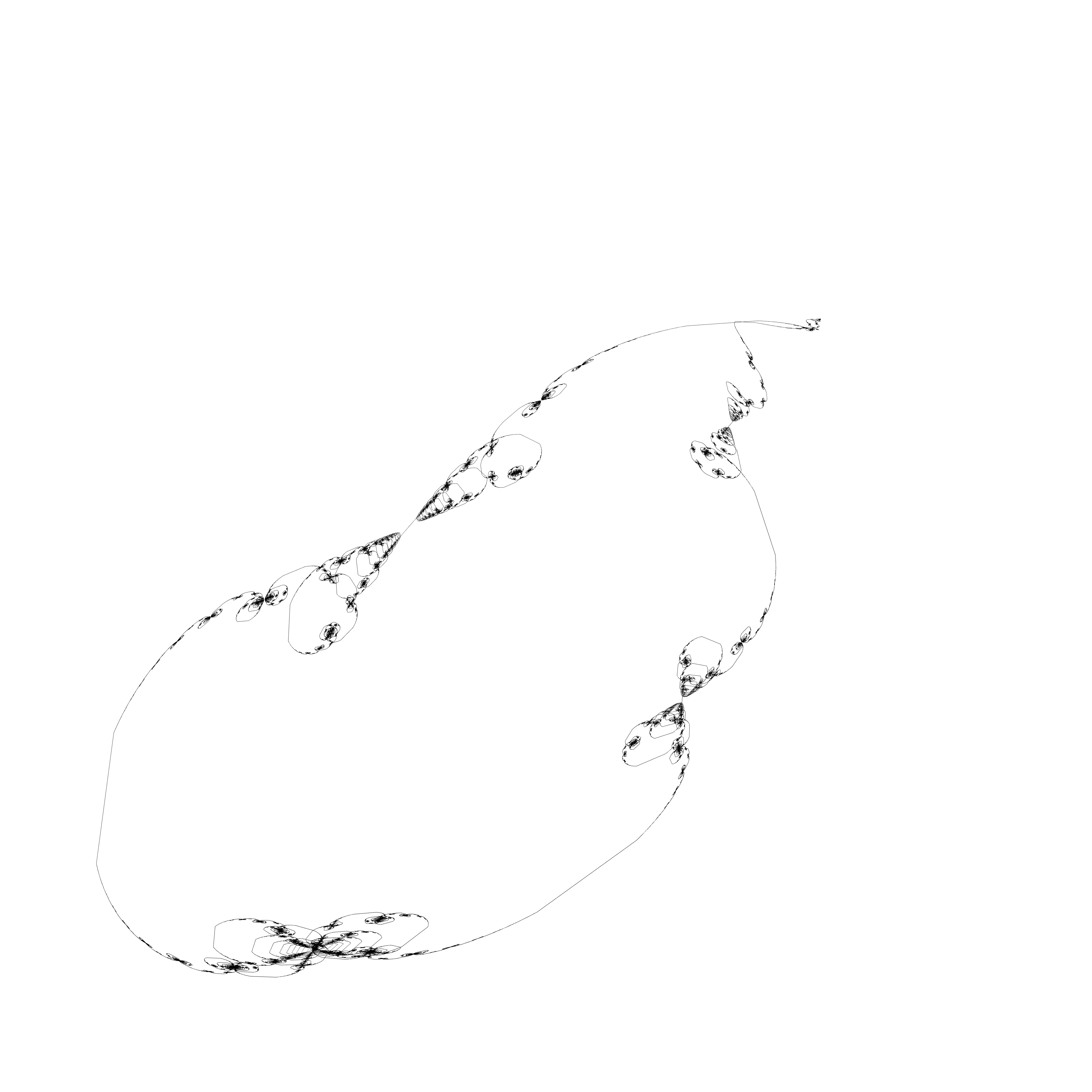}
	\includegraphics[width=0.49\linewidth]{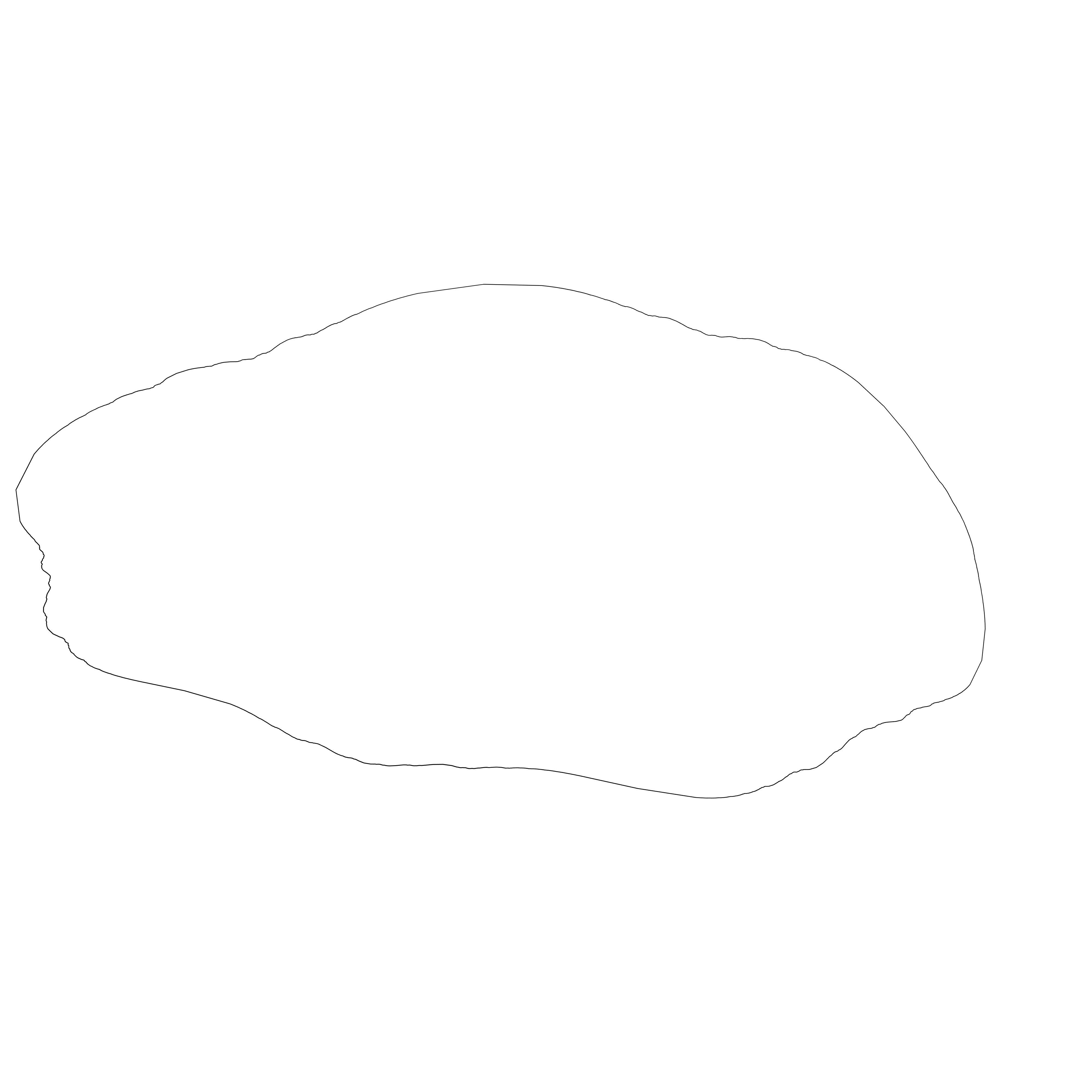}
	\captionsetup{labelformat=empty}
	\caption{\textsc{Figure:} Equivariant boundary curves in the Dwork family, on the left for the mirror quintic (rank $4$ monodromy in $\Sp_4(\bR)$ satisfying assumption A), on the right for the next case (rank $5$ monodromy in $\SO_{2,3}(\bR)$ satisfying assumption B, i.e. maximal).}
	\label{fig:Dwork}
\end{figure*}

\subsubsection*{Maximal representations}
	% \label{sssec:maximal_representations}
In \autoref{ssec:maximal_hypergeometric_monodromy} we consider a situation analogous to ``assumption A'', but which requires a different second fundamental form to be an isomorphism.
The condition on the monodromy then translates to a well-studied notion -- maximal representations of surface groups \cite{BurgerIozziLabourie2005_Maximal-representations-of-surface-groups:-symplectic-Anosov-structures}.
Among the symplectic rank $4$ hypergeometric equations it is not possible to have a maximal representation, because of the rank $1$ unipotent at the singular point $1\in \bP^1(\bC)$.
On the other hand, it is possible to have rank $5$ hypergeometric maximal representations into $\SO_{2,3}(\bR)$.
They are classified in \autoref{table:maximal_hypgeom}.
% Additionally, maximal representations in rank $5$ lead to a formula for Lyapunov exponents which is different from that coming from assumption A, see \autoref{sssec:lyapunov_exponent_maximal_case} for the statement.
Note that a VHS which is both maximal and satisfies assumption A must necessarily be a tensor construction on the uniformizing (Fuchsian) VHS of the base Riemann surface.
A detailed investigation of the relation between maximal representations, Hodge theory, as well as a classification in higher rank of maximal hypergeometric monodromy will appear in a separate text.

Many of the properties of maximal representations are fundamentally different from the Anosov representations that satisfy assumption A.
Perhaps the most striking is the difference in their limit sets, illustrated in the figure with a numerical simulation.
The limit set of a maximal representation is always Lipschitz, whereas in the case of hypergeometric equations satisfying assumption A it is not rectifiable.
Computing its Hausdorff dimension would be an interesting challenge.

Let us finally note that the case of maximal representations also leads to a formula for Lyapunov exponents, this time expressing the top one:
\[
	\lambda_1(\bV)=\frac{\deg \cV^{0,3}_{ext}}{\chi(X)}
\]
See \autoref{sssec:lyapunov_exponent_maximal_case} for a discussion.

\begin{table}
	\begin{tabular}{cc}
	\toprule
	$\alpha_{\bullet}$                                     & $\beta_{\bullet}$                           \\ \hline
	% making table rows wider
	\renewcommand{\arraystretch}{1.5}
	$\left(\mu,\tfrac 12,\tfrac 12, \tfrac 12, 1-\mu \right)$ 
	& 
	$\left(0,0,0,\frac{M}{2M+1},\frac{M+1}{2M+1}\right)$ \\
	\text{ or } & \text{ or }\\
	 $\left(\frac{N-k_N}{2N},\frac{N-1}{2N}, \frac{1}{2},
	 \frac{N+1}{2N},\frac{N+k_N}{2N}\right)$ 
	 &
	 $\left(0, \frac{k_M}{M}, \frac{k_M+1}{M},\frac{M-(k_M+1)}{M},\frac{M-k_M}{M}\right)$
	 \\
	 \bottomrule
	\end{tabular}
	\caption{\label{table:maximal_hypgeom}Any set choice from the first column is compatible with any choice from the second, subject to the condition $\alpha_{min}>\beta_{med}$, where $\alpha_{min}:=\mu$ or $\tfrac{N-k_N}{2N}$, and $\beta_{med}:=\frac{M}{2M+1}$ or $\frac{k_M+1}{M}$ depending on the choices.
	The parameter $\mu$ is real, while $M,N,k_M,k_N$ are positive integers with $1<k_N<N$ and $2(k_M+1)<M$.}
\end{table}

\subsubsection*{MUM Lagrangian is good, explicitly}
	\label{sssec:mum_lagrangian_is_good_explicitly}
The proof of the main conjectures from \cite{EskinKontsevichMoller2018_Lower-bounds-for-Lyapunov-exponents-of-flat-bundles-on-curves} is accomplished in \autoref{ssec:gradient_and_index_estimates}.
In \autoref{rmk:formula_for_sum_of_lyapunov_exponents}, we explain why \autoref{thm:index_estimates} already suffices to obtain both the formula for the sum of Lyapunov exponents, and the assertion that the ``MUM Lagrangian'' (Conjecture~6.4 of loc.cit.) is ``good''.
% In \autoref{ssec:explicit_nonvanishing_of_the_matrix_coefficient} we work out the explicit power series whose nonvanishing is equivalent to the MUM Lagrangian being ``good''.
In particular, this establishes \cite[Conjecture~6.4]{EskinKontsevichMoller2018_Lower-bounds-for-Lyapunov-exponents-of-flat-bundles-on-curves} on the explicit power series that vanishes nowhere on the unit disk.

\subsubsection*{Formula for Lyapunov exponents, and extra punctures}
	\label{sssec:formula_for_lyapunov_exponents_and_punctures}
The formula for the sum of Lyapunov exponents from \autoref{thmintro:formula_for_the_sum_of_lyapunov_exponents} holds in slightly greater generality.
Namely, suppose that $Y=X\setminus S$, where $S$ is a finite set of points, and $X$ is as before, i.e. finite volume hyperbolic orbifold, with $\bV\to X$ the local system underlying a VHS that satisfies assumption A.
Then the formula for the sum of Lyapunov exponents holds on $Y$ as well, even though the monodromy representation of the fundamental group of $Y$ is not log-Anosov.
This holds because the ``bad locus'' is empty on $Y$, if it is empty on $X$.

Let us note that this observation can be used to obtain information on classical dynamical systems, whose statement does not involve any Hodge theory, but whose proof needs it.
Namely, with the setup as above, let $\bF\to X$ be the local system giving the Fuchsian uniformization of $X$, i.e. the monodromy, valued in $\SL_2(\bR)$, is the discrete subgroup which yields $X$ after quotienting $\bH^2$.
Then the positive Lyapunov exponent is $1$ (with appropriate normalizations of the geodesic flow), and the system is uniformly hyperbolic.
For the geodesic flow on $Y$, the cocycle induced by $\bF\vert_Y$ is no longer uniformly hyperbolic.
However, its positive Lyapunov exponent is $\frac{\chi(Y)}{\chi(X)}$, where $\chi$ is the Euler characteristic.

\subsubsection*{Sturm's theorem on zeros and disconjugacy}
	\label{sssec:sturm_s_theorem_on_zeros_and_disconjugacy}
It is a classical result of Sturm's in the theory of ODEs that given the equation $y''+q\cdot y=0$, between any two zeros of one solution, any other solution must vanish at least once.
Equations whose solutions have at most one zero are called \emph{disconjugate} (compare with \autoref{thm:index_estimates}).
A geometric way to interpret Sturm's theorem is to view the differential operator as a map from the base space to $\bP^1$: the $\bP^1$ is the projectivization of the $2$-dimensional space of solutions and a point maps to the line of solutions vanishing at that point.
The discussion naturally complexifies, leading to complex projective structures on Riemann surfaces and univalence criteria, of which some of the most general are Epstein's \cite{Epstein1987_Univalence-criteria-and-surfaces-in-hyperbolic-space}.
Our uniformization results, and associated developing maps, can be viewed as saying that the corresponding Picard--Fuchs equation is ``disconjugate on the universal cover'', a global restriction on its behavior.
Note that this is a fragile property, as illustrated by the conditions that arise when classifying hypergeometric equations having this property.

\subsubsection*{Some speculations}
	\label{sssec:some_speculations}
Our constructions produced domains of discontinuity in real and complex flag manifolds.
After reviewing the analogous situation for compact Riemann surfaces with a real structure in \autoref{sssec:real_algebraic_curves_recollections}, we do a similar analysis for domains of discontinuity in $\LGr(V_\bC)$ and $\LGr(V_\bR)$.
Some further questions about the relationship between these domains of discontinuity are formulated in \autoref{question:properties_of_the_quotient_Frankenstein}.
Some related constructions, for domains of discontinuity in complex flag manifolds, have been considered by Dumas--Sanders \cite{DumasSanders_Uniformization-of-compact-complex-manifolds-by-Anosov-representations}.

\subsubsection*{Further remarks on Lyapunov exponents}
	\label{sssec:further_remarks_on_lyapunov_exponents}
In the case when the monodromy is a lattice, or more generally has full limit set in the corresponding flag manifold, the inequality between Lyapunov exponents and degrees of Hodge bundles proved in \cite{EskinKontsevichMoller2018_Lower-bounds-for-Lyapunov-exponents-of-flat-bundles-on-curves} can be strict.
See for instance the work of Daniel and Deroin \cite{DanielDeroin2019_Lyapunov-exponents-of-the-Brownian-motion-on-a-Kahler-manifold} where results along these lines are proved, albeit under the assumption that the base of the VHS is compact.

The numerical experiments of Kontsevich that led to the discovery of the formula in \autoref{thmintro:formula_for_the_sum_of_lyapunov_exponents} were further refined by Fougeron \cite{Fougeron2019_Parabolic-Degrees-and-Lyapunov-Exponents-for-Hypergeometric-Local-Systems}.
This lead to the numerical discovery of another infinite family of examples, which was proved to be thin in \cite{FilipFougeron2021_A-cyclotomic-family-of-thin-hypergeometric-monodromy-groups-in-Sp4R}.
Note that this family, as well as the original 7 thin cases, are contained in \autoref{table:good_hypgeom_integral}.

For some further results on how the Lyapunov exponents depend on the representation, from the point of view of Higgs bundles, see the work of Costantini \cite{Costantini2020_Lyapunov-exponents-holomorphic-flat-bundles-and-de-Rham-moduli}.

\subsubsection*{Outline of the text}
	\label{sssec:outline_of_the_text}
\autoref{sec:hodge_theory_and_growth} establishes the main analytic results regarding the class of variations of Hodge structure that we consider.
The conjecture from \cite{EskinKontsevichMoller2018_Lower-bounds-for-Lyapunov-exponents-of-flat-bundles-on-curves} on the sum of Lyapunov exponents is established there, and the reader interested only in this formula can stop reading at \autoref{rmk:formula_for_sum_of_lyapunov_exponents}, halfway through the section.
After the basic analytic properties are established, we obtain some further quantitative estimates and construct the developing maps which play a role in subsequent constructions.

\autoref{sec:hypergeometric_local_systems} classifies the hypergeometric differential equations that satisfy our assumption.
Additionally, we develop the Schwarz reflection structure of the monodromy and VHS, and explicitly analyze the period map on the real line.

\autoref{sec:elements_of_lie_theory_and_anosov_representations} contains the main Lie-theoretic results.
It is independent of the previous sections.
After recalling some of the classical notions, we introduce poset-dominated cocycles and log-Anosov representations.
The analogy with GIT is developed in \autoref{ssec:stability_and_the_numerical_criterion}, which in particular gives a criterion for proper discontinuity.
We then establish \autoref{thmintro:rational_directions_on_limit_curve} which controls the rational points on the limit curve, followed by \autoref{thmintro:minimality_on_limit_set} which yields minimality of the action on the limit set.

\autoref{sec:hodge_theory_and_anosov_representations} brings together the results from the previous sections on Anosov representations and Hodge theory.
In particular we show that the domains of discontinuity are indeed uniformized by the Hodge structures, and produce some further interesting domains.
We also classify in this section the \emph{maximal} hypergeometric local systems of rank $5$, producing another class of examples where a formula for the sum of Lyapunov exponents holds.

Finally, \autoref{sec:unipotent_dynamics} shows that using adapted metrics in the cusp, the behavior of dominated cocycles is as good as in the compact part.
We recall some further notions, including the monodromy weight filtration, that illustrate the usefulness of Hodge-theoretic techniques in the context of Anosov representations.

%%					end of subsec: Further techniques and results
%%=============================================================================

\subsubsection*{Notes and references}
For more on Sturm's theorem, see \cite{Arnolcprime-d1985_Sturm-theorems-and-symplectic-geometry}, where a symplectic extension for \emph{real} and \emph{Hamiltonian} ODEs is proved.
A substantial difference with our case is that, besides complexification, our complex developing map to the Lagrangian gives a complex null-curve, whereas Arnold considers curves which are timelike.

\autoref{thmintro:rational_directions_on_limit_curve} is inspired by an analogous one, in the setting of triangle reflection groups, due to McMullen \cite[Thm.~1.1]{McMullen2021_Billiards-heights-and-the-arithmetic-of-nonarithmetic-groups}.
An introduction to some of the applications of Lyapunov exponents and Hodge theory to \Teichmuller dynamics is in the lecture notes of Forni and Matheus \cite{ForniMatheus2014_Introduction-to-Teichmuller-theory-and-its-applications-to-dynamics-of-interval-exchange}.

Collier, Tholozan, and Toulisse \cite{CollierTholozanToulisse2019_The-geometry-of-maximal-representations-of-surface-groups-into} consider in the broader context of Higgs bundles the case of maximal representations and their associated uniformization (what could be viewed as a counterpart ``assumption B'' to our ``assumption A'').
It would be interesting to find, besides the hypergeometric examples in \autoref{ssec:maximal_hypergeometric_monodromy}, monodromy representations that are maximal and have an integral structure.

The original observation that for weight $1$ variations of Hodge structure, one can obtain a formula for the sum of Lyapunov exponents, goes back to Kontsevich \cite{Kontsevich1997_Lyapunov-exponents-and-Hodge-theory} and was substantially developed in the context of \Teichmuller dynamics by Forni \cite{Forni2002_Deviation-of-ergodic-averages-for-area-preserving-flows-on-surfaces-of-higher}.
An extension to families of K3 surfaces can be found in \cite{Filip2018_Families-of-K3-surfaces-and-Lyapunov-exponents}.
% Numerical experiments that further explored the formula for Lyapunov exponents in the hypergeometric setting were performed by Fougeron \cite{Fougeron2019_Parabolic-Degrees-and-Lyapunov-Exponents-for-Hypergeometric-Local-Systems}.
% These experiments indicated the family analyzed via ping-pong in \cite{FilipFougeron2021_A-cyclotomic-family-of-thin-hypergeometric-monodromy-groups-in-Sp4R}.

More background on Anosov representations is available in Labourie's original paper \cite{LabourieAnosov}, as well as Guichard--Wienhard \cite{GuichardWienhard2012_Anosov-representations:-domains-of-discontinuity-and-applications}, or the survey of Kapovich--Leeb--Porti \cite{KapovichLeebPorti2016_Some-recent-results-on-Anosov-representations}, and references therein.

A Torelli theorem for generic quintic threefolds was established by Voisin \cite{Voisin1999_A-generic-Torelli-theorem-for-the-quintic-threefold}.
However, note that Torelli theorems \emph{can} fail for $1$-parameter families, for example in examples constructed by Szendr\"{o}i \cite{Szendroi2000_Calabi-Yau-threefolds-with-a-curve-of-singularities-and-counterexamples-to-the-Torelli}.

\subsubsection*{Conventions}
	\label{sssec:notational_conventions}
The topological notions \emph{closed, open, etc.} are always for the analytic topology, unless the modifier `Zariski-' is applied.
The constants $C,\ve,$ etc. that appear in statements are distinct from one result to the next and we do not make the dependence explicit, as only the existence of such constants is sufficient.
The notation $X=O(Y)$ means that there exists $c>0$ such that $|X|\leq c Y$.
Additionally, the inequality $X\leqapprox Y$ means there exists $c>0$ such that $X\leq cY$.
Unless further specified, the constants in such inequalities are allowed to depend on fixed background geometric data, including the Riemann surface and variation of Hodge structure.

\subsubsection*{Acknowledgments}
Most of the work on this paper was completed during the author's stay at the Institute for Advanced Study in Princeton during the academic year 2018-2019.
I am deeply grateful for the excellent working conditions.

I am grateful to Fanny Kassel for discussions around the criterion of proper discontinuity in \autoref{thm:proper_discontinuity_criterion}, to Carlos Matheus for bringing my attention to the Bochi--Potrie--Sambarino paper \cite{BochiPotrieSambarino2019_Anosov-representations-and-dominated-splittings}, to Brian Collier for questions around the constructions in \autoref{ssec:developing_maps_to_real_homogeneous_spaces}, and to Giovanni Forni for questions and suggestions that improved the presentation.

This material is based upon work supported by the US National Science Foundation under Grants No. DMS-2005470 (SF), No. DMS-1638352 (IAS), as well as
 DMS-1107452, 1107263, 1107367 ``RNMS: Geometric Structures and Representation Varieties'' (the GEAR Network).

This research was partially conducted during the period the author served as a Clay Research Fellow.

%%%%%%%%%%%%%%%%%%%%%%%%%%%%%%%%%%%%%%%%%%%%%%%%%%%%%%%%%%%%%%%%%%%%%%%%%%%%%%%
%%% 				End of Section: Introduction
%%%%%%%%%%%%%%%%%%%%%%%%%%%%%%%%%%%%%%%%%%%%%%%%%%%%%%%%%%%%%%%%%%%%%%%%%%%%%%%

%%%%%%%%%%%%%%%%%%%%%%%%%%%%%%%%%%%%%%%%%%%%%%%%%%%%%%%%%%%%%%%%%%%%%%%%%%%%%%%
%%% 				Start of Section: Hodge theory and growth
%%%%%%%%%%%%%%%%%%%%%%%%%%%%%%%%%%%%%%%%%%%%%%%%%%%%%%%%%%%%%%%%%%%%%%%%%%%%%%%

\section{Hodge theory and growth}
	\label{sec:hodge_theory_and_growth}

\paragraph{Outline}
In this section we use tools from Hodge theory establish the necessary growth estimates of flat sections.
These estimates imply the conjectured formula for the sum of Lyapunov exponents from \cite{EskinKontsevichMoller2018_Lower-bounds-for-Lyapunov-exponents-of-flat-bundles-on-curves}, and will be used in \autoref{sec:hodge_theory_and_anosov_representations} to establish the Anosov property for the particular class of variations of Hodge structure under consideration.

In \autoref{ssec:hodge_theory_background} we set up preliminaries from Hodge theory, introduce the class of examples we study and analyze their degenerations.
The basic concept is \autoref{def:assumption_a} that we call ``assumption A''.
Next, in \autoref{ssec:gradient_and_index_estimates} we establish the key analytic estimates that yield all subsequent results.
The basic technique is based on showing that specific functions associated to flat sections can only have local minima, and then using more quantitative information to ensure the correct growth at infinity.
Techniques from the calculus of variations provide a convenient framework.

Next, in \autoref{ssec:symplectic_orthogonal_dictionary} we introduce the necessary constructions from linear algebra that will be used to produce domains of discontinuity.
We also discuss the correspondence between symplectic geometry in dimension $4$ and indefinite-orthogonal geometry in dimension $5$, arising from the isogeny of Lie groups $\Sp_4(\bR)\to \SO_{2,3}(\bR)$.
At different steps in the argument, it is more convenient to use one of the two geometries.

Finally, in \autoref{ssec:developing_maps_to_real_homogeneous_spaces} we construct the uniformization maps that arise from the variations of Hodge structure satisfying assumption A.

% Photons terminology is standard in this business, refer to Collier et al and maybe some survey (Goldman?).
% ``Electrons'' is a concept introduced here and not standard terminology.
% \cite{CollierTholozanToulisse2019_The-geometry-of-maximal-representations-of-surface-groups-into}
% Note that in loc. cit. the uniformization is obtained using a continuity principle, starting from a Fuchsian representation where the picture is relatively clear.
% In our case, the representation is rigid so no such deformation argument is possible.

%%=============================================================================
%%					start of subsec: Hodge Theory background

\subsection{Hodge Theory background}
	\label{ssec:hodge_theory_background}

An introduction to variations of Hodge structures and period domains can be found in the survey of Griffiths and Schmid \cite{GriffithsSchmid_Recent-developments-in-Hodge-theory:-a-discussion-of-techniques}
as well as the more recent monograph \cite{CarlsonMuller-Stach_Period-mappings-and-period-domains}.
Everything that we need is concisely covered in the lectures of Griffiths \cite[Ch.~I-IV]{Griffiths_Topics-in-transcendental-algebraic-geometry}.

\subsubsection{Setup}
	\label{sssec:setup_hodge_theory_background}
Let $\ov{X}$ be a compact Riemann surface, $D\subset X$ a possibly empty finite set of points, and let $X:=\ov{X}\setminus D$ be the complement.
Assume that $X$ has negative Euler characteristic and so has a unique complete constant curvature $-1$ metric, which we will refer to as the hyperbolic metric.
Fix also a basepoint $x_0\in X$, which we will frequently identify with its lift to the associated universal cover $\wtilde{X}$.

\subsubsection{Variations of Hodge structure}
	\label{sssec:variations_of_hodge_structure}
Recall that a real, polarized, weight $n$ variation of Hodge structure (VHS) over $X$ consists of:
\begin{itemize}
	\item A local system of real vector spaces $\bV\to X$, as well as a nondegenerate flat bilinear pairing on the fibers, denoted $\ip{-,-}_{I}$, which is $(-1)^n$-symmetric.
	The flat connection on $\bV$ is denoted $\nabla^{GM}$ and called the Gauss--Manin connection.
	\item A decomposition of the complexified local system
	\[
		\bV_{\bC}:= \oplus \cV^{p,q} \text{ into complex subbundles}
	\]
	such that under complex conjugation $\ov{\cV^{p,q}}=\cV^{q,p}$.
	Assume also that the decomposition is orthogonal for the indefinite hermitian pairing $\ip{v,v}_{IH}:=(\sqrt{-1})^{n}\ip{v,\ov{v}}_{I}$ and this hermitian pairing is positive, resp. negative, definite on $\cV^{p,q}$ according to whether $p$ is odd, resp. even.
	Denote by $\ip{-,-}_{H}$ the positive-definite hermitian metric, called the Hodge metric, obtained from $\ip{-,-}_{IH}$ by flipping appropriately the signs on $\cV^{p,q}$.
	\item The Gauss--Manin connection decomposes as
	\begin{align}
		\label{eqn:Gauss_Manin_Higgs_field}
		\nabla^{GM} = \nabla^{Ch} + \sigma + \sigma^{\dag}
	\end{align}
	where $\nabla^{Ch}$ is the Chern connection for a holomorphic structure on $\cV^{p,q}$, and $\sigma = \sum \sigma_{p,q}$ with
	\[
		\sigma_{p,q}\colon \cV^{p,q}\to \cV^{p-1,q+1}\otimes \cK_{X}
	\]
	holomorphic maps, called second fundamental forms.
	Here $\cK_{X}$ denotes the canonical bundle of $X$, and adjoints are for the Hodge metric.
\end{itemize}
The \emph{Hodge filtration} bundles are $F^p\cV:= \oplus_{p'\geq p} \cV^{p',n-p'}$ and yield holomorphic subbundles of $\bV_{\bC}$.
We will work mostly on the universal cover $\wtilde{X}$, on which the local system can be identified, using the Gauss--Manin connection, with $\wtilde{X}\times V_{\bR}$, where $V_{\bR}:=\bV(x_0)$ is the fiber over $x_0$.

Given a flat section of $\bV_{\bC}$ on $\wtilde{X}$, we can identify it with a vector $v$ in the fixed vector space.
Its $(p,q)$-components are defined by $v= \sum v^{p,q}(x)$ with $v^{p,q}(x)\in \cV^{p,q}(x)$, and satisfy, using $\autoref{eqn:Gauss_Manin_Higgs_field}$ and $\nabla^{GM}v=0$:
\begin{align}
	\label{eqn:nabla_Ch_flat_components}
	\nabla^{Ch}v^{p,q} =-\sigma_{p+1,q-1}v^{p+1,q-1} - \sigma^{\dag}_{p,q}v^{p-1,q+1}.
\end{align}
Note in particular that for the extremal components, i.e. $v^{n,0}$ and $v^{0,n}$, one of the two terms vanishes.

\subsubsection{The weight $3$ notation}
	\label{sssec:the_weight_3_notation}
The situation of interest for this paper comes from a weight $3$ variation
\[
	\bV_{\bC}:=\cV^{3,0}\oplus \cV^{2,1} \oplus \cV^{1,2} \oplus \cV^{0,3}
\]
which is equipped with a symplectic form.
We will take $\dim \cV^{p,q}=1$ for all $p,q$.

\subsubsection{The weight $4$ notation}
	\label{sssec:the_weight_4_notation}
The second exterior power of the earlier weight $3$ VHS yields a weight $6$ VHS, which however splits off a $1$-dimensional piece given by the invariant symplectic form -- this yields a $1$-dimensional $(3,3)$-piece which is flat and orthogonal to the rest of the VHS.
Denote by $\bW$ the remaining $5$-dimensional piece, and shift the weights by $(1,1)$ to simplify notation:
\[
	\bW_{\bC} = \cW^{4,0} \oplus \cW^{3,1} \oplus \cW^{2,2} \oplus \cW^{1,3} \oplus \cW^{0,4}
\]
where the signature of the indefinite hermitian metric is $(-,+,-,+,-)$ on the respective summands, each of which is $1$-dimensional.

It is immediate to check that for instance 
$\cW^{4,0}\isom \cV^{3,0}\otimes \cV^{2,1}$
and
$\cW^{3,1}\isom \cV^{3,0}\otimes \cV^{1,2}$
and so $\sigma_{4,0}=\id\otimes \sigma_{2,1}$.

More pictorially, we can describe the maps and bundles as:
\begin{equation}
	\label{eqn:W_hodge_bundles_sigma_maps}
\begin{tikzcd}
	\cW^{4,0}
	\arrow[r, bend right, swap, "\sigma_{4,0}"]
	\arrow[r, bend left, leftarrow, "{\sigma_{4,0}^\dag}"]
	& 
	\cW^{3,1}
	\arrow[r, bend right, swap, "\sigma_{3,1}"]
	\arrow[r, bend left, leftarrow, "{\sigma_{3,1}^\dag}"]
	&
	\cW^{2,2}
	\arrow[r, bend right, swap, "\sigma_{3,1}^t"]
	\arrow[r, bend left, leftarrow, "\ov{\sigma_{3,1}}"]
	&
	\cW^{1,3}
	\arrow[r, bend right, swap, "\sigma_{4,0}^t"]
	\arrow[r, bend left, leftarrow, "\ov{\sigma_{4,0}}"]
	&
	\cW^{0,4}
\end{tikzcd}
\end{equation}
where we use the identification of duals to write the right part of the diagram of second fundamental forms.

\subsubsection{Deligne extension}
	\label{sssec:deligne_extension}
As explained in \autoref{sssec:orbifold_points_elliptic_parts} below, orbifold points are treated by passing to the appropriate finite cover.
Suppose now that $d\in D=\ov{X}\setminus X$ is a cusp, and further that the monodromy of the local system around $d$ is unipotent.
Then there exists a canonical choice of extension $\bV_{ext}$ of the local system in a neighborhood of $d$, called the Deligne extension.
Additionally, the Hodge bundles $\cV^{p,q}$ have extensions $\cV^{p,q}_{ext}$ as well, obtained by first extending the Hodge filtration bundles.

\begin{definition}[Assumption A]
	\label{def:assumption_a}
	A weight $3$, rank $4$, variation of Hodge structure will be said to satisfy \emph{assumption A} if
	\[
		\sigma_{2,1}\colon \cV^{2,1}_{ext}\to \cV^{1,2}_{ext}\otimes \cK_{\ov{X}}(\log D)
	\]
	is an isomorphism, where $\cK_{\ov{X}}(\log D)$ denotes the bundle of meromorphic $1$-forms on $\ov{X}$ with logarithmic (i.e. $\tfrac{dz}{z}$) poles at the points in $D$.

	Analogously, a weight $4$, rank $5$ variation of Hodge structure will be said to satisfy \emph{assumption A} if
	\[
		\sigma_{1,3}\colon \cW^{1,3}_{ext} \to \cW^{0,4}_{ext}\otimes \cK_{\ov{X}}(\log D)
	\]
	is an isomorphism.
\end{definition}
Note that a weight $3$ VHS satisfies assumption A if and only if its associated weight $4$ VHS via the construction in \autoref{sssec:the_weight_4_notation} satisfies it.
Indeed, we have that $\sigma_{1,3}=\sigma_{4,0}^{t}$ and the earlier remarks about the relation between second fundamental forms give the equivalence.

\emph{From now on, we assume that our variations of Hodge structure satisfy assumption A.}

\subsubsection{Notation for norms}
	\label{sssec:notation_for_norms}
Several norms will appear throughout the arguments below.
The flat indefinite hermitian form is denoted $\ip{-,-}_{IH}$, and $\norm{v}^2_{IH}:=\ip{v,v}_{IH}$.
Similarly $\ip{-,-}_H$ and $\norm{-}^2_{H}$ will denote the positive-definite Hodge pairing and norm.
For tangent vectors or covectors, we will use the norm coming from the fixed background hyperbolic metric on $X$, and which will be denoted $\norm{-}_{hyp}$.
Note that the hyperbolic metric can be viewed as a singular hermitian metric on (the dual of) $\cK_{X}(\log D)$.

When it will be clear from the context (i.e. which bundle the vector $v$ belongs to), the notation $\norm{v}$ will denote either the Hodge norm, or the hyperbolic norm.
When speaking of ``absolute'' constants, or uniformly bounded quantities, we will refer to constants that only depend on the original VHS over $X$, and the fixed hyperbolic metric on $X$.
The notation $A\geqapprox B$ means that there exists an absolute constant $c>0$ such that $c\cdot A\geq B$.

\begin{proposition}[Boundedness of second fundamental form]
	\label{prop:boundedness_of_second_fundamental_form}
	If the weight $3$ VHS satisfies assumption A, then there exists a constant $C>0$ such that
	\[
		C\geq	\norm{\sigma_{2,1}} \geq \frac{1}{C}
	\]
	where we view $\sigma$ as a section of $\left(\cV^{2,1}\right)^{\dual}\otimes \cV^{1,2}\otimes \cK_{X}(\log D)$ and each term of the tensor product is equipped with its Hodge, resp. hyperbolic metric.
	The same holds for $\norm{\sigma_{1,3}}$ in the case of a weight $4$ VHS.

	Furthermore, in weight $3$, at any of the punctures, the Hodge numbers of the limiting mixed Hodge structure, and corresponding monodromy, are the second and fourth in \autoref{fig:Hodge_numbers}.
	
	Conversely, if only the two listed situations occur at the punctures, and $\sigma_{2,1}$ does not vanish on $X$, then assumption A is satisfied.
\end{proposition}

\begin{figure}[htbp!]
	\centering
	\begin{tabular}{c|c|c|c}
	1 &   &   & \\ \hline
	  & 1 &   & \\ \hline
	  &   & 1 & \\ \hline
	  &   &   & 1 \\
	\end{tabular}
	\quad
	\begin{tabular}{c|c|c|c}
	1 &   &   & \\ \hline
	  &   & 1 & \\ \hline
	  & 1 &   & \\ \hline
	  &   &   & 1 \\
	\end{tabular}
	\quad
	\begin{tabular}{c|c|c|c}
	  & 1 &   & \\ \hline
	1 &   &   & \\ \hline
	  &   &   & 1 \\ \hline
	  &   & 1 & \\
	\end{tabular}
	\quad
	\begin{tabular}{c|c|c|c}
	  &   &   & 1 \\ \hline
	  &   & 1 & \\ \hline
	  & 1 & 	 & \\ \hline
	1 &   &   & \\
	\end{tabular}
	\caption{The possible Hodge numbers of a degeneration.
	From left to right: first, the original (pre-degeneration) numbers,
	then a rank $1$ unipotent with an invariant line,
	then a rank $1$ unipotent with an invariant Lagrangian,
	then a maximally unipotent monodromy.
	}
	\label{fig:Hodge_numbers}
\end{figure}

\begin{remark}[On singularities of the hyperbolic metric]
	\label{rmk:on_singularities_of_the_hyperbolic_metric}
	Recall that on the punctured unit disc with coordinate $q$, the hyperbolic metric is (up to a constant factor) equal to $\tfrac{|dq|^2}{\left(|q|\log|q|\right)^2}$.
	This is the metric $\frac{|d\tau|^2}{(\Im \tau)^2}$ descended from the upper half-plane by $\tau \mapsto \exp(\twopii \tau)=:q$.
	In particular, the vector field $q\partial_q$ has norm $\frac{1}{|\log |q||}$ and thus the logarithmic differential $\frac{dq}{q}$ has unbounded norm.

	Let us see that in the standard weight $1$ degeneration the bound claimed in \autoref{prop:boundedness_of_second_fundamental_form} holds.
	Specifically, we have that $\bC^2 = v^{1,0}(\tau) \oplus \ov{v^{1,0}(\tau)}$ where $v^{1,0}(\tau)=a+\tau b$.
	Then $\nabla_{\partial_{\tau}}v^{1,0} = b$, and so as sections of line bundles, we have that $\sigma_{1,0}=\frac{b}{v^{1,0}\cdot \partial_{\tau}}$.
	Computing the Hodge and hyperbolic norms yields $\norm{v^{1,0}}^{2}=2\Im \tau$, $\norm{b}^{2}=\frac{2}{\Im\tau}$ and $\norm{\partial_{\tau}}^2=\frac{1}{(\Im \tau)^2}$ from which the bounded (in fact, constant) norm of $\sigma^{1,0}$ follows.
\end{remark}

\begin{proof}[Proof of \autoref{prop:boundedness_of_second_fundamental_form}]
	It is immediate that all questions are local at the punctures, since on a compact set a continuous nonvanishing function is bounded above and below.
	In any of the cases, it suffices to establish the bounds for the nilpotent orbit corresponding to the degeneration, and in fact the $\SL_2$-orbit as well, since by the results of Schmid \cite[Thm.~4.12, Thm.~5.13]{Schmid_VHS}, the metrics in the two cases are comparable by uniform constants, which in fact go to $1$ as $\Im \tau\to +\infty$.

	In the maximally unipotent monodromy case, the behavior of the nilpotent orbit is the same as that of the third symmetric power of the example analyzed in \autoref{rmk:on_singularities_of_the_hyperbolic_metric}, so the conclusions follow.

	To proceed with the next cases, we will freely use standard terminology in the subject, see e.g. \cite[\S2,3]{Robles2017_Degenerations-of-Hodge-structure} and references therein.
	Fix a symplectic basis $e_1,f_1,e_2,f_2$ of $V_\bR$ with $\ip{e_i,f_i}_I=1$ and $0$ otherwise.
	Let $N:=\log T$ be the logarithm of the monodromy, working under the assumption that $T$ is purely unipotent.
	The elliptic part, when present, does not affect the metric estimates (see also \autoref{sssec:orbifold_points_elliptic_parts} below).
	Recall also that the nilpotent orbit on a (subset of) the upper half plane is given by the formula $e^{\tau N}F^{\bullet}_{lim}$ where $F^{\bullet}_{lim}$ is the limit of the Hodge filtration (in the Deligne extension of the local system).
	For asymptotic questions, we can also work under the assumption that the limiting mixed Hodge structure is $\bR$-split, so it suffices to describe the spaces $I^{p,q}$, subject to the polarization and complex-conjugation axioms of a polarized mixed Hodge structure.
	Then the limiting Hodge filtrations are given by
	\[
		F^{p}_{lim} = \bigoplus_{r\geq p, \forall s}I^{r,s}.
	\]

	The second diagram in \autoref{fig:Hodge_numbers} corresponds to $Ne_1=-f_1$ and $0$ otherwise (the reason for $-f_1$ instead of $f_1$ is to have the correct alternating signature below).
	Then the $\bR$-splitting is given by
	\begin{align}
		\label{eqn:nilpotent_orbit_rank1_line}
		\begin{split}
		I^{3,0} & = e_2 + \sqrt{-1}f_2 \\
		I^{0,3} & = \ov{I^{3,0}}
		\end{split}
		\begin{split}
		I^{2,2} & = e_1\\
		I^{1,1} & = f_1
		\end{split}
	\end{align}
	so that for the Hodge filtrations and nilpotent orbit we find:
	\begin{align*}
		\begin{split}
		F^3_{lim} & = span(e_2 + \sqrt{-1}f_2)\\
		F^2_{lim} & = F^3_{lim} + e_1\\
		F^1_{lim} & = F^2_{lim} + f_1
		\end{split}
		\begin{split}
		F^3(\tau) & = F^3_{lim}\\
		F^2(\tau) & = F^3(\tau) + (e_1 - \tau f_1)\\
		F^1(\tau) & = F^2(\tau) + f_1
		\end{split}
	\end{align*}
	The desired estimate on $\sigma^{2,1}$ now follows exactly as in \autoref{rmk:on_singularities_of_the_hyperbolic_metric}.

	The third diagram of Hodge numbers corresponds to the nilpotent operator acting by $Ne_i=f_i$ and $I^{p,q}$-components:
	\begin{align}
		\label{eqn:nilpotent_orbit_rank1_Lagrangian}
		\begin{split}
		I^{3,1} = e_1 + \sqrt{-1}e_2\\
		I^{1,3} = e_1 - \sqrt{-1}e_2\\
		\end{split}
		\begin{split}
		I^{2,0} = f_1 + \sqrt{-1}f_2\\
		I^{0,2} = f_1 - \sqrt{-1}f_2\\
		\end{split}
	\end{align}
	which leads to the limit Hodge filtrations and nilpotent orbit:
	\begin{align*}
		\begin{split}
		F^3_{lim} & = e_1 + \sqrt{-1}e_2\\
		F^2_{lim} & = F^3_{lim} + (f_1+\sqrt{-1}f_2)\\
		F^1_{lim} & = F^2_{lim} + (e_1+\sqrt{-1}e_2)
		\end{split}
		\begin{split}
		F^3(\tau) & = (e_1+\tau f_1)+\sqrt{-1}(e_2+\tau f_2)\\
		F^2(\tau) & = F^2_{lim}\\
		F^1(\tau) & = F^2_{lim} + \tau(f_1+\sqrt{-1}f_2)
		\end{split}
	\end{align*}
	Visibly, the second fundamental form $\sigma_{2,1}$ vanishes.
\end{proof}

\subsubsection{Orbifold points, elliptic parts}
	\label{sssec:orbifold_points_elliptic_parts}
The arguments below, and definitions above, apply also to the situation where $X$ is an orbifold, under the following assumption:
There exists a finite orbifold cover $X'\to X$ such that $X'$ is a Riemann surface and assumption A applies to the VHS pulled back to $X'$.
In other words, there exists a torsion-free finite index subgroup $\pi_1(X')$ of $\pi_1^{orb}(X)$ such that the assumptions and definitions apply to $X'$.

Additionally, it is possible for the monodromy matrix $T$ to decompose into commuting elliptic and unipotent parts $T=T^{u}\cdot T^e$.
Furthermore, when the monodromy is not integral it is possible for $T^e$ to be of infinite order.
Nonetheless, the metric estimates of \autoref{prop:boundedness_of_second_fundamental_form} apply and this also leads to cases that satisfy assumption A in the hypergeometric examples.

%%					end of subsec: Hodge Theory background
%%=============================================================================

%%=============================================================================
%%					start of subsec: Gradient and index estimates

\subsection{Gradient and index estimates}
	\label{ssec:gradient_and_index_estimates}

We proceed with a weight $4$ VHS as in \autoref{sssec:the_weight_4_notation}.
In this section, we establish the basic gradient estimates and show that certain functions can only have critical points which are local minima, and in fact at most one local local minimum can occur.

The conjectured formula for the sum of Lyapunov exponents from \cite{EskinKontsevichMoller2018_Lower-bounds-for-Lyapunov-exponents-of-flat-bundles-on-curves} follows from \autoref{thm:index_estimates}, and also the slightly stronger statement that the MUM Lagrangian is good, see \autoref{rmk:formula_for_sum_of_lyapunov_exponents}.

\subsubsection{Setup}
	\label{sssec:setup_gradient_and_index_estimates}
Let $w\in W_\bR\setminus 0$ be a null, or a positive vector.
In other words $\ip{w,w}_{I}\geq 0$, and if $w$ is a positive vector we normalize it to $\ip{w,w}_I=1$.
Using that $w$ is real, for any Hodge decomposition
\[
	w = w^{4,0} \oplus w^{3,1}\oplus w^{2,2} \oplus w^{1,3} \oplus w^{0,4}
\]
we thus have
\[
	2\norm{w^{1,3}}^{2} - \left(2\norm{{w}^{0,4}}^2 + \norm{w^{2,2}}^2\right) = \ip{w,w}_I \geq 0
\]
where the last inequality holds because the vector is nonzero.
In particular, note that
\begin{align}
	\label{eqn:basic_non_degeneracy}
	\begin{split}
	w^{1,3}\neq 0 & \text{ and } \norm{w^{1,3}} \geq \norm{w^{0,4}}\\
	\norm{w^{1,3}} \geq 1 & \text{ if }\ip{w,w}_{I}=1
	\end{split}
\end{align}
Introduce now on $\wtilde{X}$ the function
\[
	f_w(x) := \norm{w^{0,4}(x)}^2
\]
relative to the Hodge decomposition at $x\in \wtilde{X}$.
When $w$ is fixed and clear from the context, we will omit it as a subscript.

\subsubsection{The directing vector field}
	\label{sssec:the_directing_vector_field}
Recall also that we have the second fundamental form $\sigma_{1,3}\colon \cW^{1,3}\to \cW^{0,4}\otimes \cK_X$.
Since $\sigma_{1,3}$ is nowhere vanishing by \autoref{def:assumption_a} and $w^{1,3}$ doesn't vanish on $\wtilde{X}$ either, there exists a unique vector field $V_w$ such that
\[
	\sigma_{1,3}(V_w)w^{1,3} = -w^{0,4} \text{ or in local coordinates }
	V_w = \frac{-w^{0,4}}{\sigma_{1,3}\left(w^{1,3}\right)}.
\]
Crucially, note that for the definition of $V_w$, both the nonvanishing of $\sigma_{1,3}$ and $w^{1,3}$ were used.

\begin{proposition}[Gradient estimates]
	\label{prop:gradient_estimates}
	Let $V=V_w$ and $f=f_w$ be as above.
	Equip $\wtilde{X}$ with the hyperbolic metric.
	\begin{enumerate}
		\item For the Lie derivative $\cL_V$ we have that
		\[
			\cL_V f = 2f \text{ and }V, \grad f \text{ are proportional.}
		\]
		Furthermore we have
		\begin{align}
			\label{eqn:gradf_V_2f_identity}
			\norm{\grad f}_{hyp}\cdot \norm{V}_{hyp} = 2f.
		\end{align}
		\item We have the following ``steepness'' estimate
		\begin{align}
			\label{eqn:steepness_basic}
			\norm{V}_{hyp}\leqapprox 1 \text{ and thus } \norm{\grad f}_{hyp}\geqapprox f.
		\end{align}
		
		\item If furthermore $\ip{w,w}_I=1$ then we also have the inequality
		\begin{align}
			\label{eqn:steepness_low_f}
			\norm{V}_{hyp}\leqapprox f^{1/2} \text{ and so }
			\norm{\grad f}_{hyp} \geqapprox f^{1/2}.
		\end{align}
	\end{enumerate}
\end{proposition}
Note that the inequality in the last point is stronger than the one in the second point only when $f$ is sufficiently small.
\begin{proof}
	Let us compute the differential of $f$:
	\begin{align*}
		df  & = d \ip{w^{0,4},w^{0,4}}_{H}\\
			& = \ip{\nabla^{Ch}w^{0,4},w^{0,4}}_H + 
			\ip{w^{0,4},\nabla^{Ch}w^{0,4}}_H\\
			& = \ip{-\sigma_{1,3}w^{1,3},w^{0,4}}_H +
			\ip{w^{0,4},-\sigma_{1,3}w^{1,3}}_H
	\end{align*}
	Note that the two terms are complex-conjugates of one another.
	Using the defining property of $V$ that $\sigma_{1,3}(V)w^{1,3}=-w^{0,4}$ it follows that $\cL_V f = 2f$ as claimed.
	Note that for a complex $1$-form $\alpha$ and vector field $V$ we have that $\iota_V\ov{\alpha}=\ov{\iota_V \alpha}$ for the insertion of a vector field.

	Next, observe that $\sqrt{-1} V$ satisfies $\cL_{\sqrt{-1}V}f=0$, since the first term in the formula for $df$ would yield $\sqrt{-1}f$ and the second one its complex-conjugate, so the two terms cancel out.
	It follows that $V$ and $\grad f$ are proportional.
	Combining the two formulas (all vectors measured in the hyperbolic metric)
	\[
		\cL_V f = 2f \quad \text{ and }\cL_{\grad f} f = \norm{\grad f}^{2}
	\]
	it follows that $\norm{\grad f}^2 V = 2f \grad f$, or by taking norms we find that 
	\[
		\norm{\grad f}\cdot \norm{V} = 2f.
	\]
	Note that if we rescale the background hyperbolic metric by a factor of $e^{2\lambda}$, so that $\norm{V}$ scales by $e^{\lambda}$, then $\grad f$ scales by $e^{-2\lambda}$ and thus $\norm{\grad f}$ will scale by $e^{-\lambda}$ so the stated identity is scale-invariant.

	To establish the estimate $1\geqapprox \norm{V}$, note that as a section of the tangent bundle $TX$, we have that $V=w^{0,4}/\sigma_{1,3}w^{1,3}$ by definition.
	Now from \autoref{prop:boundedness_of_second_fundamental_form} it follows that $1\leqapprox \norm{\sigma_{1,3}}\leqapprox 1$.
	From the condition $\ip{w,w}_I\geq 0$, see \autoref{eqn:basic_non_degeneracy}, we also have that $\norm{w^{0,4}}\leq \norm{w^{1,3}}$ and so the claimed estimate follows.

	To establish also the estimate $\norm{V}\leqapprox f^{1/2}$, observe that we can also use $\norm{w^{1,3}}\geq 1$, as well as $\norm{w^{0,4}}=f^{1/2}$, which can be plugged into the definition of $V$.
\end{proof}

The estimates we have obtained can now be used to analyze the Morse index of $f$, as well as establish some growth properties.

\begin{theorem}[Index estimates]
	\label{thm:index_estimates}
	Let $f=f_w$ be as above.
	\begin{enumerate}
		\item The only critical points of $f$ are local minima, which occur when $w^{0,4}(x)=0$.
		Furthermore these are nondegenerate local minima.
		\item If $\ip{w,w}_I\geq 0$ then $f$ has at most one critical point, and furthermore $\inf_{\wtilde{X}}f=0$.
		\item If $\ip{w,w}_I > 0$ then $f$ has precisely one critical point.
	\end{enumerate}
\end{theorem}
\begin{proof}
	It follows from \autoref{eqn:gradf_V_2f_identity} that if $\grad f$ vanishes, then necessarily $f$ also vanishes.
	So the only critical points are local minima.
	To check that they are nondegenerate, observe that $w^{0,4}$ is a holomorphic section of $\cW^{0,4}$, so it suffices to establish that it only has simple zeros.
	But $\nabla^{Ch}w^{0,4}=\sigma_{1,3}w^{1,3}\neq 0$ pointwise, so $w^{0,4}$ cannot vanish to second order or higher at a point.

	To establish (ii), namely that $f$ cannot have two critical points (necessarily local minima), we would like to invoke a mountain pass lemma.
	Most formulations assume that $f$ already satisfies a Palais--Smale condition, which is not (yet) available for $f$, so we will use it in the form which produces a Palais--Smale sequence.
	Specifically, by \cite[Thm.~3.1]{Bisgard2015_Mountain-passes-and-saddle-points}, suppose that $f(x_0)=f(x_1)=0$, with $x_0\neq x_1$, and let $r_0,\ve_0>0$ be such that $x_1\notin B(x_0,r_0)$ and $f\vert_{\partial B(x_0,r_0)}\geq \ve_0$.
	Then there exists a sequence $y_n$ such that $f(y_n)\geq \ve_0$ and $\norm{\grad f(y_n)}\to 0$.
	The assumptions are satisfied, by picking sufficiently small $r_0,\ve_0>0$.
	But the conclusion contradicts the steepness estimate \autoref{eqn:steepness_basic}, since that estimate shows that $\norm{\grad f(y_n)}\to 0$ implies $f(y_n)\to 0$.

	To check that $\inf_{\wtilde{X}}f=0$, let $c\geq 0$ be this infimum.
	By the Ekeland variational principle, see e.g. \cite[\S I.5]{Struwe2008_Variational-methods}, there exists a sequence $x_n$ with $f(x_n)\to c$ and $\norm{\grad f(x_n)}\to 0$.
	Again from the steepness condition \autoref{eqn:steepness_basic} it follows that $f(x_n)\to 0$, i.e. $c=0$.

	To establish (iii), namely that when $\ip{w,w}>0$ (say normalized to $1$), $f$ must achieve a minimum, we will show that the gradient flow converges, in fact, to a minimum located at finite distance.
	Specifically, start at some $x_0\in \wtilde{X}$ and let $\gamma\colon [0,t_0)\to \wtilde{X}$ be the unique curve with $\gamma(0)=x_0$ and $\dot{\gamma}(t)=-\grad f(\gamma(t))$, and $t_0>0$ is the maximal existence time (along the way we will establish $t_0=+\infty$).
	Set $g(t):=f(\gamma(t))$, so $-\dot{g}(t)=\norm{\grad f(\gamma(t))}^2$ which by the other steepness estimate \autoref{eqn:steepness_low_f} satisfies $-\dot{g}(t)\geq \ve_0\cdot g(t)$ for some $\ve_0>0$.
	It follows that $\dot{g}/g \leq -\ve_0$, so by integrating this logarithmic derivative (and exponentiating) on any interval $[a,b]\subset[0,t_0)$ we have $g(b)\leq g(a)e^{-\ve_0 \cdot t}$.

	The distance traversed by the curve $\gamma(t)$ on the time interval $[a,b]$ is $\int_{a}^{b}\norm{\dot{\gamma}(t)}dt$ which can be estimated by Cauchy--Schwartz:
	\begin{align*}
		\left(\int_{a}^{b}\norm{\dot{\gamma}(t)}dt\right)^2 & \leq 
		\left[\int_{a}^{b}\norm{\dot{\gamma}(t)}^2dt\right] \cdot (b-a)\\
		& = \left[\int_{a}^b |\dot{g}|dt\right] (b-a)\\
		& = (g(a)-g(b))\cdot (b-a) \leq g(a)\left(1-e^{-\ve_0\cdot (b-a)}\right)(b-a)
	\end{align*}
	Take now $b-a=1$ and use the estimate $g(a)\leq e^{-\ve_0\cdot a}g(0)$ to find that the length of the curve on a time interval $[x,x+1]$ is bounded by $\left(g(0) e^{-\ve_0\cdot x}(1-e^{-\ve_0})\right)^{1/2}$.
	Any interval $[x,y]\subseteq [0,\infty+)$ can now be divided into $\lceil y-x\rceil$ consecutive intervals $[x,x+1], [x+1,x+2], \ldots \left[\lfloor y\rfloor,y\right]$, and applying the previous estimate we find that the total length is bounded by
	\begin{align}
		\label{eqn:distance_bound_O1}
		\begin{split}
		\dist(\gamma(x),\gamma(y))& \leq g(0)^{1/2}(1-e^{-\ve_0})^{1/2}\sum_{n=0}^{\infty} e^{-\ve_0\cdot (x+n)/2}\\
		& 
		\leq g(0)^{1/2}\frac{\left(1-e^{-\ve_0}\right)^{1/2}}{1-e^{-\ve_0/2}}\\
		& \leq c_3 \cdot g(0)^{1/2}
		\end{split}
	\end{align}
	where $c_3$ is some fixed constant.
	In particular, $\gamma$ stays in a compact set $K$, for which there exists an $\ve_1(K)>0$ such that the gradient flow exists for time $\ve_1$ at any point in $K$, therefore the gradient flow exists for all (positive) times.

	The limit point of the curve $\gamma(t)$ is necessarily critical, thus a local minimum, and hence the global minimum with value $0$ as claimed.
\end{proof}

\begin{remark}[Formula for sum of Lyapunov exponents]
	\label{rmk:formula_for_sum_of_lyapunov_exponents}
	\autoref{thm:index_estimates} already implies the formula for Lyapunov exponents conjectured in \cite[\S1]{EskinKontsevichMoller2018_Lower-bounds-for-Lyapunov-exponents-of-flat-bundles-on-curves}.
	Indeed \autoref{thm:good_hypergeometrics} below shows that the seven ``good'' families of loc. cit. satisfy assumption A (and in fact our \autoref{table:good_hypgeom_real} and \autoref{table:good_hypgeom_integral} provide infinitely many other examples).

	Next, the proof of \cite[Thm.~4.1]{EskinKontsevichMoller2018_Lower-bounds-for-Lyapunov-exponents-of-flat-bundles-on-curves}, in the case at hand, starts with a Lagrangian subspace $L$ of $V_\bR$ and considers the corresponding vector $w_{L}\in W_{\bR}$ obtained by wedging two independent vectors in $L$.
	The argument then applies verbatim, with $u=w_L$, and observing that the function $\norm{u_z}$ in loc. cit. agrees with $f_{w_L}$ in the present text.
	Next, $f_{w_L}$ has at most one zero in all of $\bH$, by \autoref{thm:index_estimates}(ii).
	Then the inequality (15) in loc. cit. is an equality, up to adding a constant on the lhs if the hyperbolic disc $D(u)$ contains the one allowed zero.
	The chain of deductions following inequality (15) in loc. cit. then becomes a chain of equalities, if we take into account that $\lim_{T\to +\infty}\frac{1}{T}C=0$ for any constant $C$.
\end{remark}

\begin{remark}[MUM Lagrangian is good]
	\label{rmk:mum_lagrangian_is_good}
	Let us also explain why Conjecture~6.4 of \cite{EskinKontsevichMoller2018_Lower-bounds-for-Lyapunov-exponents-of-flat-bundles-on-curves} follows from \autoref{thm:index_estimates}(ii).
	The stated conjecture is equivalent to the statement that for a particular choice of Lagrangian $L \subset V_{\bR}$, the locus of vanishing of the function $\norm{u_z}$ is empty.
	The Lagrangian has the property that it fixed by the monodromy around a cusp of the Riemann surface.
	More generally, suppose that there exists an infinite order element $\gamma\in \pi_1(X)$ which fixes a Lagrangian $L\subset V_{\bR}$.
	Then if $w_{L}\in \Lambda^2L$ denotes the (necessarily isotropic) vector in $W_{\bR}$, the associated function $f_{w_L}$ of \autoref{thm:index_estimates} cannot have any zeros on $\wtilde{X}$.
	Indeed, the function and hence the set of zeros will be invariant under $\gamma$, any orbit of $\gamma$ is infinite on $\wtilde{X}$, while the theorem just proved shows that there is at most one zero.
\end{remark}

We next establish a quantitative growth property for the norm of positive-definite vectors.

\begin{lemma}[Exponential growth]
	\label{lem:exponential_growth_pos_def}
	Fix $w\in W_\bR$ satisfying $\ip{w,w}_{I}=1$, let $x_0\in \wtilde{X}$ be the unique minimum of $f:=f_w$.
	There exist uniform constants $\ve,C_1,C_2>0$ (independent of $w$ or $x_0$) such that for any $x\in \wtilde{X}$ we have:
	\[
		f(x)\geq \frac{1}{C_1}e^{\ve \dist(x_0,x)} - C_2
	\]
	where the distance is in the hyperbolic metric on $\wtilde{X}$.
\end{lemma}
\begin{proof}
	We will prove the equivalent
	\[
		\dist(x_0,x) \leq c_1' + \frac{1}{\ve}\log\left(1 + f(x)\right)
	\]
	As in the proof of \autoref{thm:index_estimates}, take $\gamma\colon [0,\infty)\to \wtilde{X}$ to be the gradient flow line satisfying $\dot{\gamma}(t)=-\grad f$.
	Next, observe that if we started at a point $x'$ where $f(x')\leq 2$, then by the estimate in \autoref{eqn:distance_bound_O1}, we have that $\dist(x_0,x')\leq c_3$ for some absolute constant $c_3$.
	So it suffices to prove that
	\[
		\dist(x,x')\leq c_1''+\frac{1}{\ve}\log(c_2+f(x))
	\]
	where $f(x)>2$ and $x'=\gamma(t')$ is such that $f(x')=\frac{f(x)}{2^k}\in [1,2)$, for some integer $k\geq 1$.

	Setting as before $g(t):=f(\gamma(t))$ we have that $\dot{g}=-\norm{\grad f}^2$ and so we need to estimate $\int_{0}^{t'}|\dot{g}(t)|^{1/2}dt$.
	To do so, we will use the other steepness inequality \autoref{eqn:steepness_basic} which gives
	\[
		-\dot{g}\geq \ve g^2 \text{ for some fixed }\ve>0.
	\]
	Let $t_1>0$ be such that $g(t_1)=\frac{1}{2}g(0)$, which exists since the gradient flow line eventually gets arbitrarily close to the minimum by \autoref{eqn:distance_bound_O1}.
	Let us rewrite the differential inequality as
	\[
		\frac{d}{dt}\frac{1}{g(t)}\geq \ve
	\]
	and integrate from $0$ to $t_1$ to find that
	\[
		\frac{1}{g(0)}\geq \ve t_1 \quad \text{ i.e. } t_1 \leq \frac{1}{g(0)\ve}.
	\]
	In particular, since $g(0)\geq 2$ we have that $t_1$ is uniformly bounded.
	Next, we apply Cauchy--Schwartz again to find
	\begin{align*}
		\left(\int_{0}^{t_1}|\dot{g}(t)|^{1/2}\right)^2 & \leq
		\left(\int_0^{t_1}|\dot{g}(t)|(t+1)dt\right)
		\left(\int_{0}^{t_1}\frac{dt}{t+1}\right)\\
		& \leq (t_1+1)\left(\int_{0}^{t_1}|\dot{g}(t)|dt\right)
		\log(1+t_1)\\
		& = (t_1+1)\cdot \frac{g(0)}{2}\cdot \log(1+t_1)
	\end{align*}
	Using that $\log(1+x)\leq x$ for $x\geq 0$, that $t_1\leq \frac{1}{g(0)\cdot \ve}$ but is also uniformly bounded above, it follows that the distance is uniformly bounded by some absolute constant $c_4$, independently of $g(0)$.
	Repeating this argument for times $t_1<t_2<\cdots <t_k$ where $k=\lfloor\log_2 g(0)\rfloor $, we find that
	\[
		\dist(x,x')\leq c_4 \cdot k \leq c_4 \cdot \log_2 f(x).
	\]
	This implies the desired bound
	\[
		\dist(x,x_0)\leq c_3 + c_4 \log_2(1+f(x))
	\]
	where the additive factor of $1$ inside the logarithm is to ensure that the bound is valid even if $f(x)\leq 1$.	
\end{proof}
The same proof yields the following variant for isotropic vectors.
Indeed, the proof above shows that in logarithmic time the gradient flow brings the value of the function to within $O(1)$, and it suffices to get the following conclusion:

\begin{lemma}[Exponential growth, isotropic version]
	\label{lem:exponential_growth_isotropic_version}
	Suppose that $w\in W_{\bR}\setminus \{0\}$ is an isotropic vector, such that $f_w$ achieves its (necessarily unique) minimum at $x_0$.
	Then there exist constants $C_1,C_2,\ve>0$, uniform for $w$ varying in a compact set and still satisfying the assumptions, such that
	\[
		f_w(x)\geq \frac{1}{C_1}e^{\ve \dist(x_0,x)} - C_2 \quad \forall x\in \wtilde{X}.
	\]
\end{lemma}

The above growth estimates will be used in \autoref{sec:hodge_theory_and_anosov_representations}, see e.g. \autoref{thm:anosov_property_from_assumption_A}, to establish that the monodromy representation is log-Anosov (after introducing the appropriate Lie-theoretic notions in \autoref{sec:elements_of_lie_theory_and_anosov_representations}).

%%					end of subsec: Gradient and index estimates
%%=============================================================================

%%=============================================================================
%%					start of subsec: Symplectic--Orthogonal dictionary

\subsection{Symplectic--Orthogonal dictionary}
	\label{ssec:symplectic_orthogonal_dictionary}

Before we can draw more consequences from the results of the previous section, we need to introduce some further concepts from symplectic geometry and relate it to the orthogonal group.
An exposition of many of the constructions described here is in the lectures notes \cite{BarbotCharetteDrumm2008_A-primer-on-the-21-Einstein-universe}, with the exception of the concept of ``electron'' which we introduce in \autoref{def:electrons} below.

\subsubsection{Setup}
	\label{sssec:setup_symplectic_orthogonal_dictionary}
Let $V$ be a real $4$-dimensional symplectic vector space.
We will typically write $I(v,w)$ for the symplectic pairing of two vectors.
When a basis will be useful, we take $e_1,e_2,f_1,f_2\in V$ to be one, satisfying $I(e_i,f_i)=1$ and vanishing otherwise.

We will also need $W:=\Lambda^2_\circ V$, the subspace of $\Lambda^2V$ of vectors in the kernel of the symplectic form.
Indeed $\Lambda^2 V$ is equipped with a symmetric inner product of signature $(3,3)$, given by wedging and using the volume form on $\Lambda^4 V$.
Restricted to $W$, the inner product has signature $(2,3)$.

The dictionary between the symplectic geometry of $V$ and the Minkowski geometry of $W$ is summarized in \autoref{sssec:symplectic_orthogonal_dictionary}.
Many constructions work over $\bR$ and over $\bC$ -- when a specific field is considered it will appear as a subscript; otherwise, the stated construction works over both fields.

\subsubsection{The corresponding groups}
	\label{sssec:the_corresponding_groups}
The map $\Sp(V)\to \SO(W)$ gives over $\bR$ the double cover $\Sp_4(\bR)\to \SO_{2,3}(\bR)$.
Restricted to maximal compact subgroups this yields $\bbU(2)\to \operatorname{S}(\Orthog_2(\bR)\times \Orthog_3(\bR))$.
When dealing with Hodge decompositions, the group $\bbU(1,1)$ will arise as the stabilizer of $F^2=V^{3,0}\oplus V^{2,1}$, and it maps to $\operatorname{S}(\Orthog_2(\bR)\times \Orthog_{1,2}(\bR))$.

\subsubsection{Lagrangian Grassmannian}
	\label{sssec:lagrangian_grassmannian}
Let $\LGr(V)$ denote the Grassmannian of Lagrangian subspaces of $V$.
The natural \Plucker embedding $\LGr(V)\into \bP(W)$ identifies the Lagrangian Grassmannian with the quadric of primitive vectors, when viewing $W$ as a subset of $\Lambda^2 V$.
Equivalently, $\LGr(V)$ is the quadric of isotropic vectors in $W$.

Over $\bR$ it can be useful to consider the Grassmannian of \emph{oriented} Lagrangians, which is a double cover of $\LGr(V_{\bR})$ and will be denoted $\LGr^+(V_{\bR})$.
If $\bP^+(W_{\bR})$ is the double cover of $\bP(W_{\bR})$, then we have the natural inclusion $\LGr^{+}(V_{\bR})\subset \bP^+(W_{\bR})$.

\subsubsection{Causal structure}
	\label{sssec:causal_structure_LGr}
At an isotropic line $[w]$, the tangent space to the quadric of null vectors can be identified with the linear space
\[
	T_{[w]}\LGr(V) = \Hom\left([w],\leftquot{[w]}{[w]^\perp}\right)
\]
where $[w]^\perp$ is the orthogonal of $[w]$ for the indefinite pairing.
Over $\bR$ the quotient $\leftquot{[w]}{[w]^\perp}$ carries a natural inner product of signature $(1,2)$ and so $Q^0(W_\bR)$ carries a natural conformal class of such inner products.
This determines a causal structure invariant under the corresponding group.
For this reason $\LGr(V_\bR)$ is also denoted $\Ein^{1,2}(W_\bR)$ as a ``conformal Einstein universe''.
It is a compactification of flat Minkowski space of signature $(1,2)$, which can be obtained by taking the complement of $[w]^{\perp}\cap\Ein^{1,2}(W_\bR)$, for any $[w]\in \Ein^{1,2}$ (the described intersection is the union of all null rays through $[w]$).

\subsubsection{Photons}
	\label{sssec:photons}
Let $l\subset V$ be a line.
It determines a flag $l\subset l^\perp\subset V$, where $\perp$ is taken for the symplectic form.
Define the \emph{photon} associated to $l$ by
\[
	\phi(l):= \left\lbrace L\subset V \colon L \text{ is Lagrangian and }l\subset L \right\rbrace \subset \LGr(V).
\]
The definition makes sense both over $\bR$ and over $\bC$, with associated photons denoted $\phi_{\bR}(l)$ (or just $\phi(l)$) and $\phi_{\bC}(l)$.

Note that any Lagrangian $L$ satisfying $l\subset L$ necessarily satisfies (by taking symplectic-orthogonal complements) $L\subset l^\perp$.
Therefore $\phi(l)$ is naturally parametrized by $\bP\left(\leftquot{l}{l^\perp}\right)$.
In the embedding $\LGr(V)\subset \bP(W)$ we have the natural identification $\phi(l)\isom \bP(l\wedge l^\perp)$.

In fact, any isotropic $2$-dimensional space $\xi_2\subset W$ determines a photon $\bP(\xi_2)\subset \bP(W)$.
For the Minkowski structure on $\LGr(V_{\bR})$ the curves $\phi(l)$ are lightlike, i.e. their tangent directions are null vectors.

\begin{definition}[Electrons]
	\label{def:electrons}
	Let $F^2\subset V_\bC$ be a $2$-dimensional subspace, Lagrangian for the symplectic form and with the restriction of the Hermitian pairing on $F^2$ of signature $(1,1)$.
	Define the subset
	\[
		\beta(F^2):=\left\lbrace L\in \LGr(V_\bR)\colon L_\bC \cap F^2 \neq \left\lbrace 0 \right\rbrace \right\rbrace
	\]
	and call it the \emph{electron} associated to $F^2$.
\end{definition}

\subsubsection{On electrons}
	\label{sssec:on_electrons}
Some of the reasons for calling $\beta(F^2)$ above an ``electron'' are:
\begin{itemize}
	\item In $\LGr(V_{\bR})$ with its Minkowski structure, the curves $\beta(F^2)$ are timelike (so they correspond to a particle moving below the speed of light).
	It will follow from \autoref{thm:anosov_property_from_assumption_A} that when the electron approaches the boundary it tends to a photon.
	\item In the physics literature, electrons are denoted by $e^-$ or $\beta^-$, and $\beta$ is also convenient for the terminology of ``bad'' Lagrangians, see \autoref{def:bad_bundle} below.
\end{itemize}
To geometrically identify the set $\beta(F^2)$, note that if a real Lagrangian $L$ satisfies $L_\bC\cap F^2 = \bC\alpha$ for $\alpha\in F^2$, then we have the following properties:
\begin{enumerate}
	\item $L_{\bC} = span(\alpha,\ov{\alpha})$
	\item $\ip{\alpha,\ov{\alpha}}_{I}=0$
\end{enumerate}
The first follows because $\alpha$ and $\ov{\alpha}$ are never proportional (since $V_{\bC}=F^2\oplus \ov{F^2}$ so $F^2$ cannot contain a real vector) and because $L$ is a real subspace, so $\ov{\alpha}\in L_{\bC}$.
The second property follows because $L$ is Lagrangian.

This discussion implies that we have a natural identification
\[
	\beta(F^2) = \left\lbrace [\alpha]\in \bP(F^2) \colon \ip{\alpha,\ov{\alpha}}_{I}=0 \right\rbrace \subset \bP(F^2)\isom \bP^1(\bC)
\]
In other words the photon identifies with the equator of null vectors in $\bP(F^2)$; this equator separates the upper/lower half planes of positive/negative directions.

The ``moduli space'' of electrons is the Grassmannian of complex Lagrangians where the hermitian pairing has signature $(1,1)$, which as a homogeneous space is:
\begin{align}
 	\label{eqn:space_of_electrons}
 	\rightquot{\Sp_4(\bR)}{U(1,1)}=L\LGr^{1,1}(V_\bC)
 \end{align}
 and is naturally a complex manifold.

\subsubsection{Electrons via second exterior power}
 	\label{sssec:electrons_via_second_exterior_power}
Just like there is a definition of the photon as $\phi(l)=\bP(l\wedge l^\perp)\subset \Ein^{1,2}\subset \bP(W)$, an analogous construction exists for electrons.
If $V_\bC=F^2\oplus \ov{F^2}$ with $F^2$ a complex Lagrangian of signature $(1,1)$, then in $W_\bC$ we have:
\[
	W_{\bC} = \underbrace{\left(\Lambda^2 F^2 \oplus \ov{\Lambda^2F^2}\right)}_{N_{\bC}} \oplus N^{\perp}_{\bC}
\]
The subspace $N_\bC$ is invariant under complex conjugation, hence leads to a real subspace $N\subset W_{\bR}$.
Its orthogonal complement also comes from a real subspace $N^\perp$.

Suppose now that $L\in \beta(F^2)$ is a real Lagrangian.
Then its \Plucker coordinate is orthogonal to $\Lambda^2F^2$, since a pair of subspaces intersect if and only if their \Plucker coordinates wedge to zero (so, are orthogonal for the specified inner product).
Because $L$ is real it is therefore also orthogonal to $\ov{\Lambda^2F^2}$, so the \Plucker coordinate of $L$ is in $N^\perp$, and it is an isotropic vector there.

To emphasize the dependence of $N$ on $F^2$, define:
\begin{align}
	\label{eqn:P_definition}
	N\left(F^2\right):= \left(\Lambda^2 F^2 \oplus \ov{\Lambda^2 F^2}\right)\cap W_\bR
\end{align}
and let $\Ein^{1}\left(N(F^2)^\perp\right)$ denote the projectivized quadric of null vectors in $N^\perp(F^2)$.
Then we have the identification
\[
	\beta(F^2) \isom \Ein^{1}\left(N(F^2)^\perp\right) \subset 
	\Ein^{1,2}\subset \bP(W_\bR).
\]
We have denoted by $\Ein$ the quadric of null vectors, and more generally we will use the notation $\Ein^{p,q}\subset \bP(\bR^{p+1,q+1})$ for the quadric of null vectors in a space of indefinite signature.

\subsubsection{Oriented electrons}
	\label{sssec:oriented_electrons}
Given $\alpha \in F^2\subset V_\bC$ with $\ip{\alpha,\ov{\alpha}}_{I}=0$, the associated real Lagrangian $span(\alpha,\ov{\alpha})$ has a natural orientation given by $\sqrt{-1}\alpha\wedge\ov{\alpha}$.
Changing $\alpha$ by a complex number does not affect the orientation.
Let us therefore denote by $\beta^{+}(F^2)\subset \LGr^{+}(V_{\bR})$ the family of oriented Lagrangians, and call this the oriented electron.

It follows that given an electron in $\LGr(V_\bR)$, its preimage in $\LGr^{+}(V_\bR)$ has two connected components.
This can also be seen in the model $\bP^+(W_\bR)$, where the preimage of the electron is a null quadric in a subspace of signature $(2,1)$.
It is connected in $\bP(W_\bR)$ but has two components in $\bP^+(W_\bR)$.
Since in $W_\bC$, the data of the electron is determined by the Hodge bundle $\cW^{4,0}=\Lambda^2 F^2\cV$, we will write $\beta^+(\cW^{4,0})$ to denote the two connected components of the oriented electron in $\bP^+(W_\bR)$.

\subsubsection{Pseudospheres}
	\label{sssec:pseudospheres}
Starting from a Hodge decomposition of $W_\bC$, let us denote by
\[
	\bS^{1,1}(\cW^{0,4}):=\{ w\in W_\bR\colon \ip{w,w}_{I}=1 \text{ and }\ip{w,\cW^{0,4}}_{I}=0 \}
\]
or in other words, the normalized real positive definite vectors that belong to $\cW^{3,1}\oplus \cW^{2,2}\oplus \cW^{1,3}$.
The corresponding projective two-to-one quotient will be denoted by $\bP^{1,1}(\cW^{0,4})$.

Note that $\bS^{1,1}(\cW^{0,4})$ is homeomorphic to $\bS^1\times \bR^1$, carries a pseudo-Riemannian metric of signature $(1,1)$, and its two ``ends'' on the null quadric are precisely the two components of the oriented electron.
Furthermore, we have the containment
\[
	\bS^{1,1}(\cW^{0,4})\subset \bS^{1,3}(W_{\bR})=\{w\in W_{\bR}\colon \ip{w,w}_I=1\}
\]
and we will denote by $\bP^{1,3}(W_{\bR})$ the two-to-one projective quotient.

\subsubsection{Symplectic--Orthogonal dictionary}
	\label{sssec:symplectic_orthogonal_dictionary}
The next table summarizes the constructions we introduced above.
\newcommand{\mc}[3]{\multicolumn{#1}{#2}{#3}}
\renewcommand{\arraystretch}{1.1}
\noindent
\begin{longtable}[c]{C{0.45\textwidth} C{0.54\textwidth}}
% \toprule
{\textbf{Symplectic}}            & {\textbf{Minkowski}}\\
{\centering $V_\bR \isom \bR^4$} & {$W_\bR = \Lambda^2_\circ (V_\bR) \isom \bR^{2,3}$} \\
\midrule
$\Sp_4(\bR)$ & $\SO_{2,3}(\bR)$ \\
$\bbU(2)$       & $S(\bbO_2(\bR)\times \bbO_3(\bR))$\\
$\bbU(1,1)$     & $S(\bbO_{2,1}(\bR)\times \bbO_2(\bR))$\\
\midrule
Lagrangian Grassmannian          & Einstein universe \\
$\LGr(V)$                        & $E^{1,2}\subset \bP(W)$ \\
& Quadric of isotropic vectors \\
\midrule
\multicolumn{2}{c}{\textbf{Photons}}\\
Line $l\subset V \Rightarrow$ Flag $l\subset l^\perp \subset V$ & Isotropic $2$-plane $l\wedge l^\perp \subset W$\\
$\phi(l)\subset \LGr(V)$                                        & $E^{1,2}\supset \bP(l\wedge l^\perp)\isom \phi(l)$\\
\midrule
\multicolumn{2}{c}{\textbf{Electrons}}\\
$V_{\bC} = F^2\oplus \ov{F^{2}}$                                & $W_{\bC} = \underbrace{\left(\Lambda^2 F^2 \oplus \ov{\Lambda^2F^2}\right)}_{N_{\bC}} \oplus N^{\perp}_{\bC}$ \\
signature of $F^2$ is $(1,1)$                                   & signature of $N$ is $(0,2)$, of $N^\perp$ is $(2,1)$ \\
$\beta(F^2)\subset \LGr(V_{\bR})$                               & $E^{1,2} \supset \Ein^1(N^{\perp})\isom \beta(F^2)$ \\
\midrule
% \midrule
$1+2+1$ dominated                & $2+1+2$ dominated\\
\bottomrule
 \end{longtable}
The last row, related to domination, will be clarified after \autoref{thm:domination_in_log_anosov_representations}.

%%					end of subsec: Symplectic--Orthogonal dictionary
%%=============================================================================

%%=============================================================================
%%					start of subsec: Developing maps to real homogeneous spaces

\subsection{Developing maps to real homogeneous spaces}
	\label{ssec:developing_maps_to_real_homogeneous_spaces}

In this section we show that an assumption A VHS can be used to uniformize domains of discontinuity associated to its monodromy.
One domain is an open set in the real Lagrangian Grassmannian, while another is the full pseudosphere of unit vectors in $W_{\bR}$.

\subsubsection{Setup}
	\label{sssec:setup_developing_maps_to_real_homogeneous_spaces}
We keep the notation from \autoref{ssec:hodge_theory_background} and \autoref{ssec:gradient_and_index_estimates}, and will make use of the concepts introduced in \autoref{ssec:symplectic_orthogonal_dictionary}.
In particular, $\wtilde{X}$ is the universal cover, we have a variation of Hodge structure $\bW\to X$ of weight $4$ and Hodge numbers $(1,1,1,1,1)$, and it satisfies assumption A from \autoref{def:assumption_a}.
In the discussion below we will pretend that $\bW$ is the (reduced) second exterior power of a weight $3$ variation $\bV$, but this is not necessary.
In order for it to be the case, one might have to pass to a finite cover of $X$, so that the monodromy representation can be lifted from $\SO_{2,3}^{\circ}(\bR)$ to its spin double cover $\Sp_{4}(\bR)$.
All the constructions below can be done while speaking of the variation $\bW$, but we will keep $\bV$ to clarify some of the geometric concepts.
To fix further notation, let
\[
	\rho_V\colon \pi_1(X)\to \Sp(V_{\bR})
	\quad
	\rho_W\colon \pi_1(X)\to \SO(W_{\bR})
\]
be the monodromy representations of the variations of Hodge structure.

In the context of the formula for Lyapunov exponents, in \cite{EskinKontsevichMoller2018_Lower-bounds-for-Lyapunov-exponents-of-flat-bundles-on-curves} a real Lagrangian is called ``bad'' if its complexification $L_\bC$ intersects the middle piece of the Hodge filtration $F^2(x)$ at some point $x\in \wtilde{X}$.
For a fixed $F^2(x)$, the collection of all ``bad'' real Lagrangians is the electron $\beta(F^2(x))$.

\begin{definition}[Bad bundle]
	\label{def:bad_bundle}
	Define $\wtilde{Bad}\to \wtilde{X}$ to be the bundle whose fiber above $x\in \wtilde{X}$ is $\beta(F^2(x))$.
	The bad bundle is naturally a subset $\wtilde{Bad}\subset \wtilde{X}\times \LGr(V_{\bR})$, which is $\pi_1(X)$-invariant when the action on the second factor is via the representation $\rho_V$.
	This invariance implies that $\wtilde{Bad}$ descends to a bundle
	\[
		Bad:=\left(\leftquot{\pi_1(X)}{\wtilde{Bad}}\right) \to \left(\leftquot{\pi_1(X)}{\wtilde{X}}\right) = X
	\]
\end{definition}

\begin{definition}[Pseudosphere bundle]
	\label{def:pseudosphere_bundle}
	Let $\wtilde{PS}\to \wtilde{X}$ be the bundle whose fiber above $x\in \wtilde{X}$ is $\bP^{1,1}(\cW^{0,4}(x))\subset \bP(W_\bR)$, see \autoref{sssec:pseudospheres}.
	There is a natural inclusion $\wtilde{PS}\subset \wtilde{X}\times \bP(W_{\bR})$ and since $\wtilde{PS}$ is a $\rho_W$-equivariant bundle, it descends to a bundle $PS\to X$.
\end{definition}

\subsubsection{Developing maps}
	\label{sssec:developing_maps}
Define the Lagrangian \emph{developing map} by projecting the bad bundle to the Lagrangian Grassmannian factor
\[
	\wtilde{Bad} \xrightarrow{\Dev^{L}} \LGr(V_{\bR}).
\]
The map $\Dev^{L}$ is by construction $\pi_1(X)$-equivariant.

Similarly, define the ``positive'' \emph{developing map} by projecting the pseudosphere bundle to the projective factor:
\[
	\wtilde{PS}\xrightarrow{\Dev^P} \bP(W_\bR)
\]
Recall also that $\bP^{1,3}(W_{\bR})$ denotes the projectivized positive-definite vectors, and so $\Dev^P$ lands there.

\subsubsection{Oriented versions}
	\label{sssec:oriented_versions}
We can also decorate the above constructions by taking the double covers that correspond to a choice of relevant orientation.
The corresponding objects will be denoted by a $+$ superscript, or by replacing $\bP$ with $\bS$ where appropriate.

Remark that if $\bW$ is the (reduced) second exterior power of $\bV$, then we have a preferred choice of connected component in $Bad^+$, as explained in \autoref{sssec:oriented_electrons}.

\begin{theorem}[Real uniformization, Lagrangian]
	\label{thm:real_uniformization_lagrangian}
	\leavevmode
	\begin{enumerate}
		\item There exists an open set $\Omega_{L}\subset \LGr(V_\bR)=\Ein^{1,2}(W_\bR)$ such that $\Dev^L$ is a bijective diffeomorphism between $\wtilde{Bad}$ and $\Omega_L$.
		\item Similarly, if $\bW=\Lambda^2_{\circ}\bV$, then the preimage $\Omega^+_L$ of $\Omega_L$ in $\LGr^+(V_\bR)$ has two components, and $\Dev^{L,+}$ is a bijection onto one of them.
		\item The action of the monodromy group on $\Omega_L$ is properly discontinuous and the quotient is diffeomorphic to the total space of the bundle $Bad\to X$.
		This, in turn, is diffeomorphic to the unit circle bundle of the complex line bundle $\cW^{3,1}\to X$, or a two-to-one quotient thereof, depending on whether the real part of the Hodge bundle $\cW^{2,2}$ is topologically trivial, or not.
	\end{enumerate}
\end{theorem}

\begin{theorem}[Real uniformization, pseudo-hyperbolic]
	\label{thm:real_uniformization_pseudo_hyperbolic}
	\leavevmode
	\begin{enumerate}
		\item The developing map $\Dev^{P}$ is a bijective diffeomorphism between $\wtilde{PS}$ and $\bP^{1,3}(W_{\bR})$.
		\item The same holds for $\Dev^{P,+}$, $\wtilde{PS}^+$, and $\bS^{1,3}(W_{\bR})$.
		\item The action of the monodromy group on $\bS^{1,3}(W_{\bR})$ is properly discontinuous, and the quotient is diffeomorphic to the total space of the bundle $PS^+\to X$.
		This, in turn, is diffeomorphic to the $\bC^{\times}$-bundle of nonzero vectors of the complex line bundle $\cW^{1,3}\to X$.
	\end{enumerate}
\end{theorem}

\begin{remark}[On uniformization]
	\label{rmk:on_uniformization}
	As will be proved later on, in \autoref{thm:anosov_property_from_assumption_A}, the open set $\Omega_L$ appearing in \autoref{thm:real_uniformization_lagrangian} is a domain of discontinuity coming from the log-Anosov property of the monodromy representation.
	These further properties will allow us to also identify the complement $\Lambda_L=\LGr(V_{\bR})\setminus \Omega_L$ with a limit set of the Anosov representation.
	As will be clear from the proof below, the information in \autoref{thm:index_estimates} already implies that for instance the ``MUM Lagrangian'' in the case of hypergeometric equations, see \autoref{rmk:formula_for_sum_of_lyapunov_exponents}, is in $\Lambda_L$.
\end{remark}

\begin{proof}[Proof of \autoref{thm:real_uniformization_lagrangian}]
	Both $\LGr(V_\bR)$ and $\wtilde{Bad}$ are real $3$-dimensional manifolds.
	The developing map is clearly smooth.
	We will establish below that it is also a local diffeomorphism (\autoref{sssec:the_developing_map_is_a_local_diffeomorphism}), i.e. has invertible differential.
	Let us check that it is also globally injective.
	Indeed, with notation from \autoref{ssec:gradient_and_index_estimates}, a null vector $w\in W_{\bR}$ determines a function $f_{w}$ on $\wtilde{X}$, as well as a Lagrangian $[w]\in \LGr(V_\bR)=\Ein^{1,2}(W_{\bR})$.
	Scaling $w$ only multiplies $f_w$ by a positive scalar.

	Now from the definition of $f_w$, it vanishes at $x\in \wtilde{X}$ if and only if $w^{0,4}(x)=0$, which happens if and only if $[w]\in \beta(\cW^{0,4})$, i.e. $[w]$ belongs to the corresponding electron.
	By \autoref{thm:index_estimates}, since $w$ is a null vector, $f_w$ can have at most one zero.
	Therefore, it is in the image of $\Dev^L$ for at most one point.

	The identification of $\wtilde{Bad}$ and the image of $\Dev^L$, denoted $\Omega_L$, is now clear.
	Proper discontinuity of the action on $\Omega_L$ follows since the identification is equivariant, and the action on $\wtilde{Bad}$ is visibly properly discontinuous since that is the case for the action of $\pi_1(X)$ on $\wtilde{X}$.

	It remains to identify the bundle $Bad$ with the appropriate circle bundle.
	It suffices to do so on the lift to the universal cover $\wtilde{Bad}$, in an equivariant way.
	First, by possibly passing to a double cover of $X$, we can assume that $\cW^{2,2}$ is a trivial bundle and has a preferred trivializing section.
	A double cover suffices since $\cW^{2,2}$ has a real structure and compatible inner product which yields a zero curvature metric.

	Now, given a Hodge decomposition
	\[
		\cW^{4,0}(x)\oplus \cW^{3,1}(x) \oplus \cW^{2,2}(x)
		\oplus \cW^{1,3}(x) \oplus \cW^{0,4}(x)
	\]
	we can parametrize the real null vectors orthogonal to $\cW^{0,4}$ by the unit circle bundle of $\cW^{1,3}$ by:
	\[
		w = \tfrac{1}{\sqrt{2}}\ov{z} \oplus 1 \oplus \tfrac{1}{\sqrt 2}z \quad \text{ for } z\in \cW^{1,3} \text{ and }\norm{z}^2=1
	\]
	where $1$ denotes the fixed unit norm trivializing section of $\cW^{2,2}$.
\end{proof}
\begin{proof}[Proof of \autoref{thm:real_uniformization_pseudo_hyperbolic}]
	The proof is completely analogous to the previous one.
	The only difference is that if $w$ is a positive definite vector in $W_{\bR}$, then by \autoref{thm:index_estimates} the function $f_w$ has precisely one zero, so $\Dev^{P}$ is not only injective but also surjective.

	To exhibit the identification of the fibers with Hodge data, we proceed analogously, although the formulas are slightly different:
	\begin{align}
		\label{eqn:developing_map_pseudosphere}
		w := \frac{\ov{z}}{\sqrt{2}\norm{z}} \oplus
		\frac{\norm{z}^2-1}{\norm{z}^2+1} \oplus
		\frac{z}{\sqrt{2}\norm{z}} \quad \text{ for } z\in \cW^{1,3}\setminus \{0\}.
	\end{align}
	Clearly the $\cW^{3,1}$ and $\cW^{1,3}$ components have norm $\tfrac{1}{2}$ each, and the middle component in $\cW^{2,2}$ has norm strictly bounded by $1$.
	The formula is obtained by normalizing any positive-definite vector to have $\cW^{3,1}$ component of norm $\tfrac 12$, and then identifying the interval $(0,\infty)$, to which $\norm{z}$ belongs, with $(-1,1)$ by a succession of $\log$ and hyperbolic tangent.
\end{proof}

\subsubsection{The developing map is a local diffeomorphism}
	\label{sssec:the_developing_map_is_a_local_diffeomorphism}
It remains to verify that $\Dev^L$ and $\Dev^P$ are local diffeomorphisms (as opposed to merely injective maps).
We will only present the calculations in the case of $\Dev^P$, the case of $\Dev^L$ being completely analogous.
When computing the relevant differential, it will also be the case that the matrix of interest will be block-triangular, and in fact the differential of $\Dev^L$ could be identified with a block-diagonal subpart.
Furthermore, the case of $\Dev^L$ is treated by \cite[Prop.~4.21]{CollierTholozanToulisse2019_The-geometry-of-maximal-representations-of-surface-groups-into} in a slightly more general setting, but it also involves a global argument, since their derivative ends up having an extra term coming from a more general Higgs field.
In the case of variations of Hodge structure the corresponding term vanishes, so the result is a purely local calculation.

\subsubsection{Explicit trivialization of the developing map}
	\label{sssec:explicit_trivialization_of_the_developing_map}
We work locally on $\wtilde{X}$, equipped with a coordinate $z$.
Let also $\alpha(z)$ be a local holomorphic section of $\cW^{1,3}$ and define $s(z):=\frac{\alpha(z)}{\norm{\alpha(z)}}$ to be the associated unit-norm section.
Let also $s_{2,2}$ be a local trivializing, unit norm, section of $\cW^{2,2}$.
Then the developing map can be described locally as
\begin{align}
	\label{eqn:local_dev_map_formula}
	\begin{split}
	\wtilde{X}\times (-1,1) \times \bT^1
	& \xrightarrow{\Dev^S}
	\bP^{1,3}(W_{\bR})\\
	(z,t,\theta) & \mapsto
	[w(z,t,\theta)]:=
	e^{-\sqrt{-1}\theta} \frac{\ov{s(z)}}{\sqrt{2}} +
	t\cdot s_{2,2} +
	e^{\sqrt{-1}\theta} \frac{s(z)}{\sqrt{2}}
	\end{split}
\end{align}
Note that the three terms on the right are the coordinates with respect to the Hodge decomposition at $z$ in positions $(3,1),(2,2),(1,3)$.
We will omit the superscript $S$ and just write $\Dev$ from now on, and also write $\bP^{1,3}$ instead of $\bP^{1,3}\left(W_{\bR}\right)$.

Let us introduce some further notation.
Work in a neighborhood of a point $z_0\in \wtilde{X}$.
Using the Hodge decomposition at $z_0$, we can write:
\begin{align*}
	N = N(z_0) & :=\left(\cW^{4,0}\oplus \cW^{0,4}\right)\cap W_{\bR}\\
	P = P(z_0) & :=\left(\cW^{3,1}\oplus \cW^{1,3}\right)\cap W_{\bR}\\
	l = l(z_0) & :=\cW^{2,2}\cap W_{\bR}
\end{align*}
so that $W_{\bR} = N \oplus P \oplus l$ which will be our basic decomposition with respect to which we will further decompose vectors and linear maps.

Consider now a point $[w_0]\in \bP^{1,3}$ in the image of the fiber over $z_0$, i.e. $[w_0]=\Dev(z_0,t_0,\theta_0)$ for some $(t_0,\theta_0)$.
By changing $\alpha(z)$ by a complex scalar of norm $1$, and to simplify notation, we can assume that $\theta_0=0$.
Introduce the subspaces
\begin{align*}
	P' & =P\cap w_0^{\perp} \text{ which is of the form} \left(\sqrt{-1}\cdot \ov{s(z_0)}+ 0+ \sqrt{-1}\cdot s(z_0)\right)\\
	l' & = \text{vectors of the form }\left(\frac{t_0}{\sqrt{2}}\ov{s(z_0)} + s_{2,2} + \frac{t_0}{\sqrt{2}}s(z_0)\right)
\end{align*}
where we have included the components with respect to the three pieces of the Hodge decomposition separately.
Computing the dot products, it follows that
\[
	[w_0]^{\perp} \isom N \oplus l' \oplus P'
\]
and so:
\[
	T_{[w_0]}\bP^{1,3} = \Hom([w_0],[w_0]^{\perp}) = \Hom([w_0], \oplus l' \oplus P' ).
\]

\subsubsection{The derivative of the developing map}
	\label{sssec:the_derivative_of_the_developing_map}
We now have a decomposition of the tangent space on the source of the developing map, and on the target, so the differential can be decomposed (writing $\bI:=(-1,1)$):
\[
	T_{z_0}\wtilde{X} \times T_{t_0}\bI \times T_{0}\bT^1
	\xrightarrow{D_{(z_0,\theta_0,t_0)}\Dev^S}
	\Hom([w_0],N\oplus l' \oplus P')
\]
To simplify notation, we will write $D$ for $D_{(z_0,\theta_0,t_0)}\Dev^S$ and omit the $\Hom([w_0],-)$ in the decomposition on the right.

It is now immediate to see from \autoref{eqn:local_dev_map_formula} that the differential is lower-triangular:
\[
	D
	=
	\begin{bmatrix}
		D\left(T_{z_0}{\wtilde{X}}\to N\right) & 0 & 0\\
		D\left(T_{z_0}{\wtilde{X}}\to l'\right) & D\left(T_{t_0}\bI \to l'\right) & 0\\
		D\left(T_{z_0}{\wtilde{X}}\to P'\right) & D\left(T_{t_0}\bI \to P'\right) & D\left(T_{0}\bT^1\to P'\right)
	\end{bmatrix}
\]
Indeed differentiating \autoref{eqn:local_dev_map_formula} in the $\theta$-parameter only changes the $P$-coordinate of the image, and moves it in the $P'$-direction.
Similarly, differentiating in the $t$-direction gives a non-trivial variation in the $l'$-direction.
So it remains to check that $D(T_{z_0}\wtilde{X}\to N)$ is nondegenerate.
As we will see in a moment, this is precisely given by the second fundamental form $\sigma_{1,3}$, which is the same as $\sigma_{4,0}^{t}$ and is nondegenerate by assumption A.

\subsubsection{Differentiating using the flat connection}
	\label{sssec:differentiating_using_the_flat_connection}
To compute the needed differential, we apply the Gauss--Manin connection to the individual Hodge components of $\Dev^S$ as defined in \autoref{eqn:local_dev_map_formula}.
We can drop any terms which belong to $P\oplus l$, to find (recall that $s$ is a section of $\cW^{1,3}$):
\begin{align*}
	\nabla^{GM}s & = \cancel{\nabla^{Ch}s} + \sigma_{1,3} s + \cancel{\sigma_{2,2}^{\dag}s}\\
	& = \sigma_{1,3}s \mod P\oplus l
\end{align*}
Similarly we find
\begin{align*}
	\nabla^{GM}\ov{s} & = \sigma_{4,0}^{\dag}s \mod P\oplus l \\
	\intertext{and }
	\nabla^{GM}s_{2,2} & = 0 \mod P\oplus l
\end{align*}
So we find that the induced map $T_{z_0}\wtilde{X}\to N$ is
\[
	v\mapsto \sigma_{1,3}(v)s + \ov{\sigma_{1,3}(v)s} \in N = \left(\cW^{4,0}(z_0) \oplus \cW^{0,4}(z_0)\right)\cap W_{\bR}
\]
In particular it is nondegenerate, and in fact if we equip $T_{z_0}\wtilde{X}$ and $N$ with their natural complex structures (on $N$ from the Hodge decomposition) then the isomorphism is conformal.

\begin{remark}[On naturality]
	\label{rmk:on_naturality}
	It is in fact also possible to avoid altogether local trivializations in the above computations.
	Indeed, on the total space of $\cW^{1,3}\to \wtilde{X}$, with zero section removed, there is a natural decomposition of the tangent space coming from the Chern connection.
	One can then decompose the differential as in \autoref{sssec:the_derivative_of_the_developing_map} and write each block of the matrix in terms of naturally defined operators, without the need to choose a trivialization.
\end{remark}

%%					end of subsec: Developing maps to real homogeneous spaces
%%=============================================================================

%%%%%%%%%%%%%%%%%%%%%%%%%%%%%%%%%%%%%%%%%%%%%%%%%%%%%%%%%%%%%%%%%%%%%%%%%%%%%%%
%%% 				End of Section: Hodge theory and growth
%%%%%%%%%%%%%%%%%%%%%%%%%%%%%%%%%%%%%%%%%%%%%%%%%%%%%%%%%%%%%%%%%%%%%%%%%%%%%%%

%%%%%%%%%%%%%%%%%%%%%%%%%%%%%%%%%%%%%%%%%%%%%%%%%%%%%%%%%%%%%%%%%%%%%%%%%%%%%%%
%%% 				Start of Section: Hypergeometric local systems
%%%%%%%%%%%%%%%%%%%%%%%%%%%%%%%%%%%%%%%%%%%%%%%%%%%%%%%%%%%%%%%%%%%%%%%%%%%%%%%

\section{Hypergeometric local systems}
	\label{sec:hypergeometric_local_systems}

For general information on hypergeometric equations, see Beukers--Heckman \cite{Beukers_Heckman} or Yoshida's book \cite{Yoshida1997_Hypergeometric-functions-my-love}.

\paragraph{Outline}
In this section we show that a large class of examples that satisfy assumption A arises from hypergeometric equations.
The classification of those hypergeometric parameters is accomplished in \autoref{ssec:assumption_a_for_hypergeometrics} and amounts to essentially a local calculation.
Indeed the condition can be seen as a ``tiling'' property which is local, in the spirit of triangle reflection groups arising from hypergeometric equations that must have angles of the form $\pi/N$ for integers $N$.
The local classification, based on the possible degenerations and the ``tiling'' condition, is in \autoref{prop:assumption_a_in_the_local_case}.

We take up the ``tiling'' point of view further and in \autoref{ssec:schwarz_reflection_for_hypergeometrics} we describe the Schwarz reflection structure of general hypergeometric equations.
This embeds the monodromy in a reflection group with index $2$, and we provide the explicit matrices in terms of the hypergeometric parameters.

Next, in \autoref{ssec:polyhedra_and_reflections} we specialize again to the assumption A case, and describe the tiling of a pseudosphere obtained from using Schwarz reflection.
One can reverse the reasoning and obtain many of the global results on the monodromy by starting from the tiling, but we will not pursue this here.

The period map of the VHS on the real axis takes a particularly simple geometric form and we describe it explicitly in \autoref{sssec:period_map_on_the_real_axis} and subsequent paragraphs.
This is of interest since in general the Griffiths transversality conditions impose restrictions on the period map which are hard to understand globally, but as we show here some control coming from the pseudo-Riemannian point of view can be obtained.

% Finally, in \autoref{ssec:explicit_nonvanishing_of_the_matrix_coefficient}, we make explicit the non-vanishing conjecture of \cite{EskinKontsevichMoller2018_Lower-bounds-for-Lyapunov-exponents-of-flat-bundles-on-curves}.
% One modification is that we work out the orbifold uniformization of the base, instead of just putting a puncture at the orbifold as in loc.cit., obtaining a slightly stronger statement.

\laternote{Resurrect remark above}

%%=============================================================================
%%					start of subsec: Assumption A for hypergeometrics

\subsection{Assumption A for hypergeometrics}
	\label{ssec:assumption_a_for_hypergeometrics}

\subsubsection{Setup}
	\label{sssec:setup_assumption_a_for_hypergeometrics}
Hypergeometric equations are traditionally defined on $X:=\bC\setminus \{0,1\}$, although as we shall see in some of the examples of interest it will be useful to consider either $0$ or $\infty$ as orbifold points of some finite order.
Let $\alpha_1,\beta_1,\ldots,\alpha_n,\beta_n$ be a list of $2n$ numbers (while we'll aways take them real later, at this stage they can be complex).
Define the differential operator
\[
	D_{\alpha,\beta}:=\prod_{i=1}^n \left(D - \beta_i\right)
	- z\prod_{i=1}^n (D+\alpha_i)\quad \text{ where } D = z\partial_z.
\]
With these conventions, the exponents (or: Riemann scheme) of the operator are:
\begin{itemize}
	\item at $0$: $\beta_1,\ldots, \beta_n$
	\item at $\infty$: $\alpha_1,\ldots, \alpha_n$
	\item at $1$: $0, 1, \ldots, n-2, \gamma:=(n-1) - \sum_{i=1}^n (\alpha_i + \beta_i)$
\end{itemize}
The parameters are slightly different from those of \cite{Beukers_Heckman}, specifically $\beta_i^{BH}=1-\beta_i$, and also slightly different from \cite{Fedorov2018_Variations-of-Hodge-structures-for-hypergeometric-differential-operators-and-parabolic} where $\alpha_i^{Fe}=-\alpha_i$.

\subsubsection{Monodromy}
	\label{sssec:monodromy_hypergeometric}
The solutions of hypergeometric equations give a local system over the base space, and so a representation $\rho\colon \pi_1(X)\to \GL_n(\bC)$.
When the parameters $\alpha_i$ are invariant modulo $1$ under $x\mapsto -x$, and similarly for $\beta_i$, the representation takes real values.
Furthermore, if irreducible (see below) these local systems are (relatively) rigid, i.e. once the conjugacy classes of the monodromy matrices around cusps are fixed, the representation is unique.

\subsubsection{Variation of Hodge structure}
	\label{sssec:variation_of_hodge_structure_fedorov}
If the local system is rigid, it underlies a variation of Hodge structure by the results of Simpson \cite{Simpson1990_Harmonic-bundles-on-noncompact-curves}.
The Hodge numbers were calculated by Fedorov \cite[Thm.~1]{Fedorov2018_Variations-of-Hodge-structures-for-hypergeometric-differential-operators-and-parabolic}.
The recipe is as follows.
We assume that $\alpha_i,\beta_i$ are real, belong to $[0,1)$ and listed in nondecreasing order (this can always be assumed since shifting the parameters by integers does not affect the monodromy).
Assume, for simplicity, that both lists of numbers are symmetric with respect to $x\mapsto 1-x \mod 1$, so the passage between our list of parameters and Fedorov's is unaffected.
Assume, also, that $\alpha_i\neq \beta_j,\forall i,j$: this ensures the local system is irreducible.

Define now $\rho(k):=\# \{j \colon \alpha_j\leq \beta_k\}-k$
In other words, $\rho(k)$ can be computed by reading the list of all the numbers from left to right until reaching $\beta_k$, starting with the value $0$, and every time an $\alpha$-value is encountered adding $1$, and every time a $\beta$-value is encountered subtracting $1$.
Then the Hodge numbers of the underling VHS are given, up to shift, by $h^{p}=\# \rho^{-1}(p)$.

\subsubsection{Local exponents}
	\label{sssec:local_exponents}
The above algorithm will yield the parameter space of $\alpha_\bullet,\beta_{\bullet}$ for which the Hodge numbers of the rank $4$ local system are $(1,1,1,1)$.
We also need to make sure that assumption A from \autoref{def:assumption_a} is satisfied.
The criterion is provided by Lemma~6.3 of \cite{EskinKontsevichMoller2018_Lower-bounds-for-Lyapunov-exponents-of-flat-bundles-on-curves}, although see some caveats in \autoref{sssec:on_local_exponents_and_regularity} below.
The map $\sigma_{2,1}$ of interest to us is denoted in loc.cit. by $\tau_1$.

Specifically, the criterion implies that away from the singular points $\{0,1,\infty\}$ all the $\sigma_{p,q}$ are isomorphisms.
At the point $1$, the monodromy is a rank $1$ unipotent matrix preserving a line but not a $2$-dimensional subspace, so the map $\sigma_{2,1}$ is an isomorphism by the analysis in \autoref{prop:boundedness_of_second_fundamental_form}, it corresponds to the second diagram in \autoref{fig:Hodge_numbers} (see also \autoref{sssec:on_local_exponents_and_regularity} below).
We now list in more detail the situations around singular points where assumption A is satisfied locally, depending on the local exponents and including the orbifold case.

\begin{proposition}[Assumption A in the local case]
	\label{prop:assumption_a_in_the_local_case}
	Suppose that $0\leq\mu_1\leq\mu_2\leq \mu_4\leq \mu_4<1$ are also symmetric under $x\mapsto-x \mod 1$ and give the local exponents at the origin of the Picard--Fuchs equation of a variation of Hodge structure on the punctured unit disc.
	Then there is a sufficiently small neighborhood of the origin on which assumption A is satisfied, perhaps in the orbifold sense i.e. after passing to a finite cover, if and only if we are in the following cases:
	\begin{align*}
		(0,0,0,0) \text{ or } (\tfrac 12,\tfrac 12,\tfrac 12,\tfrac 12)
		&\\
		(0,0,\mu,1-\mu) \text{ or } \left(\mu,\tfrac{1}{2},\tfrac{1}{2},1-\mu\right)
		& \quad \text{any real }\mu\in\left(0,\tfrac{1}{2}\right)\\
		\left(
		\tfrac{N-(2k+1)}{2N},
		\tfrac{N-1}{2N},
		\tfrac{N+1}{2N},
		\tfrac{N+(2k+1)}{2N}
		\right) & \quad
		\text{any integers } k\geq 1, N>2k+1.
	\end{align*}
\end{proposition}
\begin{proof}
	We will refer to \autoref{fig:Hodge_numbers} and refer to the possibilities for the unipotent part as case 1, case 2, and case MUM, for the second, third, and fourth diagrams of Hodge numbers.
	We will now list the possibilities for exponents, in order of increasing nondegeneracy, and assign them respective cases.
	Case 1 and case MUM are good, case 2 isn't.

	Suppose first that $\mu_1=0$ then by symmetry also $\mu_2=0$.
	If $\mu_3=0$ then necessarily $\mu_4=0$, we are in case MUM.
	If $\mu_3>0$, we could have $\mu_3=\mu_4=\tfrac 12$.
	This leads to case 2 (there's on a finite cover a $2$-dimensional invariant space).
	Otherwise $\mu_4=1-\mu_3\neq \mu_3>0$.
	This leads to case 1, and note that $\mu_3\in(0,\tfrac 12)$ can be an arbitrary real.

	Now suppose that $\mu_1> 0$, so we have $\mu_4=1-\mu_1$.
	If $\mu_2=\mu_1$, so $\mu_3=\mu_4$ then we are again in case 2, unless also $\mu_2=\mu_3$, i.e. all are equal leading to exponents $(\tfrac 12,\tfrac 12,\tfrac 12,\tfrac 12)$, which is case MUM (the elliptic part is $-1$, which can be eliminated by passing to a double cover).
	If $\mu_2>\mu_1$, we could have $\mu_2=\mu_3=\tfrac 12$, which leads again to case 1 and exponents $(\mu,\tfrac 12, \tfrac 12,1-\mu)$ for any real $\mu\in(0,\tfrac 12)$.

	So we are left with the least degenerate case $0<\mu_1<\mu_2<1-\mu_2<1-\mu_1$.
	The monodromy is thus an elliptic matrix, so unless it is of finite order we are not in a good case.
	When passing to a finite cover of order $N$, the exponents multiply by $N$.
	By \cite[Lemma~6.3]{EskinKontsevichMoller2018_Lower-bounds-for-Lyapunov-exponents-of-flat-bundles-on-curves} (it actually applies in this case) for assumption A to hold, we must have $N(1-\mu_2)-N(\mu_2)=1$, or equivalently $\mu_2=\frac{N-1}{2N}$.
	Furthermore the local system $\bW$ should be trivial after passage to this finite cover, so its exponents should be integers.
	This leads to the requirement $N(\mu_2-\mu_1)=k\in \bN$, or equivalently $\mu_2 = \frac{N-(2k+1)}{2N}$ with $N>2k+1$.
	So this leads to the final set of good possibilities:
	\[
		\left(
		\frac{N-(2k+1)}{2N},
		\frac{N-1}{2N},
		\frac{N+1}{2N},
		\frac{N+(2k+1)}{2N}
		\right)
	\]
	with the requirements $k\geq 1, N>2k+1$ (in particular $N\geq 4$).
\end{proof}
We are now ready to list the hypergeometric equations which satisfy assumption A.
\begin{theorem}[Good hypergeometrics]
	\label{thm:good_hypergeometrics}
	The variation of Hodge structures underling a hypergeometric equation satisfies assumption A if and only if it appears, up to exchanging the roles of $\alpha$ and $\beta$, in either \autoref{table:good_hypgeom_real} or \autoref{table:good_hypgeom_integral} from the introduction.
	% The last inequality in the table is equivalent to $\tfrac{N-1}{2N} < \tfrac{M-(2k_M+1)}{2M}$.
\end{theorem}
\begin{proof}
	To obtain the list, we take all the possible good local choices from \autoref{prop:assumption_a_in_the_local_case}, subject to the condition that $\alpha_i\neq \beta_j \quad \forall i,j$ and that the Hodge numbers, as dictated by Fedorov's recipe in \autoref{sssec:variation_of_hodge_structure_fedorov}, are $(1,1,1,1)$.
	This occurs if and only if, when placing the numbers on the circle, all the $\alpha$'s and all the $\beta$'s form contiguous clusters (the opposite of interlacing, which leads to finite monodromy).
\end{proof}

\subsubsection{On local exponents and regularity}
	\label{sssec:on_local_exponents_and_regularity}
Let us clarify the following point.
It is asserted in \cite{EskinKontsevichMoller2018_Lower-bounds-for-Lyapunov-exponents-of-flat-bundles-on-curves}, towards the top of pg. 2328, that if the local exponents at a point are $k,k+1,k+2,\ldots, k+n-1$ then the point is regular.
This requires some qualifications, since for instance the differential operator $(D+1)D + z$ has local exponents $-1,0$ at the origin.
It has a single-valued solution $f_0(z)=1+a_1z+\cdots$ corresponding to the exponent $0$, but a multi-valued solution $f_0(z)\log z + f_{-1}(z)$ where $f_{-1}(z)=\tfrac{1}{z} + b_0 + b_1z + \cdots$.

Such a situation can (and does) occur at the point $z=1$ of the hypergeometric equation, where taking the parameters to be $(\alpha_1,\alpha_2,1-\alpha_2,1-\alpha_1)$ and $(\beta_1,\beta_2,1-\beta_2,1-\beta_1)$ yields the exponents at $1$ to be $-1,0,1,2$.
Nonetheless, the monodromy matrix around $1$ is a rank $1$ unipotent preserving a line.

%%					end of subsec: Assumption A for hypergeometrics
%%=============================================================================

%%=============================================================================
%%					start of subsec: Schwarz reflection for hypergeometrics

\subsection{Schwarz reflection for hypergeometrics}
	\label{ssec:schwarz_reflection_for_hypergeometrics}

\subsubsection{Setup}
	\label{sssec:setup_schwarz_reflection_for_hypergeometrics}
We return for this section to the general setting of hypergeometric equations, see the notation in \autoref{sssec:setup_assumption_a_for_hypergeometrics}.
Assume that the parameters $\alpha_i,\beta_j$ are all real.
We now explain how to obtain a Schwarz reflection structure for the associated local system.
This will provide further geometric content to the underlying variation of Hodge structure.

\subsubsection{Levelt basis}
	\label{sssec:levelt_basis}
Let us fix a point in the upper half-plane, say $\sqrt{-1}$, which will be the basepoint for the fundamental group and monodromy representation (see however \autoref{rmk:groupoids_duality} below).
A key role throughout the discussion below will be played by the open real segments $(\infty,0),(0,1), (1,\infty)$.
Let the generators of the fundamental group be $g_0$, which crosses first $(0,1)$ and then $(\infty,0)$ and $g_{\infty}$ to cross first $(1,\infty)$ and then $(\infty,0)$.
Let also $g_1$ be the path which crosses first $(1,\infty)$ and then $(0,1)$, so that we have the relation $g_0g_1=g_{\infty}$.

Note that relative to the convention in \cite{Beukers_Heckman}, our $g_{\infty}$ is the same but $g_0,g_1$ are inverses of loc.cit.
Set also $a_j:=\exp(\twopii \alpha_j)$ and $b_j:=\exp(\twopii \beta_j)$, again note that our $a_j$ are the same as in loc.cit, while $b_j$ are the complex conjugates.
Setting
\begin{align*}
	p_A(t):=\prod_{j=1}^n (t-a_j) & = t^n + A_1 t^{n-1} + \cdots + A_n\\
	p_B(t):=\prod_{j=1}^n (t-b_j) & = t^n + B_1 t^{n-1} + \cdots + B_n
\end{align*}
the theorem of Levelt, \cite[Thm.~3.5]{Beukers_Heckman} gives that the monodromy of the hypergeometric equation can be conjugated to have the following generators along $g_\infty$ and $g_0$:
\begin{align}
	\label{eqn:Levelt_matrices_my_convention}
	h_{\infty}=A := 
	\begin{bmatrix}
	0 & 0 & \ldots & 0 & -A_n\\
	1 & 0 & \ldots & 0 & -A_{n-1}\\
	0 & 1 & \ldots & 0 & -A_{n-2}\\
	  &   & \ddots &   & \\
	0 & 0 & \ldots & 1 & -A_{1}
	\end{bmatrix}
	\quad
	h_0 = B:= \begin{bmatrix}
	0 & 0 & \ldots & 0 & -\ov{B_n}\\
	1 & 0 & \ldots & 0 & -\ov{B_{n-1}}\\
	0 & 1 & \ldots & 0 & -\ov{B_{n-2}}\\
	  &   & \ddots &   & \\
	0 & 0 & \ldots & 1 & -\ov{B_{1}}
	\end{bmatrix}
\end{align}

\subsubsection{Schwarz reflection}
	\label{sssec:schwarz_reflection}
Since we assumed that our parameters $\alpha_i,\beta_j$ are real, it follows that the differential equation has a complex-conjugation symmetry.
Specifically, starting with a basis of solutions near $\sqrt{-1}$, we can analytically continue along a path that crosses the real axis and reaches $-\sqrt{-1}$, and then apply complex-conjugation, i.e. $\tilde{f}(z):=\ov{f(\ov{z})}$ to obtain another set of functions defined near $\sqrt{-1}$.
Because the differential operator has real coefficients, the functions are again a basis of solutions.

If $V_{\bC}$ denotes the complex vector space of solutions near $\sqrt{-1}$, then the mapped just described is a complex anti-linear self-map of $V_{\bC}$.
Furthermore, it is visibly an involution.
The real subspace of solutions fixed by this operation can be described as follows.
Pick a point $p\in (a,b)$, where $(a,b)$ is the open segment of the real axis crossed by the defining path.
Then around $p$, there is a basis of solutions whose power series have real coefficients.
The analytic continuation of these solutions to $\sqrt{-1}$ yields the fixed points of the complex anti-linear involution described.

\subsubsection{Matrices of anti-involutions}
	\label{sssec:matrices_of_anti_involutions}
Let $\cR_A,\cR_B,\cR_C$ be the transformations described in \autoref{sssec:schwarz_reflection} above, corresponding to crossing along $(1,\infty)$, $(0,1)$, and $(\infty,0)$ on the real axis.
Then we must have
\[
	h_0 = B = \cR_C \circ \cR_B \quad
	h_\infty = A = \cR_C \circ \cR_A \quad
	h_1 = \cR_B \circ \cR_A \quad
\]
Denote by $\cC\colon \bC^n\to \bC^n$ the operation of coordinate-wise complex conjugation.
Since the transformations $\cR_{\bullet}$ are complex anti-linear, there exist complex linear transformations $R_A,R_B,R_C$ such that $\cR_{X}=\cC\circ R_X$.
After conjugation, these can be arranged to be:
\begin{gather}
	\label{eqn:reflection_matrices_Levelt}
	\begin{align*}
	R_A:= \begin{bmatrix}
	0 & \ldots & 0 & 1 & -{A_1}\\
	0 & \ldots & 1 & 0 & -{A_{2}}\\
	  & \reflectbox{$\ddots$} &   &   & \\
	% 0 & \ldots & 0 & 0 & -{A_{n-1}}\\
	1 & \ldots & 0 & 0 & -{A_{n-1}}\\
	0 & \ldots & 0 & 0 & -{A_{n}}
	\end{bmatrix}
	\quad
	R_B:= \begin{bmatrix}
	0 & \ldots & 0 & 1 & -\ov{B_1}\\
	0 & \ldots & 1 & 0 & -\ov{B_{2}}\\
	  & \reflectbox{$\ddots$} &   &   & \\
	% 0 & \ldots & 0 & 0 & -\ov{B_{n-1}}\\
	1 & \ldots & 0 & 0 & -\ov{B_{n-1}}\\
	0 & \ldots & 0 & 0 & -\ov{B_{n}}
	\end{bmatrix}
	\end{align*}\\
	% \begin{align*}
	\text{ as well as }
	R_C:= \begin{bmatrix}
	0 & 0 & \ldots & 0 & 1 \\
	0 & 0 & \ldots & 1 & 0 \\
	  & & \reflectbox{$\ddots$} &   &   \\
	% 0 & \ldots & 0 & 0 \\
	0 & 1 & \ldots & 0 & 0\\
	1 & 0 & \ldots & 0 & 0
	\end{bmatrix}
	% \end{align*}
\end{gather}
Clearly $\cR_{C}^2=1$, so let us verify that the same holds for $\cR_{A}$ (the calculation for $\cR_{B}$ is analogous).
Recall that $\alpha_i$ is real, therefore $a_i=\exp(\twopii \alpha_i)$ satisfies $a_i^{-1}=\ov{a_i}$.
The resulting relation for the coefficients $A_i$ becomes:
\begin{align*}
	p_A(t) & = \prod (t-a_i) = t^n \prod \left(1-\frac{a_i}{t}\right)\\
	& = t^n \left(\prod a_i\right)\prod\left(a_i^{-1}-\frac{1}{t}\right)\\
	& = t^n A_n \prod\left(\ov{a_i}-\frac{1}{t}\right) = t^n A_n \cdot \ov{p_A\left(\frac{1}{t}\right)}
\end{align*}
so we find that
\[
	A_l = A_n \cdot \ov{A_{n-l}} \quad \forall l=0,\ldots,n \text{ setting }A_0=1
\]
Computing now $R_A\circ \cC \circ R_A$ yields precisely the identity matrix, showing that $\cR_A^2=\id$.
\laternote{Prove that the above matrices are actually the correct ones.
Perhaps use Levelt basis?}	

\begin{remark}[Groupoids, duality]
	\label{rmk:groupoids_duality}
	\leavevmode
	\begin{enumerate}
		\item It would be more intrinsic to work with a fundamental groupoid in the above situation.
		Specifically, $\bP^1(\bC)\setminus \{0,1,\infty\}$ can be broken up into the two half-planes with nonzero imaginary part, each of which is contractible.
		These would be the two objects of the groupoid.
		The generating arrows (and their inverses) are labeled by which one of the real segments $(\infty,0),(0,1),(1,\infty)$ they cross.
		\item Some of the above discussion is implicit in \cite{Beukers_Heckman}, see e.g. Prop.~6.1 which shows that the dual local system is given by taking the inverse of the parameters $a_i=\exp(\twopii \alpha_i)$ and $b_j=\exp(\twopii \beta_j)$.
		\item A bit more generally, for the above discussion it suffices to have that the sets $\{\alpha_i\}$ and $\{\beta_j\}$ are invariant under $z\mapsto \ov{z}$.
		In this case the differential operator will still have real coefficients and the needed relation among the $A_i,B_j$ still holds.
	\end{enumerate}
\end{remark}

\subsubsection{Eigenvectors}
	\label{sssec:eigenvectors}
We make a further simplifying assumption, which holds in our examples of interest: assume that the coefficients $A_i$ and $B_j$ are real, i.e. the set of parameters $\alpha_{\bullet}$ is invariant under $x\mapsto -x \mod 1$ and the same for $\beta_{\bullet}$.
In this case, the monodromy group is real, and we can assume that the complex conjugation $\cC$ is with respect to this real structure.
Note also that each of the complex anti-linear transformations $\cR_{A},\cR_{B},\cR_C$ also induces a real structure -- given by its set of fixed vectors -- which corresponds to solutions of the differential equation which have real Taylor coefficients on the corresponding segments of the real axis.

We next look at the eigenvectors of the matrices $R_A,R_B$ and $R_C$ from \autoref{eqn:reflection_matrices_Levelt}.
It is clear that $R_C$ has $\left\lceil\tfrac{n}{2}\right\rceil$ eigenvalues $+1$, and $\left\lfloor\tfrac{n}{2}\right\rfloor$ eigenvalues $-1$, with eigenvectors which are symmetric, resp. antisymmetric, when flipping the coordinates across the center.
For $R_A$ and $R_B$, there are again $\left\lceil\tfrac{n-1}{2}\right\rceil$ eigenvalues $+1$, and $\left\lfloor\tfrac{n-1}{2}\right\rfloor$ eigenvalues $-1$, which are in fact common eigenvectors with the same eigenvalues, whose span is the subspace of vectors with last coordinate vanishing.
Finally, $R_A$ has an eigenvector
\[
	\begin{bmatrix}
	A_{n-1} & A_{n-2} & \cdots &  A_{1} &2
	\end{bmatrix}^t
	\text{with eigenvalue }-A_n.
\]
Note that by assumption $A_n^2=1$, and that we have the basic relation $A_i = A_n A_{n-i}$.
The eigenvector and eigenvalue for $R_B$ is analogous.

\laternote{include discussion of the bilinear form
I suspect that the action of the $R_\bullet$ is $(-1)^n$-symmetric, i.e. flips the symplectic form but preserves the orthogonal form.
Maybe easier to see at the complex level?}

%%					end of subsec: Schwarz reflection for hypergeometrics
%%=============================================================================

%%=============================================================================
%%					start of subsec: Polyhedra and reflections

\subsection{Polyhedra and reflections}
	\label{ssec:polyhedra_and_reflections}

An analysis of the local structure of the period map in the case which is of interest to us is in the paper of Bryant and Griffiths 
\cite{BryantGriffiths1983_Some-observations-on-the-infinitesimal-period-relations-for-regular-threefolds}.
Using the Schwarz reflection structure, we are also able to get some of its global properties as well.
Specifically, we get analogues of hyperbolic triangles, in this case polyhedra in the pseudo-sphere $\bS^{1,3}$.
The tiling condition, analogous to that in the \Poincare polyhedron theorem, is essentially equivalent to assumption A.

\subsubsection{Setup}
	\label{sssec:setup_polyhedra_and_reflections}
We specialize the discussion in \autoref{ssec:schwarz_reflection_for_hypergeometrics} to the case $n=4$ and with parameters as in \autoref{ssec:assumption_a_for_hypergeometrics}, i.e. satisfying assumption A.
Then all three reflections have the same structure: they have $2$-dimensional eigenspaces for the two eigenvalues $\pm 1$, and the eigenspaces are Lagrangian for the symplectic form.
We therefore focus on one of them, say $R_B$.

Pick a basis $e_1,e_2,f_1,f_2$ of $V_{\bR}$ such that for the symplectic pairing we have $\ip{e_i,f_i}_I=1$ and $0$ otherwise.
Assume also that
\[
	R_B e_i = -e_i \text{ and }R_B f_i = f_i
\]
and take a basis of the reduced second exterior power $W:=\Lambda^2_{\circ}V$:
\begin{align}
	\label{eqn:exterior_power_basis_coordinates_eigenvalues}
	\begin{matrix}
	e_1 \wedge e_2 & e_1 \wedge f_2 & e_1\wedge f_1 - e_2 \wedge f_2
	& f_1 \wedge e_2 & f_1\wedge f_2\\
	a & b & c & d & e \\
	+1 & -1 & -1 & -1 & +1
	\end{matrix}
\end{align}
where underneath we have written the corresponding names of the coordinate functions for an element of the vector space, and below them the eigenvalues of the matrix $\Lambda^2_{\circ} R_B$.
So the quadratic form on $W$ is (up to a factor of $\tfrac 12$):
\[
	Q(w,w):=-a\cdot e + b\cdot d - c^2.
\]
Note that the signature of $Q$ on $W_{\bR}$ is $(2,3)$.

\subsubsection{Period map on the real axis}
	\label{sssec:period_map_on_the_real_axis}
Recall that we are working in a neighborhood $U$ of the segment $(0,1)$ on the real axis.
Recall also that there are two variations of Hodge structure, $\bV$ and $\bW:=\Lambda^2_{\circ}\bV$.
We will describe the middle element of the Hodge filtration, i.e. $F^2\cV$, which is the same as to describe $\cW^{4,0}$.
Viewing it as a line in $W_{\bC}$, this gives a holomorphic map
\[
	\gamma\colon U \to \bP(W_{\bC})
\]
which furthermore is equivariant for the action of complex conjugation and $\cR_{\cB}$:
\[
	\gamma(\ov{z}) = \cR_{B} \left(\gamma(z)\right)
\]
where recall that $R_{\cB}=\cC\circ R_B$, where $\cC$ is complex conjugation of all coordinates.
Restricting to the real axis, and calling the coordinate $t$, we find
\[
	\gamma(t) = \cC \left(R_{B}\cdot \gamma(t)\right)
\]
Let us lift $\gamma$ to a map to $W_{\bC}$, with coordinates $(a:b:c:d:e)$, to find, using the information on the eigenvalues of $R_B$ from \autoref{eqn:exterior_power_basis_coordinates_eigenvalues}:
\[
	a = \ov{a} \quad 
	b = -\ov{b} \quad
	c = -\ov{c} \quad
	d = -\ov{d} \quad
	e = \ov{d}
\]
or equivalently:
\[
	\gamma = \left(a:\sqrt{-1}b:\sqrt{-1}c:\sqrt{-1}d:e\right) \text{ for }a,b,c,d,e \text{ real-valued functions.}
\]
Let us note that it is straightforward to read off the Lagrangian subspace in $V_{\bC}$ parametrized by $\gamma$, or equivalently factorize the vector with coordinates as above:
\begin{align}
	\label{eqn:gamma_factorization}
	\gamma =\frac{1}{a} \left[
	a\cdot  e_1 + \sqrt{-1}\left(d\cdot  f_1 + c\cdot f_2\right) 
	\right]
	\wedge
	\left[
	a\cdot  e_2 + \sqrt{-1}\left(c\cdot f_1 +  b\cdot f_2\right) 
	\right]
\end{align}
under the assumptions $ae = c^2 - bd = 1$, see \autoref{sssec:further_analysis_of_the_image_of_the_period_map} below.

\subsubsection{Riemann bilinear relations}
	\label{sssec:riemann_bilinear_relations}
To extract further information, we recall the constraints on $\cW^{4,0}$, which we denote by $\gamma$, and its motion:
\begin{description}
	\item[Lagrangian] The line is isotropic, i.e. $Q(\gamma,\gamma)=0$.
	\item[Definiteness] The line is negative-definite for the hermitian pairing, i.e. $Q(\gamma,\ov{\gamma})<0$.
	\item[Griffiths transversality] The tangent $\dot{\gamma}$ is also isotropic, i.e. $Q(\dot{\gamma},\dot{\gamma})=0$.
\end{description}
Let us now express these three conditions in terms of the coordinate functions introduced above (keeping in mind the factors of $\sqrt{-1})$:
\begin{align*}
	Q(\gamma,\gamma)=0 & \iff -ae-bd + c^2 =0\\
	Q(\gamma,\ov{\gamma})<0 & \iff -ae + bd - c^2 < 0\\
	Q(\dot{\gamma},\dot{\gamma})=0 & \iff -\dot{a}\dot{e} - \dot{b}\dot{d} + \dot{c}^2 = 0
\end{align*}
By adding the first two equations, it follows that $a\cdot e>0$.
We also have a scaling symmetry, i.e. we can multiply $\gamma$ by any nonvanishing function.
Let us then rescale it by the unique real-analytic function such that $a\cdot e=1$ holds pointwise, and $a>0$.
Then we the above equations can be rewritten as
\begin{align*}
	a\, e & = c^2 - b\, d = 1\\
	\dot{a}\, \dot{e} & = (\dot{c})^2 - \dot{b}\, \dot{d} \leq 0
\end{align*}
The assertion that the last line is nonpositive follows from differentiating $a\, e = 1$.
In fact, under assumption A (\autoref{def:assumption_a}), the inequality is strict since its vanishing at a point is equivalent to the vanishing of the second fundamental form $\sigma_{4,0}$ at the same point.
We assume for convenience assumption A from now on.

\subsubsection{Further analysis of the image of the period map}
	\label{sssec:further_analysis_of_the_image_of_the_period_map}
In the classical case of triangle reflection groups, the constraints of the Schwartz reflection are enough to pin down geometrically the \emph{image} of the period map along the real axis -- it is a piece of the conformal image of a great circle on the Riemann sphere.
Of course, one has in addition the \emph{parametrization} of this image, given by hypergeometric functions.

In the present case, the constraints coming from Schwarz reflection do not pin down uniquely the image, but they do give substantial geometric constraints.
Specifically, we can reparametrize the image of $(0,1)$ under $\gamma$ such that $a(t)=\exp(t)$ and $e(t)=\exp(-t)$ for $t\in(t_0,t_1)$ an appropriate time interval (perhaps infinite to one or both sides).
With these normalizations, the only remaining information is in the central three coordinates:
\begin{align}
	\label{eqn:period_map_causal_piece}
	\begin{split}
	c^2 - b\, d & = +1\\
	(\dot{c})^2 - \dot{b}\, \dot{d} & = -1
	\end{split}
\end{align}
Set $\gamma_1=(b,c,d)$, $\bR^{2,1}$ with coordinates $(b,c,d)$ and quadratic form $Q_1=c^2-bd$, and $\bS^{1,1}:=Q_1^{-1}(1)$.
Then $\gamma_1$ yields a curve on $\bS^{1,1}$ which moves in a timelike manner.
Note that because we normalized $a>0$, the curve is well-defined on $\bS^{1,1}$ and not its two-to-one projective quotient.

\subsubsection{Causal structure}
	\label{sssec:causal_structure}
A curve on $\bS^{1,1}$ satisfying \autoref{eqn:period_map_causal_piece} is rather heavily constrained.
Let us recall the basic structure of the pseudo-Riemannian spaces $\bS^{p,q}$ which are the unit vectors in $\bR^{p+1,q}$, with the induced pseudo-Riemannian metric $g_{p,q}$.
Let $(\bB^{q},g_q)$ denote hyperbolic space, in its ball model, with associated hyperbolic metric.
Let $(\bS^p,g_p)$ denote the sphere with its standard metric, and $\bS^{p}_{north}$ the open ``north'' spherical cap, with its induced spherical metric.
Then we have
\[
	\left(\bS^{p,q},g_{p,q}\right) \equiv \left(\bS^{p},g_p\right)\rtimes_{\rho}\left(\bB^q,-{g_q}\right)
	\approx_{conf} \left(\bS^{p},g_p\right)\times \left(\bS_{north}^q,-g_{q}\right)
\]
Let us explain the notation.
The middle term is a warped product pseudo-Riemannian metric, where the warping function $\rho$ on $\bB^{q}$ is $\left(\frac{1+\norm{\bbx}^2}{1-\norm{\bbx}^2}\right)^2$.
The first identity is an isometry between this warped product and the original $\bS^{p,q}$.
The second identity asserts that this warped product model is \emph{conformal}, as a pseudo-Riemannian space, to a product between $(\bS^{p},-g_{p})$ and the unit ball viewed as a spherical cap in $\bS^{q}$.
This last identity puts heavy constraints on any $q$-dimensional submanifold of $\bS^{p,q}$ on which the pseudo-Riemannian product is strictly negative-definite: it must be locally the graph of a Lipschitz function $f\colon \bB^{q}\to \bS^p$, and in fact globally so under some completeness assumptions.

The calculations that establish the above identities are contained, for instance, in \cite[Prop.~3.5]{CollierTholozanToulisse2019_The-geometry-of-maximal-representations-of-surface-groups-into} (except that loc.cit. treats $\bH^{p,q}$ and $p=2$, but the calculations are analogous in general).

\subsubsection{Behavior at endpoints}
	\label{sssec:behavior_at_endpoints}
For each of the endpoints of the original interval $(0,1)$, there are several possibilities for the behavior, depending on the type of degeneration of the Hodge structure.
We analyze the possibilities, according to the list in \autoref{fig:Hodge_numbers}.
Let us also emphasize that there are two kinds of limits we can take:
\begin{description}
	\item[Line limit] The limit of $\cW^{4,0}$ in $\LGr(V_\bC)\subset \bP(W_\bC)$.
	This is also affectionately called in the Hodge theory literature the ``naive'' limit, and will be denoted $\cW^{4,0}_{nv}$.
	\item[Plane limit] The limit of $\left(\cW^{4,0}\oplus \cW^{0,4}\right)\cap W_{\bR}$ in $\Gr(2;W_{\bR})$, which will be denoted $P^2_{nv}$.
\end{description}
Note that in the corresponding second exterior power of $W_{\bR}$, the coordinates of the vector of interest for the plane limit are
\[
	(a:0:0:0:e)\wedge (0:b:c:d:0) \in \Lambda^2 W_{\bR}
\]
since it is determined by the real and imaginary parts of the vector $\gamma$.
In fact, by \cite[Lemma~3.12]{CattaniKaplanSchmid1986_Degeneration-of-Hodge-structures}, see also \cite[Rmk.~3.15]{Robles2017_Degenerations-of-Hodge-structure}, the limits can be computed by looking at the decompositions $I^{p,q}$ of the limit mixed Hodge structure.

Let us note that $\LGr(V_{\bC})$ and $\Gr^+(2;W_{\bR})$ (the double cover of oriented $2$-planes) share some similar features, but also differences, for the action of $G=\Sp_{4}(\bR)$.
The open orbits coincide and give the two copies of Siegel space (depending on orientation/signature) and a copy of $\LGr^{1,1}(V_\bC)$ (Lagrangians where the hermitian pairing has signature $(1,1)$), resp. $\Gr^+_{0,2}\left(W_{\bR}\right)$ ($2$-planes where the quadratic form has signature $(0,2)$).
However, the minimal, closed, orbits are different.
In the Lagrangian Grassmannian it's $\LGr(V_{\bR})$, while in $\Gr^+(2;W_{\bR})$ it is the Grassmannian of isotropic $2$-planes in $W_{\bR}$, which is naturally identified with $\bP^{+}(V_{\bR})$.
As we shall see, the ``plane limit'', i.e. in $\Gr^+(2;W_{\bR})$, will be more useful -- in \autoref{thm:boundary_values_in_adjoint_representation} we show that \emph{all} directions in the boundary of the universal cover have continuous limits in $\Gr^+(2;W_{\bR})$.

Before proceeding to the analysis of degenerations, let us warn that we will refer to the formulas from the proof of \autoref{prop:boundedness_of_second_fundamental_form}, but the basis $e_i,f_i$ used there, and the one used to express the reflection $R_B$, need not be the same.
So we will extract the conclusions geometrically, in terms of algebraic data associated to the monodromy matrix, instead of referring to a specific basis.
Furthermore, the reflection $R_B$ alone does not determine the unipotent monodromy around the corresponding cusp, so this will lead to some further cases.

\subsubsection{Elliptic case}
	\label{sssec:elliptic_case_limit_point_schwarz}
It could be that the Hodge structure does not degenerate at all.
In this case, the line $\cW^{4,0}$ is a fixed point of the corresponding elliptic monodromy matrix and remains in $\LGr^{1,1}(V_{\bC})$.
It also determines a negative-definite two-plane, and these are the line and plane limits.

\subsubsection{Rank 1, invariant line}
	\label{sssec:rank_1_invariant_line}
This corresponds to the limit mixed Hodge structure and decomposition described in \autoref{eqn:nilpotent_orbit_rank1_line}.
With notation as in in loc.cit. it follows that the ``naive'' limit line is
\[
	\cW^{4,0}_{nv} = (e_2+\sqrt{-1}f_2)\wedge f_1
\]
but recall that this notation is not compatible with the basis used to express $R_B$.
Geometrically, $\cW^{4,0}_{nv}$ is the second exterior of a complex Lagrangian, which is not real, but which contains the line $l$ which is the image of the nilpotent $N$.
The plane limit can be computed similarly to be
\[
	P^2_{nv}=span(f_1\wedge f_2, f_1\wedge e_2)
\]
in the notation of loc.cit, which in this case agrees with $\left(\cW^{4,0}_{nv}\oplus \ov{\cW^{4,0}_{nv}}\right)\cap W_{\bR}$, and geometrically corresponds to $l\wedge l^{\perp}$, where $l$ is the image of the nilpotent $N$ and $l^{\perp}$ is its symplectic-orthogonal.
In other words, the limit is the photon associated to $l$, see \autoref{sssec:photons}.
Translating to the coordinates used for $W_{\bR}$, we find that the limit is
\[
	\left(a_0:\sqrt{-1}b_0:\sqrt{-1}c_0:\sqrt{-1}d_0\sqrt{-1}:e_0\right)
\]
with
\[
	a_0 e_0 = 0 \quad c_0^2-b_0d_0 = 0
\]
and $(a_0:e_0)\neq 0 \neq (c_0:b_0:d_0)$.

\subsubsection{Rank 1, invariant Lagrangian}
	\label{sssec:rank_1_invariant_lagrangian}
In this case, it follows from the decomposition used in \autoref{sssec:rank_1_invariant_lagrangian}, and the associated nilpotent orbit, that the naive limit $F^2_{nv}$ belongs to $\LGr^{1,1}(V_{\bC})$, i.e. there is no degeneration of this piece of the Hodge filtration.
Recall that this case contradicts assumption A, and in fact the developing map in \autoref{thm:real_uniformization_pseudo_hyperbolic} will be incomplete at this point.

\subsubsection{MUM case}
	\label{sssec:mum_case}
The corresponding decompositions $I^{p,q}$ of the limit mixed Hodge structure can be written by taking the third symmetric power of a weight $1$, dimension $2$, situation.
Let $l\subset L \subset l^{\perp}\subset V_{\bR}$ be the flag preserved by the MUM monodromy matrix, so $l$ is a line and $L$ is a Lagrangian.
Then a direction calculation gives that
\[
	\cW^{4,0}_{nv} = \Lambda^2 L \text{ and }P^2_{nv} = l\wedge l^{\perp}.
\]
In particular, note that the limit line does not determine the limit plane in this case.
In terms of the coordinates for $R_B$, it follows that
\[
	\gamma(t) \text{ limits to }(0:b_0:c_0:d_0:0)\text{ where }c_0^2-b_0d_0 = 0.
\]
The reasons is that the MUM Lagrangian will be preserved by $R_B$, for instance because the fixed curve of the reflection on the base will still have the same naive limit, so it must be that for the coordinates of $\cW^{4,0}_{nv}$, either the middle three, or the outer two, vanish.
However, the action of $R_B$ changes the orientation of the MUM Lagrangian $L$, since the composition with another reflection should yield a nontrivial unipotent transformation inside $L$, so it must be the case that the outer two coordinates of $L$ vanish.

On the other hand, the limit of
\[
	(a:0:0:0:e) \wedge (0:b:c:d:0) \text{ is } (a_0:0:0:0:e_0) \wedge (0:b_0:c_0:d_0:0)
\]
where $a_0e_0=0$ and $a_0+e_0=1$, and where $(b_0:c_0:d_0)$ is the same as the coordinates of $L$.
This isotropic $2$-plane corresponds to the photon $\phi(l)$ determined by the line $d_0 f_1 + c_0 f_2$ (or equivalently $c_0 f_1 + b_0 f_2$, if $d_0=0$ and using the defining relation).

\subsubsection{Fundamental polyhedron}
	\label{sssec:fundamental_polyhedron}
We can now put together the information from the analysis of the period map on the each of the segments on the real axis, together with \autoref{thm:real_uniformization_pseudo_hyperbolic}.
Specifically, let $X$ be the orbifold associated to one of the good hypergeometric parameters provided by \autoref{thm:good_hypergeometrics}.
It has an open set identified with $\bP^1(\bC)\setminus \{0,1\infty\}$, and possibly orbifold points of orders $M$ or $N$ at $0$, or $\infty$.
Let $\ov{X}$ denote the completed orbifold, so $\ov{X}\setminus X=:S$ consists of $1$, and possibly $0$ or $\infty$.
We can decompose $X$ into cells:
\[
	X = X^+ \coprod X^- \coprod X^{r} \coprod X^{orb}
\]
where $X^{\pm}$ correspond to the upper/lower half planes, $X^{r}$ corresponds to the part on the real axis and itself decomposes into three open segments $X^{r}_A,X^r_{B},X^{r}_C$ according to the three reflections from \autoref{sssec:matrices_of_anti_involutions}, and $X^{orb}$ corresponds to the orbifold points (there can be $0,1$ or $2$ such points).

Recall that \autoref{thm:real_uniformization_pseudo_hyperbolic} provides a developing map $\Dev^{P,+}$ from a $\bC^{\times}$-bundle over $\wtilde{X}$ to the pseudosphere $\bS^{1,3}(W_{\bR})$, equivariant for the action of the monodromy group $\Gamma$ generated by the three involutions $R_A,R_B,R_C$.
A direct consequence of theorem \autoref{thm:real_uniformization_pseudo_hyperbolic} and the analysis of the period map on the real axis is:

\begin{theorem}[Fundamental polyhedron]
	\label{thm:fundamental_polyhedron}
	Define $\cP^+$ to be the closure in $\bS^{1,3}(W_{\bR})$ of $\Dev^{P,+}(X^+)$.
	Then it has a cell decomposition, with
	\begin{description}
		\item[interior] the image of $\Dev^{P,+}(X^+)$, diffeomorphic to $\bR^{3}\times \bS^1$.
		\item[codimension 1] the image of $\Dev^{P,+}(X^r)$, consisting of three faces diffeomorphic to $\bR^{2}\times \bS^1$.
		\item[codimension 2] the image of the finitely many orbifold points $X^{orb}$, each diffeomorphic to $\bR^{1}\times \bS^1$.
	\end{description}
	Additionally, the closure of $\cP^{+}$ in $\bP^{+}(W_{\bR})$ consists of finitely many photons, corresponding to the cusps, i.e. points in $\ov{X}\setminus X$.
	The cusps are of two types, depending on the corresponding unipotent part of the monodromy matrix -- a rank $1$ unipotent (\autoref{sssec:rank_1_invariant_line}) or maximally unipotent monodromy (\autoref{sssec:mum_case}).

	The polyhedron $\cP^+$ is a fundamental domain for the action of the group $\Gamma=\ip{R_A,R_B,R_C}$ generated by the three reflections.
	Each reflection fixes one of the codimension $1$ faces as a set, and inside it has a fixed point locus of real dimension $1$.
\end{theorem}

\begin{remark}[On diffeomorphism types]
	\label{rmk:on_diffeomorphism_types}
	By the discussion in \autoref{sssec:causal_structure}, we have that $\bS^{1,3}(W_{\bR})$ is diffeomorphic to $\bS^{1}\times \bR^{3}$ (where the second factor is better viewed as $\bH^3$).
	In fact, all the spaces appearing in \autoref{thm:fundamental_polyhedron} have an $\bS^1$-factor.
	If we ignore it, and view the picture in $\bH^3$, then a cartoon structure of $\cP^+$ is as follows.
	Its interior if foliated by a real $2$-parameter family of geodesics.
	The codimension $1$ faces are $1$-parameter families of geodesics, and the codimension $2$ faces (coming from orbifold points) are single geodesics.
	The faces that occur ``at infinity'' yield just points in the boundary (recall that we ignore the $\bS^1$-factor).
	In this sense, the picture is reminiscent of a quasi-Fuchsian surface in $\bH^3$, of small curvature, and its family of normal geodesics, except that the surface is missing and only the geodesics remain (see Epstein \cite{Epstein1987_Univalence-criteria-and-surfaces-in-hyperbolic-space} for more on the hyperbolic version).
\end{remark}

\begin{proof}[Proof of \autoref{thm:fundamental_polyhedron}]
	By \autoref{thm:real_uniformization_pseudo_hyperbolic}, the developing map $\Dev^{P,+}$ is a diffeomorphism between a $\bC^{\times}$-bundle over $\wtilde{X}$ and $\bS^{1,3}(W_{\bR})$.
	The assertions about the decompositions of $\cP^+$ and $\bS^{1,3}$ follow from the analogous ones for the decomposition of $\wtilde{X}$ under the reflection group generated by the orientation-reversing hyperbolic isometries corresponding to the three open segments on the real axis.

	It remains to analyze the behavior of a single reflection acting on the $\bC^{\times}$-bundle over the real axis, since the reflection fixes pointwise the lift of the axis in the hyperbolic plane.
	Recall that if
	\[
		\gamma(t) = (a:\sqrt{-1}b:\sqrt{-1}c:\sqrt{-1}d:e)
	\]
	say normalized to $ae=1=c^2-bd$, then the fiber over the corresponding point consists of real unit positive-definite vectors that are orthogonal to $\gamma(t)$ (and hence to its complex-conjugate).
	Of these, precisely one is fixed by the reflection, namely $(-a:0:0:0:e)$ (see \autoref{eqn:exterior_power_basis_coordinates_eigenvalues}), so as $t$ moves along the real axis, the fixed points sweep out a real $1$-dimensional submanifold.
\end{proof}

%%					end of subsec: Polyhedra and reflections
%%=============================================================================

%%=============================================================================
%%					start of subsec: Explicit nonvanishing of the matrix coefficient

\laternote{Resurrect this section, with uniformization based at elliptic point}

\section{Elements of Lie theory and Anosov representations}
	\label{sec:elements_of_lie_theory_and_anosov_representations}

This section develops the necessary tools to study the global geometry of the monodromy of a VHS satisfying assumption A.
The relevant notion is that of Anosov representation, introduced by Labourie \cite{LabourieAnosov} and studied further by Guichard--Wienhard \cite{GuichardWienhard2012_Anosov-representations:-domains-of-discontinuity-and-applications}, as well as Kapovich--Leeb--Porti \cite{KapovichLeebPorti2018_Dynamics-on-flag-manifolds:-domains-of-proper-discontinuity-and-cocompactness} and many subsequent authors.

Our main examples of monodromy contain unipotents, so we need to extend the relevant concepts in Anosov representations allowing for this possibility.
We call the resulting representations ``log-Anosov'', based on the logarithmic divergence of certain sequences associated to the cusp behavior, as well as in analogy with the algebro-geometric context where logarithmic poles and more general logarithmic structures make the situation at cusps well-behaved.

In a first draft of this text we developed some of the basic notions related to log-Anosov representations, but these are no longer necessary with the appearance of a treatment of this case by Zhu \cite{Zhu2019_Relatively-dominated-representations} and Kapovich--Leeb \cite{KapovichLeeb2018_Relativizing-characterizations-of-Anosov-subgroups-I}, as well as Canary--Zhang--Zimmer \cite{CanaryZhangZimmer2021_Cusped-Hitchin-representations-and-Anosov-representations-of-geometrically-finite}.
Let us also emphasize that one can take an intrinsic approach, looking at representations with target a semisimple group $G$, or an extrinsic approach where $G$ is itself mapped via a linear representation.
Both points of view are useful, and it is particularly useful to pass between different representations.

To do so effectively, we introduce in \autoref{ssec:dominated_cocycles_and_representations} (log)-dominated cocycles.
These generalize the notions introduced by Bochi--Potrie--Sambarino \cite{BochiPotrieSambarino2019_Anosov-representations-and-dominated-splittings} to allow for a general poset structure in the domination bundles.
One (of many) motivations is that the notion of domination in cocycles is not stable under the constructions of linear algebra (direct sum, tensor product, etc.) and one looses information in the standard definition each time such an operation is performed.

We introduce another convenient concept, that of a stable point, in analogy with the notion in Geometric Invariant Theory, in \autoref{ssec:stability_and_the_numerical_criterion}.
The analogy was proposed by Kapovich--Leeb--Porti \cite{KapovichLeebPorti2018_Dynamics-on-flag-manifolds:-domains-of-proper-discontinuity-and-cocompactness}, where it was used in the context of actions on flag manifolds.
We use it to give a self-contained treatment of the proper discontinuity results we need later.
It is also likely that the same tools will have broader applicability.

Two further new results on Anosov representations might be of interest.
First, we classify the integral points on the limit set, showing that they have to come from cusps (see \autoref{thm:integral_vectors_must_come_from_cusps}).
Second, we show in \autoref{thm:minimality_of_the_action} that for an $\alpha_1$-log-Anosov representation containing a special kind of unipotent, the action on the limit set in the Lagrangian Grassmannian is minimal.
This is in contrast to Hitchin representations, where this is far from being true.

Finally, in \autoref{sec:unipotent_dynamics} we address a useful technical point, which is not needed in the main text.
Namely, with the correct choice of metric in the cusp, the dynamics under the geodesic flow will exhibit the correct exponential behavior, as if the space were compact.
We introduce there the notion of log-proximal element (an analogue of proximal semisimple elements), the monodromy weight filtration, which controls the growth, and the notion of admissible metric which ensures the correct asymptotic behavior.

%%=============================================================================
%%					start of subsec: Conventions and definitions

\subsection{Conventions and definitions}
	\label{ssec:conventions_and_definitions}

A concise summary of most of the structure theory of real Lie groups used below is contained in Knapp's \cite{Knapp1997_Structure-theory-of-semisimple-Lie-groups}.

\subsubsection{Setup}
	\label{sssec:setup_conventions_and_definitions}
Let $G$ denote the real points of a real linear semisimple algebraic group, with Lie algebra $\frakg$.
Pick a Cartan involution $\sigma\colon G\to G$, with fixed point set a maximal compact $K\subset G$.
Continue to denote the involution on the Lie algebra by $\sigma\colon \frakg \to \frakg$ and let $\frakg = \frakk \oplus \frakq$ denote the $\pm1$-eigenspace decomposition, with $\frakk = \Lie K$.

\subsubsection{Restricted root systems}
	\label{sssec:restricted_root_systems}

Pick a maximal abelian subalgebra $\fraka \subset \frakq$.
The adjoint action of $\fraka$ on $\frakg$ gives a weight space decomposition
\[
	\frakg = \frakg_0 \oplus \bigoplus_{\alpha\in \Phi}\frakg_\alpha
\]
where $\frakg_0$ is the centralizer of $\fraka$ and $\Phi \subset \fraka^\dual$ are the \emph{restricted roots}.
Note that $\frakg_0 = \fraka \oplus (\frakg_0\cap \frakk)$ since nothing in $\frakq$ can commute with $\fraka$.

Pick now a set of \emph{positive roots} $\Phi^+\subset \Phi$, i.e. $\Phi^+$ is contained in a half-space and $\Phi = \Phi^+ \coprod \Phi^-$ (with $-\Phi^+=\Phi^-$) and let $\Delta\subset \Phi^+$ denote the resulting set of \emph{simple roots}; the elements of $\Delta$ give a vector space basis for $\fraka$.
The \emph{positive Weyl chamber} $\fraka^+\subset \fraka$ is defined by
\[
	\fraka^+ = \left\lbrace v\in \fraka \colon \alpha(v)\geq 0,\, \forall \alpha \in \Delta \right\rbrace
\]
and the associated semigroup is denoted $A^+\subset G$.

Recall also that there is a Weyl-group invariant inner product, denoted $(x|y)$, on $\fraka^\dual$ and $\fraka$ which we will use to identify them for some constructions.
For the fixed choice of simple roots $\Delta$, there are associated \emph{fundamental weights} $\left\lbrace \omega_\alpha \right\rbrace_{\alpha\in \Delta}\subset \fraka^\dual$ defined by the relation
\[
	\frac{2}{(\beta|\beta)}(\beta|\omega_\alpha) = \delta_{\beta\alpha} \quad\forall \beta,\alpha\in \Delta
\]
Using the identification of $\fraka$ and $\fraka^\dual$ via the invariant inner product, the fundamental weights can also be viewed as generators of the positive Weyl chamber.

\subsubsection{The centralizer of $\fraka$ in $K$}
	\label{sssec:the_centralizer_of_a_in_k}
Consider the subgroup of elements in $K$ with trivial adjoint action on $\fraka$.
It will be denoted $Z_K(\fraka)$, and it is more commonly referred to in Lie theory as $M\subset K$.
Its Lie algebra is $\frakg_0\cap \frakk$ in the notation of \autoref{sssec:restricted_root_systems}.

\subsubsection{Parabolic subgroups}
	\label{sssec:parabolic_subgroups}
Fix a subset $\theta\subseteq \Delta$ of simple roots.
Our convention in \autoref{def:log_anosov_representation} is that the singular values of the Anosov representation will diverge away from the union of the faces $\left( \ker \alpha \right)\cap \fraka^+$ for $\alpha \in \theta$.

Set $\theta^c := \Delta \setminus \theta$.
Define $\Phi_{\theta^c}$ to be the intersection of $\Phi$ with the real span of the elements in ${\theta^c}$, and similarly the decomposition $\Phi_{\theta^c} = \Phi_{\theta^c}^+\coprod \Phi_{\theta^c}^-$.
With these conventions, define further:
\begin{itemize}
	\item The parabolic subalgebra: 
	\[
		\frakp_\theta^+:= \frakg_0 \bigoplus_{\alpha \in \Phi^+} \frakg_\alpha \bigoplus_{\alpha\in \Phi^-_{\theta^c}} \frakg_\alpha.
	\]
	Note that it is $\Phi^-_{\theta^c}$ that's used.
	\item The parabolic subgroup $P_\theta^{+}$ with $\Lie P_\theta^{+} = \frakp_\theta^{+}$
	\item The flag manifold $\cF_\theta^{+} := G/P_\theta^{+}$
	\item Analogously $\frakp_{\theta^c}^-,P_{\theta^c}^-$, and $\cF_{\theta}^-$
	\item The Levi subgroup and subalgebra
	\[
		L_{\theta} = P^{+}_{\theta} \cap P^{-}_{\theta} \quad
		\frakl_{\theta} = \frakp^{+}_{\theta} \cap \frakp^{-}_{\theta}=
		\frakg_0 \bigoplus_{\alpha\in \Phi_{\theta^c}}\frakg_{\alpha}.
	\]
	\item The maximal compact of the Levi:
	\[
		M_{\theta}:= K\cap L_{\theta}.
	\]
\end{itemize}
For example, when $\theta=\Delta$ the parabolic $P_\theta^+$ is the minimal one and the flag manifold $\cF_\theta^+$ is the maximal one, i.e. it parametrizes full flags.

\subsubsection{Partial order}
	\label{sssec:partial_order}
Continue with the fixed (nonempty) set $\theta\subset \Delta$ of simple roots.
Any other root $\lambda$ can be expressed as
\[
	\lambda = \sum_{\alpha\in\Delta} n_{\alpha}(\lambda)\alpha \quad \text{ with }n_{\alpha}(\lambda)\in \bZ.
\]
Define a partial order on the weight lattice by the requirement:
\begin{align*}
	\lambda'\succ_{\theta}\lambda \text{ if and only if: }\forall \alpha\in \Delta \text{ we have } &
	n_\alpha(\lambda')\geq n_{\alpha}(\lambda)\\
	\text{ and }\exists\alpha\in \theta \text{ such that }&n_{\alpha}(\lambda')> n_{\alpha}(\lambda).
\end{align*}
Note that the definition of the partial order naturally extends to any real $\lambda$.

\subsubsection{Polar or $KAK$ decomposition, Cartan projection}
	\label{sssec:polar_decomposition_cartan_projection}
Recall that any $g\in G$ can be expressed as
\[
	g = k_-(g) e^{\mu(g)}k_+(g) \quad k_\pm(g)\in K \text{ and } \mu(g)\in \fraka^+.
\]
While $\mu(g)$ is unique, the elements $k_\pm(g)$ are well-defined only up to a choice of element $m\in K$ which centralizes $e^{\mu(g)}$, by changing the elements to $k_-(g)m$ and $m^{-1}k_+(g)$.
Clearly any element in $Z_K(\fraka)$ can be used as $m$, but for some $e^{\mu(g)}$ the centralizer can be larger.
The map
\begin{align}
	\label{eqn:Cartan projection}
	\mu\colon G \to \fraka^+
\end{align}
is called the \emph{Cartan projection} of $G$.

\begin{definition}[$\theta$-divergence]
	\label{def:theta_divergence}
	Let $\theta\subset \Delta$ be a subset of the simple roots.
	\begin{enumerate}
		\item A sequence $v_i \in \fraka^+$ is called \emph{$\theta$-divergent} if
		\[
			\alpha(v_i) \to +\infty \quad \forall \alpha\in \theta.
		\]
		\item A sequence $\gamma_i\in G$ is \emph{$\theta$-divergent} if $\mu(\gamma_i)$ is.
		\item A subset $\Gamma\subset G$ is \emph{$\theta$-divergent} if every sequence in $\Gamma$ that leaves every compact set in $G$ is $\theta$-divergent.
	\end{enumerate}
	We will typically be interested in the situation when $\Gamma\subset G$ is a discrete subgroup.
	% Note that when $\theta\neq \emptyset$, a $\theta$-divergent set $\Gamma$ is discrete in $G$.
	% We will follow the convention that a $\emptyset$-divergent set is required to be discrete.
\end{definition}

\subsubsection{Opposition involution}
	\label{sssec:opposition_involution}
Recall (see \cite[Thm. 29.1]{Bump_Lie-groups} and discussion afterwards) that there exists $w_0\in K$, whose adjoint action preserves $\fraka$ and has the following property.
Given any $g\in G$, $w_0$ relates the $KAK$ decompositions of $g$ and $g^{-1}$ by
\[
	g = k_-e^\mu k_+ \Rightarrow g^{-1}= 
	k_+^{-1} e^{-\mu} k_-^{-1} = 
	(k_+^{-1} w_0^{-1}) e^{\Ad_{w_0}(-\mu)} (w_0 k_-^{-1})
\]
The element $w_0$ is well-defined only up to the left and right action of $Z_K(\fraka)$ .

Concretely, $\Ad_{w_0}(-\mu)\in \fraka^+$ if $\mu\in\fraka^+$.
Therefore $w_0$ gives the action of the longest element in the Weyl group and the map $\mu \mapsto \Ad_{w_0}(-\mu)$ is called the \emph{opposition involution} on $\fraka^+$ and $\fraka$.

\begin{example}[The root system of $\Sp_{2g}$]
	\label{eg:the_root_system_of_sp_2g}
	Let $V$ be a $2g$-dimensional real vector space with a non-degenerate symplectic form and set $G=\Sp(V), \frakg = \Lie G$.
	Choose an identification of $V$ with $\bR^{2g}$ such that the symplectic form is $\begin{bmatrix}
		0 & \id_{g} \\
		-\id_{g} & 0
	\end{bmatrix}$ and choose the maximal compact subgroup
	\[
		K = \left\lbrace \begin{bmatrix}
			A & -B \\
			B & A
		\end{bmatrix}\colon A+\sqrt{-1}B\in U(g) \right\rbrace
	\]
	to be the unitary group $U(g)=\left\lbrace X\in \Mat_{g\times g}(\bC) \colon  XX^\dagger = \id_g\right\rbrace$.
	The Cartan subalgebra can be chosen to consist of diagonal matrices
	\[
		\fraka = \left\lbrace v = \begin{bmatrix}
			D & 0 \\
			0 & -D \\
		\end{bmatrix}
		\colon D = \begin{bmatrix}
			\mu_1(v) &        & \\
			         & \ddots & \\
			         &        & \mu_g(v)
		\end{bmatrix}
			\right\rbrace
	\]

\begin{figure}[htbp!]
	\centering
	\includegraphics[width=0.32\linewidth]{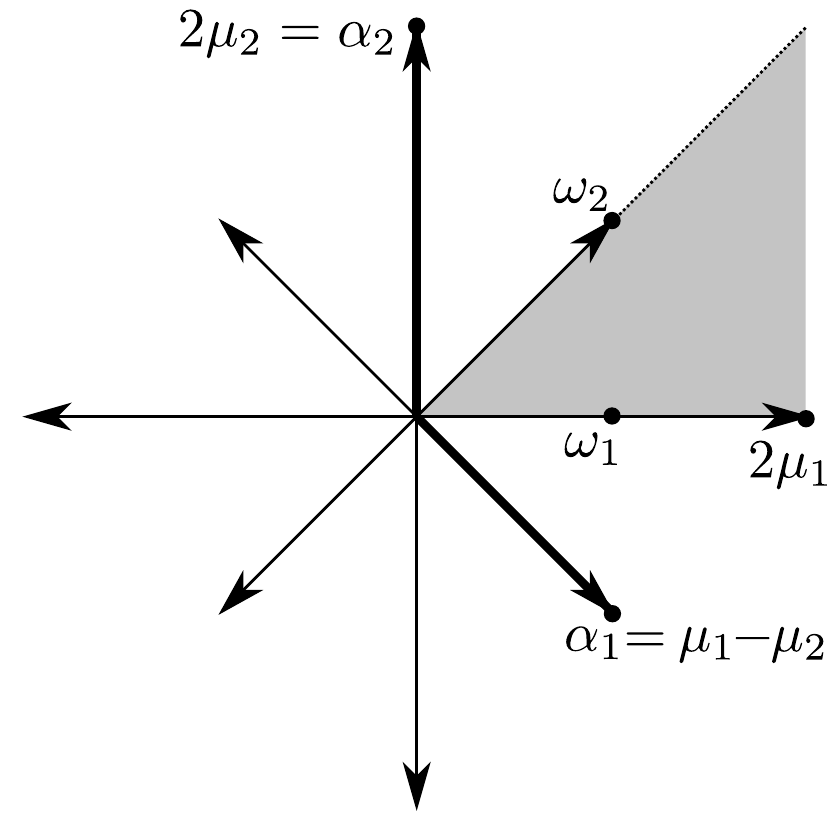}
	\includegraphics[width=0.32\linewidth]{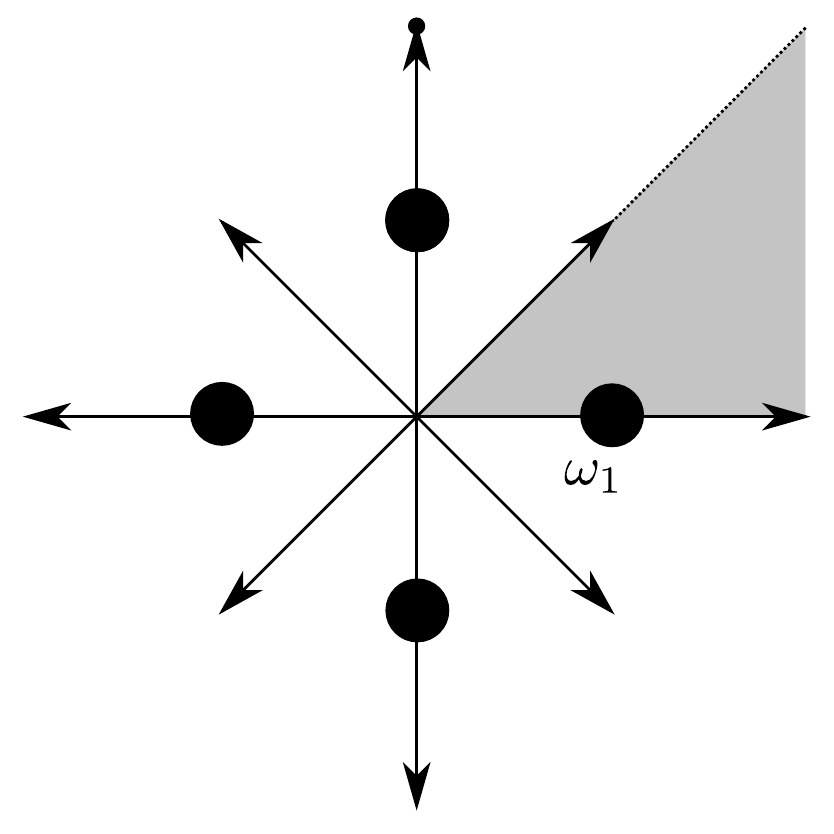}
	\includegraphics[width=0.32\linewidth]{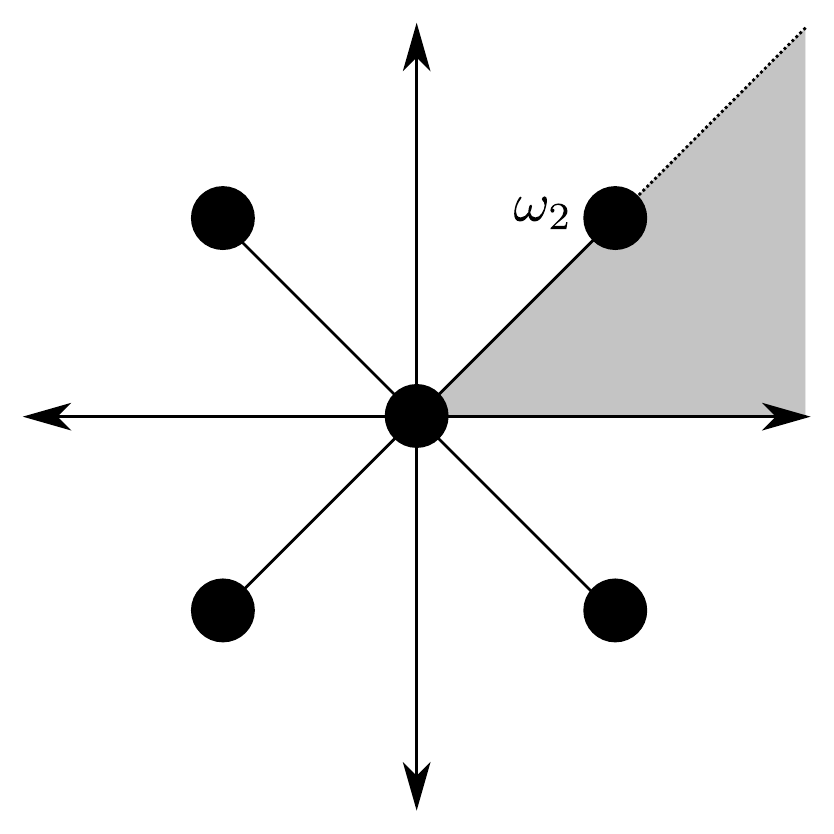}
	\captionsetup{width=0.9\linewidth}
	\caption{Left: The root system of $\Sp_4$, with its Weyl chamber.
	The fundamental weights are $\omega_1,\omega_2$.
	Middle: Weights for the standard representation.
	Right: Weights for the second exterior power of the standard representation.}
	\label{fig:sp4roots}
\end{figure}

	The roots are then
	\[
		\Phi=\left\lbrace \pm 2\mu_i \colon 1\leq i \leq g\right\rbrace \cup \left\lbrace \pm \mu_i \pm \mu_j \colon 1\leq i < j \leq g \right\rbrace
	\]
	and we can choose the simple roots
\begin{align}
	\label{eqn:simple_roots_sp2g}
	\Delta = \left\lbrace \alpha_1:=\mu_1-\mu_2,\ldots, \alpha_{g-1} = \mu_{g-1}-\mu_g, \alpha_g = 2\mu_g \right\rbrace
\end{align}
	so that the positive chamber is
	\[
		\fraka^+ = \left\lbrace v\in \fraka \colon \mu_1(v)\geq \mu_2(v)\geq \cdots \geq \mu_g(v)\geq 0 \right\rbrace
	\]
	Note that the opposition involution is trivial (equal to the identity) for the symplectic group.
	Equivalently, the longest element $w_0$ in the Weyl group is $-\id$.
\end{example}

\subsubsection{Notation and conventions for surfaces}
	\label{ssec:setup_notation_and_conventions_for_surfaces}
We now set some notation and conventions on surfaces that we follow in the text.
Throughout $X$ denotes a complete, finite volume, hyperbolic real $2$-dimensional orbifold.
Fix a basepoint $x_0\in X$ that is not in the orbifold part.
This gives a fundamental group $\pi_1(X,x_0)$ and a universal cover $\wtilde{X}$.
We will frequently write $\pi_1(X)$ instead of $\pi_1(X,x_0)$.
There is a canonical lift of the basepoint to the universal cover, which by abuse of notation will still be denoted $x_0$.
Fix an isometry of $\wtilde{X}$ with the hyperbolic plane $\bH$ in the upper half-plane model.

Most differential-geometric and coarse-geometric constructions are insensitive to passing to a finite-index subgroup.
In those situations, we can and do pass to a finite cover and assume that $X$ has no orbifold points and is orientable.
% On the other hand, for the criteria developed in \autoref{sec:cone_conditions_and_automatons}, it is useful to work with the largest possible group, which naturally leads to non-orientable orbifold surfaces.

\subsubsection{Standard cusp neighborhood}
	\label{sssec:standard_cusp_neighborhood}
Consider the region in the upper half-plane $\{\Im \tau > A\}$ and the quotient by $\tau \mapsto \tau+1$.
Parts of the surface $X$ which are isometric to such quotients will be called \emph{standard cusp neighborhoods}.
As $A\to +\infty$ the size of the standard cusp neighborhood shrinks.

In examples where the orbifold fundamental group is generated by reflections as in \autoref{sec:hypergeometric_local_systems}, standard cusp neighborhoods means quotients of $\{\Im \tau > A\}$ by the two reflections $\tau \mapsto -\ov{\tau}$ and $\tau \mapsto - \ov{\tau}+1$.

\subsubsection{Compact part of a surface}
	\label{sssec:compact_part_of_a_surface}
For each puncture of $X$ fix a standard cusp neighborhood and denote their union by $X_O$, an open set.
Denote the complement by $X_K$ and call it the \emph{compact part} of $X$.
We require that the cusp neighborhoods are disjoint and small enough that the inclusion $X_K\into X$ induces an isomorphism of fundamental groups.

\subsubsection{Geometry in the universal cover}
	\label{sssec:geometry_in_the_universal_cover}
Let $\partial \wtilde{X}$ denote the visual boundary of the universal cover.
It can be viewed as the set of equivalence classes of geodesic rays, with two rays equivalent when they are a bounded distance apart.
Equivalently, for a fixed basepoint it can be viewed as the set geodesic rays starting at the basepoint.

\begin{definition}[Opposite point]
	\label{def:opposite_point}
	For $x\in \wtilde{X}$ and $p\in \partial\wtilde{X}$, define $op_x(p)\in \partial \wtilde{X}$ to be the boundary point obtained by extending the geodesic ray from $x$ to $p$ in the opposite direction.
\end{definition}

\subsubsection{Two-point boundary}
	\label{sssec:two_point_boundary}
Consider the set of (ordered) distinct points on the boundary:
\[
	\partial \wtilde{X}^{(2)}:= \partial\wtilde{X}\times \partial\wtilde{X}
	\setminus \text{diag}(\partial\wtilde{X}).
\]
This is a noncompact space, but for every $x\in X$, the embedding
\[
	p\mapsto (p,op_{x}(p)) \quad \partial \wtilde{X}\into \partial\wtilde{X}^{(2)}
\]
has compact image.

\subsubsection{Fundamental domain}
	\label{sssec:fundamental_domain}
Fix a fundamental domain on $\wtilde{X}$ with piecewise geodesic boundary, say a Dirichlet domain for $x_0$.
Let $\wtilde{X}_{K}$ denote the lift of the compact part to the fundamental domain and $\wtilde{X}_{O}$ the lift of the complement, also inside the fundamental domain.

\subsubsection{Cusps, peripheral subgroups}
	\label{sssec:cusps_peripheral_subgroups}
Every cusp of $X$ determines a $\pi_1(X)$-invariant collection of points in $\partial\widetilde{X}$.
We call these \emph{parabolic points} and the stabilizer in $\pi_1(X)$ of a parabolic point will be called a \emph{peripheral subgroup}.
For instance, in the examples of \autoref{sec:hypergeometric_local_systems} peripheral subgroups are isomorphic to $(\bZ/2)*(\bZ/2)$ and generated by two reflections.
A \emph{parabolic subgroup} will refer to a subgroup of a peripheral subgroup that is isomorphic to $\bZ$.

\subsubsection{Unit tangent bundle and geodesic flow}
	\label{sssec:unit_tangent_bundle_and_geodesic_flow}
Let $T^1X, T^1\wtilde{X}$ denote the unit tangent bundles of $X$ and its universal cover.
After the identification of $\wtilde{X}$ with $\bH$, we have a identifications
\[
	T^1\wtilde{X} = \PSL_2(\bR) \text{ and } T^1X = \leftquot{\pi_1(X)}{\PSL_2(\bR)}.
\]
The geodesic flow $g_t$ is given by the right action of the matrix
\[
	g_t := \begin{bmatrix}
		e^t & 0 \\
		0 & e^{-t}
	\end{bmatrix}
\]
where we identify the matrix and the flow it induces when convenient.
We will also use the identification 
\begin{align}
	\label{eqn:united_tangent_bundle_trivialization}
	\wtilde{X}\times \partial\wtilde{X} \toisom T^1\wtilde{X}
\end{align}
which takes $(x,p)$ to the tangent vector at $x$ whose geodesic lands at $p$.

\subsubsection{Norm on fundamental group}
	\label{sssec:norm_on_fundamental_group}
Because we fixed an underlying complete hyperbolic metric and so a Fuchsian representation $\rho_{F}\colon \pi_1(X)\to \PSL_2(\bR)$, we can put a norm on the fundamental group via
\[
	\norm{\gamma}:=\norm{\rho_F(\gamma)} \text{ where the last expression is the matrix norm.}
\]
Any other Fuchsian representation would give comparable norms for our later purposes.
Note also that if $X$ were compact, it would suffice to use the word length to get a comparable quantity.
In geometric terms, we could have equivalently used $e^{\dist(x_0,\gamma x_0)}$, where $x_0$ is the fixed basepoint in $\wtilde{X}\isom \bH^2$, or rather $\cosh$ of the distance.

\subsubsection{Representations and flat bundles}
	\label{sssec:representations_and_flat_bundles}
Given a representation $\rho\colon \pi_1(X) \to \GL(V)$ denote by $\bV_\rho\to X$ the associated local system.
It is constructed as a quotient
\[
	\bV_\rho = \leftquot{\pi_1(X)}{\wtilde{X}\times V} \text{ with }\gamma\cdot (x,v)=(\gamma x, \rho(\gamma)v).
\]
We will abuse notation and write $\bV_{\rho}\to T^1X$ for the same local system lifted to the unit tangent bundle.
When convenient, the subscript $\rho$ in $\bV_\rho$ will be omitted.
The fiber of $\bV$ over $x\in X$, resp. $(x,p)\in T^1X$, will be denoted $\bV(x)$ resp. $\bV(x,p)$.

\subsubsection{Quasi-unipotency at the cusps}
	\label{sssec:quasi_unipotency_at_the_cusps}
All representations $\rho$ of $\pi_1(X)$ are assumed to have the following property.
For every peripheral subgroup, there is a finite-index cyclic subgroup generated by $\gamma$ and such that $\rho(\gamma)=T^{e}\cdot T^{u}$, where $T^e$ and $T^u$ commute, $T^{e}$ is an elliptic matrix (i.e. contained in a compact group) and $T^u$ is a unipotent matrix.

\subsubsection{Metrics on flat bundles}
	\label{sssec:metrics_on_flat_bundles}
For a local system $\bV_\rho$ coming from a representation, a metric on it is a choice of positive definite inner product on each fiber, that varies with appropriately specified regularity (continuous, smooth, real-analytic) on the point on the base.
Every metric on a vector bundle is assumed at least continuous, which suffices for most arguments; the Hodge metrics, to which the conclusions will be applied, are real-analytic.
We will only consider metrics that depend on the basepoint in $X$, not $T^1X$.

Fix now a continuous metric on $\bV_\rho$, such that $h(x)$ determines a metric on the vector space $\bV_\rho(x)$.
Lifting the discussion to the universal cover $\wtilde{X}$, this gives a $\pi_1(X)$-equivariant map
\begin{align}
	\label{eqn:metric_map_universal_cover}
	h\colon \widetilde{X}\to \GL(V)/K(x_0)
\end{align}
where $K(x_0)$ is the compact group associated to the metric at the basepoint $x_0$.
Using the trivialization of $\bV_{\rho}$ over $\wtilde{X}$ as $X\times V$, the norm of a vector $v\in V$ at $x\in \wtilde{X}$ is given by
\begin{align}
	\label{eqn:norm_of_vector_flat_bundle}
	\norm{v}_{x}:=\norm{h(x)^{-1}v}_{x_0}
\end{align}
where any $K(x_0)$-coset representative of $h(x)$ can be used.
With this convention, the norm function is left $\Gamma$-invariant, as it should be: $\norm{\rho(\gamma)v}_{\gamma x}=\norm{v}_{x}$ since $h$ is.

%%					end of subsec: Conventions and definitions
%%=============================================================================

%%=============================================================================
%%					start of subsec: Dominated cocycles and representations

\subsection{Dominated cocycles and representations}
	\label{ssec:dominated_cocycles_and_representations}

The notion of domination is a classical one in dynamical systems, generalizing that of hyperbolicity.
In this section we develop it further, allowing for a poset indexing the dominating bundles.
We refer to Bochi--Gourmelon \cite{BochiGourmelon2009_Some-characterizations-of-domination} and Bochi--Potrie--Sambarino \cite{BochiPotrieSambarino2019_Anosov-representations-and-dominated-splittings} for more on classical dominated representations and their relation to Anosov representations.

\subsubsection{Setup}
	\label{sssec:setup_dominated_cocycles_and_representations}
For the next definition, take $(M,g_t)$ to be a continuous $1$-parameter flow on a topological space; we will very soon specialize to $M=T^1X$ as above.
If $M$ is a metric space, we will say the flow is Lipschitz if $d(g_t x, g_t y)\leq C e^{|t|L}d(x,y)$ for any $x,y\in M$ and some fixed constants $C,L>0$.

A cocycle $E\to M$ over $g_t$ is a vector bundle with a lift of the action of $g_t$ to $E$ by fiberwise linear transformations.
Cocycles will be equipped with metrics, and if there is an underlying measure the Oseledets theorem applies (see \cite{sfilip_MET_lectures} for more background).
Most frequently the norm of $g_t$ acting on the fiber of $E$ is pointwise bounded by $e^{|t|L}$ for some $L>0$.
We will write the action of $g_t$ on the right.

\begin{definition}[Dominated cocycle]
	\label{def:dominated_cocycle}
	Let $\Theta$ be a partially ordered set (poset), with order relation $\succ$.
	A cocycle $E\to (M,g_t)$ has a (weak) $\Theta$-dominated splitting, or simply a \emph{(weak) dominated splitting}, if there exists a decomposition
	\[
		E  = \bigoplus_{A\in \Theta}E_A \text{ into continuous subcocycles}
	\]
	and for any compact set $K\subset M$ there exist $C=C(K),\ve=\ve(K)>0$, such that for any $A_1\succ A_2$ and $v_i\in E_{A_i}(m)\setminus 0$, with $m\in K$, we have that
	\begin{align}
		\label{eqn:domination_condition}
		C \cdot \frac{\norm{v_1 g_t}}{\norm{v_1}} \geq e^{\ve t}\cdot \frac{\norm{v_2 g_t}}{\norm{v_2}} \quad \forall t\geq 0.
	\end{align}
\end{definition}

\subsubsection{Domination in the compact part}
	\label{sssec:domination_in_the_compact_part}
In the above definition, if $M$ is compact then we take $K=M$ from the start.
Suppose now that $M$ is not compact.
Then we will say that the splitting of the cocycle is \emph{dominated in the compact part} if, with the notation as in \autoref{def:dominated_cocycle} the condition from \autoref{eqn:domination_condition} holds for all $t\geq 0$ \emph{and such that $mg_t\in K$}.

As we shall see in \autoref{thm:log_domination_in_the_cusps}, in the case of interest to us, namely $M=T^1X$, the condition of being dominated in the compact part implies being dominated in general, as long as the metric is adapted in the sense of \autoref{def:adapted_metric}.

\begin{definition}[Total domination]
	\label{def:total_domination}
	Let $\Theta$ be a poset.
	A subset $S\subset \Theta$ is called \emph{totally dominated} if, setting $S':=\Theta\setminus S$, it holds that $s'\succ s$ for any $s'\in S',s\in S$.
	Similarly, $S'\subset \Theta$ is called \emph{totally dominating} if $S:=\Theta\setminus S'$ is totally dominated.
	Given a $\Theta$-dominated cocycle $E$, set $E_S:=\oplus_{A\in S}E_A$, and call $E_S$ totally dominating, or dominated, if $S$ is.

	An element $X\in \Theta$ will be called a \emph{top} if $\{X\}$ is totally dominating, and a \emph{bottom} if it is totally dominated.
	A top and a bottom of $\Theta$ are unique if they exist.
	If there exists at least one totally dominating subcocycle, we will omit the adjective ``weak'' and just say ``dominated splitting''.
\end{definition}

The reader can consult \autoref{eg:adjoint_representation_of_sp_4} and \autoref{eg:reducible_representation_of_sl_2times_sl_2} for some examples illustrating the non-trivial poset structure that can arise among dominating bundles.

\begin{remark}[On total domination]
	\label{rmk:on_total_domination}
	In general, good properties of the subcocycles such as \Holder continuity, stability under perturbations, etc. will only hold for totally dominating or totally dominated ones.
	Indeed, a totally dominated subcocycle $E_S$ has a unique totally dominating complement $E_{S'}$, and then we are in the classical setting of domination.
	We will say that $E=E_{S'}\oplus E_{S}$ is a dominated splitting.

	Finally, let us remark that the dominated cocycles and posets that arise from Anosov representations have additional symmetries, coming from duality and time reversal, see for instance \autoref{eg:adjoint_representation_of_sp_4} below.
\end{remark}

\subsubsection{Domination and duality}
	\label{sssec:domination_and_duality}
Recall that to any cocycle $E$ we can associate the dual cocycle $E^{\dual}$, of dual vector spaces, with dynamics defined by
\[
	(\xi\cdot g_t)(v) := \xi(v\cdot g_{-t}) \quad \forall \xi\in E^{\dual}(x), v\in E(xg_{t}).
\]
When the cocycle comes from a representation of the fundamental group of $M$, this operation corresponds to taking the dual, or contragredient, representation.
Note that if $E$ is $\Theta$-dominated, then $E^{\dual}$ is $\Theta^{op}$-dominated, where $\Theta^{op}$ is the poset with order relation reversed.

\subsubsection{Domination and time reversal}
	\label{sssec:domination_and_time_reversal}
Suppose now that we consider instead of the flow $g_t$ the time-reversed flow $\tilde{g}_t:=g_{-t}$, but keep the cocycle the same.
Then the poset $\wtilde{\Theta}$ indexing the domination bundles becomes $\Theta^{op}$.

% \subsubsection{Domination and tensor operations}
% 	\label{sssec:domination_and_tensor_operations}
% Given several dominated bundles $E_i$ with posets $\Theta_i$, one can form the tensor product of bundles, and corresponding product of posets.
% Similarly, one can take direct sums and disjoint unions of posets (note that after such an operation one only gets weak domination).
% \fxnote{Either drop this discussion, or expand it to clarify what it means to take ``symmetric power'' of a poset, and also exterior power, other tensor things.
% Worse, what if it's symplectic or orthogonal, and there are other structures?
% See Stanley's book on enumerative combinatorics (vol 1), he has a discussion of posets there and some relevant constructions/notions}
% \fxnote{Refer to examples in the Anosov section}
% \fxnote{CHECK: for logA representations it seems that bundles of the type $E_{\geq A}$ and $E_{\leq A}$ are invariant on stables/unstables and so give Holder boundary maps}	

% \subsubsection{Stability under perturbations}
% 	\label{sssec:stability_under_perturbations}
% We have a decomposition
% \begin{align*}
% 	\End(E) & = \End^0(E)\oplus \End^{ss}(E) \oplus \End^{a}(E)\text{ where }\\
% 	\End^0(E) & = \bigoplus_{A\in \Theta} \End(E_A)\\
% 	\End^{ss}(E) & = \bigoplus_{A\succ A'} \Hom(E_A,E_{A'})\oplus \Hom(E_{A'},E_A)\\
% 	\End^a(E) & = \bigoplus_{A?A'} \Hom(E_A,E_{A'})\\
% \end{align*}
% where $A?A'$ means that $A,A'$ are not comparable in $\Theta$.

\laternote{Stability under perturbations should hold only for perturbations in $\End^0\oplus \End^{ss}$, see commented section above}

\subsubsection{Subbundles on $T^1X$ invariant under the geodesic flow}
	\label{sssec:invariant_subbundles_on_t1x}
Suppose that $\rho\colon \pi_1(X)\to \GL(V)$ is a linear representation and $\bV_\rho$ is the associated local system over $X$ or $T^1X$.
Recall that the unit tangent bundle $T^1\wtilde{X}$ can be identified with $\partial \wtilde{X}^{(2)}\times \bR$ such that the action of the geodesic flow is on the $\bR$-factor only, by translation.
A $g_t$-invariant decomposition of $\bV_{\rho}$ is then the same as a $\pi_1(X)$-equivariant map $\partial\wtilde{X}^2\to \cD_{\bbd}(V)$ where $\cD_{\bbd}(V)$ is the space of decompositions $V=\oplus V_{d_i}$ and $\bbd=\{d_i\}$.
Note that $\cD_{\bbd}(V)\isom \GL(V)/\prod \GL_{d_i}(\bR)$.

A bit more generally, one can replace $\GL(V)$ by a reductive group, or semisimple, group $G$ and $\cD_{\bbd}(V)$ by $G/L_{\theta}$, where $L_{\theta}$ is a Levi subgroup of a parabolic group, see \autoref{sssec:two_and_three_point_boundary_maps} below.

% \begin{proposition}[Relation to Lyapunov exponents]
% 	\label{prop:relation_to_lyapunov_exponents}
% 	Let $E\to (M,\mu,g_t)$ be a cocycle satisfying the assumptions of the Oseledets theorem and let $E=\oplus E^{\lambda_i}$ be the Lyapunov decomposition.
% 	\begin{enumerate}
% 		\item If $E=E_S\oplus E_{S'}$ is a dominated splitting
% 	\end{enumerate}
% \end{proposition}

%%					end of subsec: Dominated cocycles and representations
%%=============================================================================

%%=============================================================================
%%					start of subsec: log-Anosov representations

\subsection{log-Anosov representations}
	\label{ssec:log_anosov_representations}

\subsubsection{Setup}
	\label{sssec:setup_log_anosov_representations}
We will freely use the notation introduced in \autoref{ssec:conventions_and_definitions}, in particular $G$ is a real semisimple Lie group, $K$ a maximal compact, $\fraka$ a Cartan subalgebra, and $\theta \subseteq \Delta$ a fixed nonempty subset of the simple roots.

\begin{definition}[log-Anosov representation]
	\label{def:log_anosov_representation}
	A representation $\rho\colon \pi_1(X)\to G$ is \emph{$\theta$-log-Anosov} if there exist $C,\ve>0$ such that
	\[
		\alpha \Big( \mu(\rho(\gamma)) \Big) \geq \ve \cdot \log \norm{\gamma} - C \quad \forall \gamma\in \pi_1(X), \forall \alpha\in \theta
	\] 
	where $\mu\colon G\to {\fraka^+}$ is the Cartan projection and $\norm{\gamma}$ is the norm introduced in \autoref{sssec:norm_on_fundamental_group} using a(ny) Fuchsian representation.
\end{definition}

\subsubsection{Basic properties}
	\label{sssec:basic_properties}
Since the geodesic flow on $T^1X$ is conjugate to its time reversal, it follows that we can take $\theta$ to be invariant under the opposition involution, i.e. $\theta=\theta^{op}$, see \cite[Lemma~3.18]{GuichardWienhard2012_Anosov-representations:-domains-of-discontinuity-and-applications}.
This can also be deduced from the discussion in \autoref{sssec:domination_and_time_reversal}, as well as the formalism developed below for passing between log-dominated and log-Anosov representations.

The following is a consequence of results of Zhu \cite[Thm.~1.2]{Zhu2019_Relatively-dominated-representations} and can also be deduced from the results of Gu\'eritaud, Guichard, Kassel, and Wienhard \cite[Thm.~1.1]{GGKW_AnosovReps_GT} (though it doesn't apply directly in this setting, the mechanism is the same and can be deduced from Lemma~5.8 of loc.cit.)
Recall that $\theta=\theta^{op}$, so we can speak of the transversality of two $\theta$-parabolics.

\begin{theorem}[Existence of boundary maps]
	\label{thm:existence_of_boundary_maps}
	Suppose that $\rho\colon \pi_1(X)\to G$ is $\theta$-log-Anosov.
	Then there exists a continuous, $\rho$-equivariant map
	\[
		\xi\colon \partial \wtilde{X}\to \cF_{\theta}^{+}
	\]
	which is furthermore dynamics-preserving and satisfies transversality.
\end{theorem}
\noindent We will write from now on $\cF_{\theta}$ to refer to $\cF_{\theta}^+$.
Let us note that $\xi$ is \Holder-continuous, although we will not need this fact (see e.g. \cite[Thm.~A]{AraujoBufetovFilip2016_On-Holder-continuity-of-Oseledets-subspaces} for a statement which, when applied with uniform constants, implies the result).

\subsubsection{Transversality and dynamics}
	\label{sssec:transversality_and_dynamics}
In the statement of the above theorem, transversality means that such for two distinct points on the boundary $p_1\neq p_2$, the two parabolics $\xi(p_1),\xi(p_2)$ are transverse.
Being dynamics-preserving means the following.
If $\gamma\in \pi_1(X)$ corresponds to a closed geodesic with endpoints $p_{\pm}\in \partial\wtilde{X}$, then $\rho(\gamma)$ acts on $\cF_{\theta}$ with attracting/repelling fixed points at $\xi(p_{\pm})$, and vice versa for $\gamma^{-1}$.
Note that in general $\rho(\gamma)$ will have other fixed points as well, with intermediate (i.e. a mixture of attracting and repelling) behavior, but these are never in the image of $\xi$.

If $\gamma\in \pi_1(X)$ stabilizes a parabolic point $p$ on $\partial\wtilde{X}$ (and generates a cyclic subgroup), then $\xi(p)$ is a fixed point of the unipotent $\rho(\gamma)$ and furthermore the corresponding fixed flag is a subflag of the monodromy weight filtration as defined in \autoref{sssec:weight_filtration}.
Note that \cite{GGKW_AnosovReps_GT} does not directly address the unipotent case, but \cite[Def.~7.1]{Zhu2019_Relatively-dominated-representations} does.
The assertions in the parabolic case also follow from \autoref{prop:stable_spaces_near_the_cusp} below.

\subsubsection{Two and three-point boundary maps}
	\label{sssec:two_and_three_point_boundary_maps}
Because distinct parabolics in the image of the boundary map $\xi$ from \autoref{thm:existence_of_boundary_maps} are transverse, and because $\theta=\theta^{op}$ by assumption, we obtain an induced map from the two-point boundary to conjugacy classes of Levi subgroups:
\begin{align*}
	\xi^{(2)} &\colon \partial\wtilde{X}^{(2)}\to G/L_{\theta}\\
	\xi^{(2)}&(p_1,p_2):=\xi(p_1)\cap \xi(p_2)
\end{align*}
which is $\pi_1(X)$-equivariant.

Recall that $\partial\wtilde{X}^{(2)}$ is identified with the quotient of $T^1\wtilde{X}$ by the geodesic flow.
Consider the principal (right) $G$-bundle $\wtilde{\cG}\to T^1\wtilde{X}$, which descends to $\cG\to T^1X$.
Then the data of the map $\xi^{(2)}$ is the same as an equivariant section of $\wtilde{\cG}/L_{\theta}$, which therefore descends to a $g_t$-invariant section $\sigma_2$ of $\cG/L_{\theta}\to T^1X$.

Recall that $M_{\theta}:=K\cap L_{\theta}$ is a maximal compact in $L_{\theta}$, therefore $L_{\theta}/M_{\theta}$ is a contractible symmetric space.
Therefore, we can lift the section $\sigma$ to a continuous section $\sigma_3$ of $\cG/M_{\theta}\to T^1X$, since the fibers of $\cG/M_{\theta}\to \cG/L_{\theta}$ are contractible.

Note that we can make an arbitrary choice in the cusp, and then glue it by a partition of unity (using barycenters in the symmetric space) in the compact part.
In the cusp, it is most natural to use the ``strictly adapted metrics'' from \autoref{def:strictly_adapted_metric}, which in particular give homomorphisms $\SL_{2}(\bR)\to G$ compatible with the unipotent, and use the identification between $T^1\wtilde{X}$ and $\PSL_2(\bR)$.
However, we will only look at the behavior of $\sigma_3$ in the preimage of the compact part in $T^1\wtilde{X}$, so the choice in the cusp is immaterial for the arguments below.

We can now lift $\sigma_3$ to a map called $\xi^{(3)}$ on the universal cover and obtain a commutative diagram:
\begin{equation}
	\label{eqn:three_point_boundary_map}
\begin{tikzcd}
	T^1\wtilde{X}
	\arrow[r, "\xi^{(3)}"]
	\arrow[d, ""]
	& 
	G/M_{\theta}
	\arrow[d, ""]
	\\
	\partial\wtilde{X}^{(2)}
	\arrow[r, "\xi^{(2)}"]
	&
	G/L_{\theta}
\end{tikzcd}
\end{equation}

\subsubsection{Pullback and pushforward of the Anosov condition}
	\label{sssec:pullback_and_pushforward_of_the_anosov_condition}
It is quite immediate to describe the ``pullback'' behavior of the Anosov condition, see e.g. \cite[Prop.~4.1]{GuichardWienhard2012_Anosov-representations:-domains-of-discontinuity-and-applications} (note that our conventions are opposite of loc.cit. in that the roles of $\theta$ and $\theta^c$ are reversed).
Namely, suppose $\phi\colon G_1\to G_2$ is a homomorphism of real semisimple Lie groups and $\rho\colon \pi_1(X)\to G_1$ is a representation.
We will put the respective index ($1$ or $2$) on the Lie theoretic data associated to each group.
Suppose that $\phi\circ\rho$ is $\theta_2$-Anosov.
Then $\rho$ is $\theta_1$-Anosov, where $\alpha\in\theta_1$ if and only if $\alpha \notin \operatorname{span}_{\alpha_2\in \theta_2^c} \phi^t(\alpha_2)$, where we can arrange the choice of Lie theoretic data such that $\phi$ induces a map $\fraka_{1}\xrightarrow{\phi}\fraka_{2}$ and a reverse map $\fraka_2^{\dual}\xrightarrow{\phi^t}\fraka_1^{\dual}$.
This result is immediate for both Anosov, and log-Anosov representations.

The behavior under ``pushforward'', i.e. knowing that $\rho$ is $\theta_1$-Anosov, and concluding something analogous for $\phi\circ \rho$, requires a better control on how the positive Weyl chambers are mapped, see \cite[Prop.~4.6]{GuichardWienhard2012_Anosov-representations:-domains-of-discontinuity-and-applications}.
Instead of developing that aspect, we will see how the language of weakly dominated bundles allows us to get analogous information, even in situations where the Anosov conclusion might not hold on $G_2$, see e.g. \autoref{eg:reducible_representation_of_sl_2times_sl_2} and the example in \cite[\S4.3]{GuichardWienhard2012_Anosov-representations:-domains-of-discontinuity-and-applications}.

\subsubsection{Partial order on Levi-irreducible pieces}
	\label{sssec:partial_order_on_levi_irreducible_pieces}
Recall that $L_{\theta}$ denotes a Levi subgroup of $P_{\theta}^+$, defined in \autoref{sssec:parabolic_subgroups}, and $\frakl_{\theta}$ is its Lie algebra.
Let $V$ be a representation of $G$ (irreducibility is not needed here, but can be assumed).
Decompose $V$ according to irreducible $L_{\theta}$, or $\frakl_{\theta}$, representations:
\[
	V = \oplus V^{i,\theta}
\]
Each piece in turn decomposes further into weight subspaces for the Cartan subalgebra $\fraka$, with corresponding nonempty root subspaces $V_{\alpha}^{i,\theta}$ with $\alpha\in S^{i,\theta}\subset \fraka^{\dual}$.
Coarsening the partial order $\succ_{\theta}$ from \autoref{sssec:partial_order}, define a partial order on the set of $L_{\theta}$-subrepresentations of $V$ by:
\begin{align*}
	V^{i,\theta} \succ_{\theta}^{V} V^{j,\theta} \text{ if and only if }\alpha\succ_{\theta}\alpha'\quad  \forall \alpha\in S^{i,\theta},\forall \alpha'\in S^{j,\theta}.
\end{align*}
In other words, a representation $\theta$-dominates a second one if and only if any of its weights $\theta$-dominates any of the weights in the second one.

Note that this notion of domination is not absolute on $L_{\theta}$-representations, but rather depends on how they occur inside $V$.
Additionally, there could be several isotypic components of $\frakl_{\theta}$-representations (see e.g. the proof of \cite[Thm.~6.8]{Filip_Semisimplicity-and-rigidity-of-the-Kontsevich-Zorich-cocycle} for how to gather those together in an invariant way).
For economy of notation, we will omit the vector space of isotypic components.

\subsubsection{Metrics on cocycles}
	\label{sssec:metrics_on_cocycles}
Continuing with the linear representation $G\xrightarrow{\phi}\SL(V)$, and make the choices such that the maximal compact of $G$ maps inside a maximal compact in $\SL(V)$, and similarly for maximally split Cartan algebras.
Let $\bV\to T^1X$ be the cocycle associated to this linear representation.
We will equip $\bV$ with a metric which depends on the point in $T^1X$, not just $X$, using $\xi^{(3)}$.
Indeed the projection $G/M_{\theta}\to G/K$, composed with $\SL(V)/\SO(V)$, yields the needed metric.

Next, recall that a $G$-equivariant bundle on $G/L_{\theta}$ is the same as an $L_{\theta}$-representation.
Therefore, the $G$-equivariant vector bundle $\cV:=(G/L_{\theta})\times V$ has a $G$-equivariant decomposition $\cV=\oplus \cV^{i,\theta}$ according to the decomposition into $L_{\theta}$-irreducibles from \autoref{sssec:partial_order_on_levi_irreducible_pieces}.
By \autoref{sssec:two_and_three_point_boundary_maps}, the $\theta$-log-Anosov representation admits a $\pi_1(X)$-equivariant map
\[
	\xi^{(2)}\colon \partial \wtilde{X}^{(2)}\to G/L_{\theta}
\]
such that, applying the construction of \autoref{sssec:invariant_subbundles_on_t1x}, we obtain the $g_t$-invariant subbundles on $T^1X$.
These will be denoted $\bV^{i,\theta}$ and with indexing in bijection with the representations $V^{i,\theta}$.

\begin{theorem}[Domination in log-Anosov representations]
	\label{thm:domination_in_log_anosov_representations}
	Suppose that $\rho\colon \pi_1(X)\to G$ is $\theta$-log-dominated.
	Let $\phi\colon G\to \GL(V)$ be a representation of $V$.
	Define $\Theta(\phi)$ to be the set of representations of $\frakl_{\theta}$ occurring in $V$, with the order relation as defined \autoref{sssec:partial_order_on_levi_irreducible_pieces}.

	Then $\phi\circ\rho$ is $\Theta(\phi)$-log-dominated in the compact part, with bundles induced from the boundary map $\xi^{(3)}$ from \autoref{eqn:three_point_boundary_map}, and the $G$-equivariant bundles on $G/L_{\theta}$ via the construction of \autoref{sssec:invariant_subbundles_on_t1x}.
\end{theorem}
As will be proved in \autoref{thm:log_domination_in_the_cusps}, if the metric is adapted, in the sense of \autoref{def:adapted_metric}, the domination condition holds for all times $t\geq 0$ and all starting points in a compact set.

\begin{proof}
	We will use the map $\xi^{(3)}$ from \autoref{eqn:three_point_boundary_map} to check the domination condition, and the metric described in \autoref{sssec:metrics_on_cocycles}.
	First, recall that $\wtilde{X}_K\subset\wtilde{X}$ denotes a fundamental domain for the compact part $X_K\subset X$.
	We will verify the domination condition along geodesics of the form $(p,op_x(p))\in \partial \wtilde{X}^{(2)}$, for starting and returning points in the compact part of $X$, and where $x\in \wtilde{X}_K$, and our estimates will be uniform for such $x$.

	Consider now one such geodesic, say $(p,op_{x}(p))\in \partial\wtilde{X}^{(2)}$, and a point $y$ on it, projecting to the compact part of $X$.
	Let $u_x,u_y\in T^1\wtilde{X}$ be the corresponding unit vectors.
	By the construction of the map $\xi^{(3)}$, we have that $\xi^{(3)}(u_x)$ and $\xi^{(3)}(u_y)$, as elements of $G/M_{\theta}$, project to the same point in $G/L_{\theta}$, namely $\xi^{(2)}(p,op_{x}(p))$.
	It follows that there exist $g_x\in G$ and $l=l(u_x,u_y)\in L_{\theta}$ such that
	\[
		\xi^{(3)}(u_x)=g_x\cdot M_{\theta}
		\quad \text{ and }
		\xi^{(3)}(u_y) = g_x \cdot l \cdot M_{\theta} 
	\]
	and furthermore $g_x$ (it depends on $u_x$, not just $x$) can be picked in a compact set of $G$, depending only on the fixed compact part of $X$, and $\rho$ as well as the lift $\xi^{(3)}$.

	Recall that the flat vector bundle $\bV=T^1\wtilde{X}\times V$ decomposes into invariant subbundles according to the maps $\xi^{(2)}$, i.e. if $V=\oplus V^{i,\theta}$ is the decomposition into representations, then over $u\in T^{1}\wtilde{X}$ the decomposition is $\oplus \cV^{i,\theta}(u)=\xi^{(2)}(u)\cdot \oplus V^{i,\theta}$.
	To compute the norm of a vector $v\in V$ at $u$, we need to apply the $\xi^{(3)}$-map, namely $\norm{v}_{u}:=\norm{\xi^{(3)}(u)^{-1}v}_{u_0}$, where $u_0$ is a reference basepoint in $T^1\wtilde{X}$, sitting over the reference basepoint $x_0\in \wtilde{X}$, and such that $\xi^{(3)}(u_0)$ is the identity coset in $G/M_{\theta}$.

	Take now two $L_{\theta}$-subrepresentations $V^{1,\theta}\succ_{\theta}^{V}V^{2,\theta}$, and suppose $v_i \in \cV^{i,\theta}(u_x)$ are unit vectors.
	So $v_i = g_x v_{i,0}$ with $v_{i,0}\in V^{i,\theta}$, and such that $\norm{v_i}_{u_x}\approx \norm{v_{i,0}}_{u_0}$, i.e. the norms are comparable up to uniform constants since $x$ belongs to a fixed compact fundamental domain.

	We now want to compute the norms of the same vectors, but at $u_y$ (i.e. their parallel transport to $u_y$, which in the trivialization on the universal cover does not change the vectors).
	Then by definition we have
	\begin{align*}
		\norm{v_i}_{u_y} & 
		= \norm{\xi^{(3)}(u_y)^{-1}v_i}_{u_0}
		= \norm{\xi^{(3)}(u_y)^{-1} g_x v_{i,0}}_{u_0}\\
		& = \norm{(g_x\cdot  l\cdot M_{\theta})^{-1}g_x v_{i,0}}_{u_0}\\
		& = \norm{l\cdot v_{i,0}}_{u_0}
	\end{align*}
	where $l=l(x,y)\in L_{\theta}$.
	Let us now write out the KAK decomposition of $l$ in $L_{\theta}$ as
	\[
		l = m_- e^{\mu(l)} m_+ \quad \text{with }m_{\pm}\in M_{\theta}, \mu(l)\in \fraka^{+}.
	\]
	The vector $m_+ v_{i,0}$ can be arbitrary in $V^{i,\theta}$, but its norm is uniformly bounded above and away from $0$ by constants that only depend on the initial geometry.
	By assumption $V^{1,\theta}\succ_{\theta}^V V^{2,\theta}$, so any root of $V^{1,\theta}$ strictly $\theta$-dominates any root of $V^{2,\theta}$.
	Therefore, to establish that $\norm{v_1}_{u_y} \geq \tfrac{1}{C}e^{\ve \cdot d(x,y)} \norm{v_2}_{u_y}$, for some fixed $C,\ve>0$, it suffices to check that there exist constants $c_1,\ve>0$ such that
	\[
		\alpha(\mu(l)) \geq \ve\cdot d(x,y) - c_1 \quad \forall\alpha\in \theta.
	\]
	This, however, follows from \autoref{def:log_anosov_representation}, which implies this bound with $l$ replaced by $g\in G$ such that $\xi^{3}(u_y) = g\cdot \xi^{(3)}(u_x)$, and standard comparison estimates on the Cartan projection such as Fact~2.18(3) in \cite{GGKW_AnosovReps_GT}.
\end{proof}

The following is a consequence of the proof.
\begin{corollary}[Definitely contracting vectors]
	\label{cor:definitely_contracting_vectors}
	Consider the bundle
	\[
		\bV^{-}:=\bigoplus_{0\succ_{\theta}^V V^{i,\theta}} \bV^{i,\theta}
	\]
	where the summation is over all $L_{\theta}$-representations such that \emph{any} root in $V^{i,\theta}$ is $\theta$-dominated by the zero root.

	Then for any $v\in \bV^{-}(m)$ it holds that $\norm{vg_t}\to 0$ as $t\to +\infty$.
\end{corollary}

\subsubsection{An alternative proof of \autoref{thm:domination_in_log_anosov_representations}}
	\label{sssec:an_alternative_proof_of_thm:domination_in_log_anosov_representations}
It is also possible to deduce the theorem just proved, as well as the stated corollary, using slight modifications of \cite[Prop.~4.3]{GuichardWienhard2012_Anosov-representations:-domains-of-discontinuity-and-applications}.
Indeed the quoted result considers a decomposition of $V$ into a line and a hyperplane and establishes domination there.
The proof applies almost directly, except that one needs to use the quantity
\[
	\min_{\forall \beta_{x}\in V^{x,\theta}} (\beta_i-\beta_j)(\mu(u,t))
\]
where $\mu(u,t)$ denotes the Cartan projection that is being controlled, and $\beta_i,\beta_j$ are arbitrary roots in $V^{i,\theta},V^{j,\theta}$.
The condition $V^{i,\theta}\succ_\theta^V V^{j,\theta}$ implies that $\beta_i-\beta_j\succ 0$ (for the usual order on roots) and furthermore for at least one $\alpha\in \theta$ appears with a positive coefficient: $n_{\alpha}(\beta_i-\beta_j)>0$.

\begin{example}[Adjoint representation of $\Sp_4$]
	\label{eg:adjoint_representation_of_sp_4}
	We freely use the notation introduced in \autoref{eg:the_root_system_of_sp_2g}.
	The two simple roots are $\alpha_1=\mu_1-\mu_2$ and $\alpha_2=2\mu_2$.
	Assume that $\rho\colon \pi_1(X)\to \Sp_{4}(\bR)$ is $\theta:=\{\alpha_1\}$-log-dominated, which is the example of interest in view of \autoref{thm:anosov_property_from_assumption_A} below.

	The adjoint representation corresponds to the dominant weight $2\omega_1=2\mu_1$.
	The associated $g_t$-cocycle splits into six invariant subbundles, with domination relation as depicted in \autoref{fig:adjoint_rep_poset}.
	Note that the trivial root has multiplicity two, but it gets divided among two subbundles.
\end{example}

\begin{figure}[htbp!]
	\centering
	\includegraphics[width=0.29\linewidth]{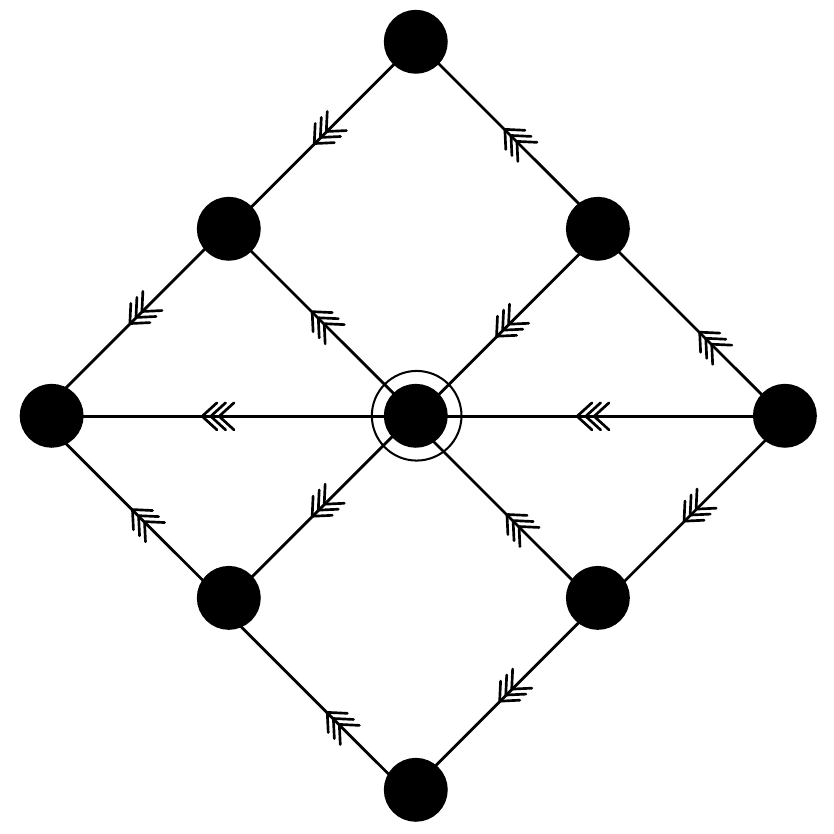}
	\includegraphics[width=0.67\linewidth]{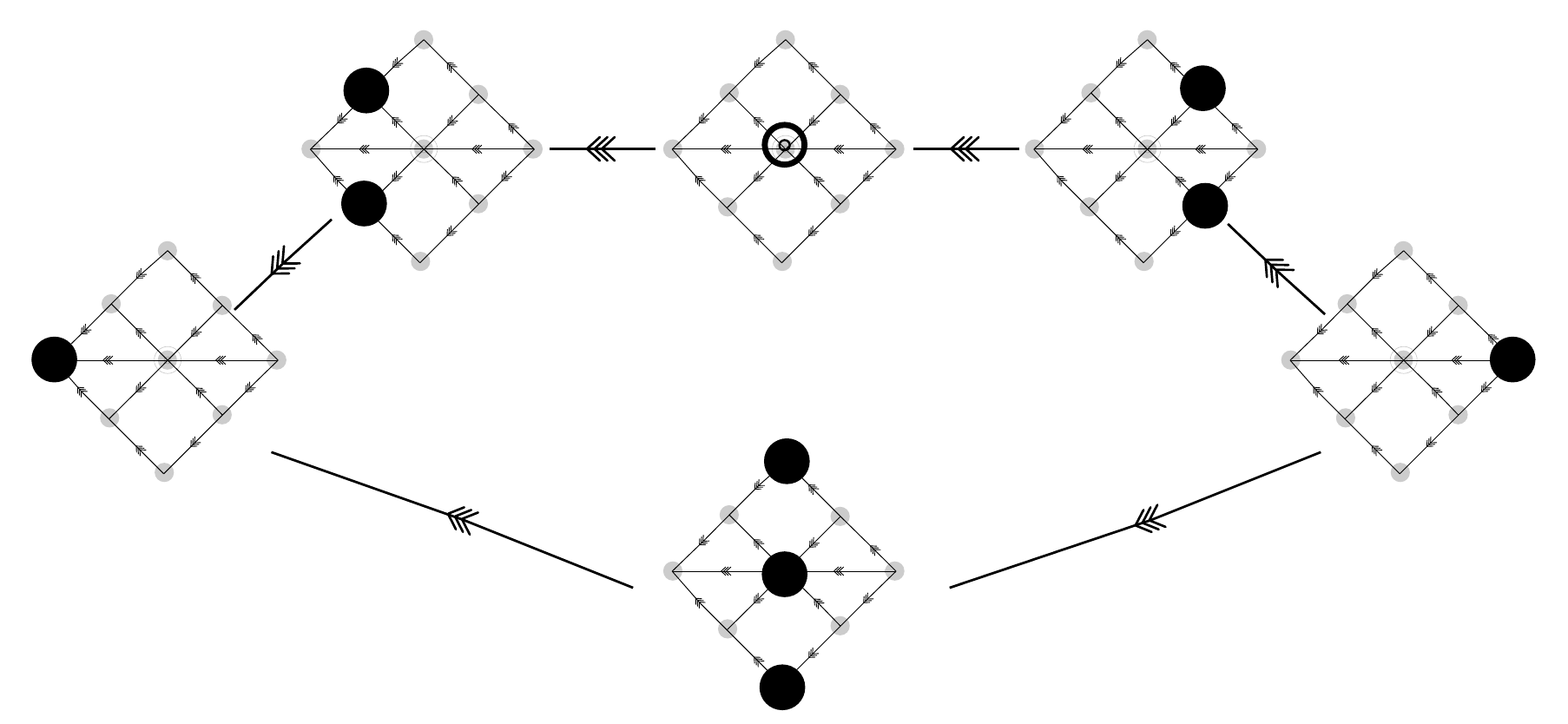}
	\caption{Left: the roots of the adjoint representation.
	Right: The poset of representations for $\frakl_{\alpha_1}$.}
	\label{fig:adjoint_rep_poset}
\end{figure}

\begin{example}[Reducible representation of $\SL_2\times \SL_2$]
	\label{eg:reducible_representation_of_sl_2times_sl_2}
	We take up an example of Guichard--Wienhard \cite[\S4.3]{GuichardWienhard2012_Anosov-representations:-domains-of-discontinuity-and-applications}.
	Namely, let $G:=\SL_2(\bR)\times \SL_2(\bR)$, and let $\rho_i\colon \pi_1(X)\to \SL_{2}(\bR)$ be such that $\rho_1$ is a Fuchsian representation, while $\rho_2$ is arbitrary (say, with dense image).
	Then $\rho:=(\rho_1,\rho_2)$ is $\alpha_1$-Anosov, where $\alpha_1$ is the simple root of the first factor.
	If we take the (reducible) representation which is the direct sum of the standard representations of each factor of $G$, then we obtain a weakly dominated representation.
	The poset has three bundles: two $1$-dimensional bundles, one of which dominates the other, coming from the first factor, and a $2$-dimensional bundle which has no domination relation with others coming from the second factor.
\end{example}

\begin{example}[Reducible representation of $\SL_2\times \SL_2$, bis]
	\label{eg:reducible_representation_of_sl_2times_sl_2_bis}
	We continue with the example \autoref{eg:reducible_representation_of_sl_2times_sl_2} above, but assume that one of the $\SL_2$-factors corresponds to a uniformizing Fuchsian representation, with entries in a ring of integers in a number field, while the second factor is the Galois-conjugate local system which also admits a variation of Hodge structure (of weight $1$).
	This situation occurs in \Teichmuller dynamics, see for instance \cite{McMullen2021_Billiards-heights-and-the-arithmetic-of-nonarithmetic-groups} or \cite{Moller2006_Variations-of-Hodge-structures-of-a-Teichmuller-curve}, as well as \cite{Filip_Semisimplicity-and-rigidity-of-the-Kontsevich-Zorich-cocycle} where the situation beyond $\SL_2$ is considered.

	In this case, the Schwarz lemma implies that the expansion in the first factor \emph{dominates} that in the second factor.
	Therefore, the poset structure of the bundles is such that the $1$-dimensional expanding subspace in the first factor dominates the second factor, which in turn dominates the contracting subspace in the first factor.
\end{example}

\subsubsection{Stable filtration}
	\label{sssec:stable_filtration_Ptheta}
The dominated decomposition constructed in \autoref{thm:domination_in_log_anosov_representations} only needs to $\rho$-equivariant map $\xi^{(2)}$ defined on $\partial \wtilde{X}^{(2)}$.
This, in turn, is constructed from the boundary map
\[
	\xi \colon \partial \wtilde{X}\to \cF_{\theta}=G/P_{\theta}
\]
from \autoref{thm:existence_of_boundary_maps}.
Over $\cF_{\theta}$, we have $G$-equivariant bundles corresponding to $P_{\theta}$-representations.
Therefore, for any $G$-representation $V$, we can consider the family of $P_{\theta}$-subrepresentations of $V$.
Pulling back to $T^1\wtilde{X}$ via the boundary map $\xi$, and passing to the finite volume quotient, gives subbundles of $\bV$ which are invariant not only under the geodesic flow $g_t$, but also the unipotent flow $u_s=\begin{bmatrix}
	1 & s\\
	0 & 1
\end{bmatrix}$ since $T^1\wtilde{X}\isom \PSL_2(\bR)$, and the quotient by the parabolic group generated by $g_t,u_s$ is $\partial \wtilde{X}$.

Note that any $P_{\theta}$-invariant representation is a direct sum of irreducible $L_{\theta}$-subrepresentations, so the new bundles are direct sums of the bundles arising in the dominated decomposition.
Note that in general, the resulting bundles, which we shall call the ``stable filtration'', do not constitute a filtration in the sense that given two distinct ones, it need not be the case that one is strictly contained in the other (see for instance \autoref{eg:adjoint_representation_of_sp_4}).

\subsubsection{Dangerous turn on terminology}
	\label{sssec:dangerous_turn_on_terminology}
We would like to warn the reader about the following inevitable clash of terminology from dynamics and GIT.
In dynamics, it is typical to call a vector, subbundle, etc. ``stable'' if it decays exponentially fast into the future (or a bit more generally, if the future behavior of the object is atypical).
In GIT (see \autoref{ssec:stability_and_the_numerical_criterion} below) the term ``stable'' is by contrast used for vectors that grow (exponentially) into the ``future'', for all possibly futures (so the behavior of vectors is typical).
We hope that this clash of terminology will not confuse the reader and at any point in the arguments, it will be apparent which context we're in.

%%					end of subsec: log-Anosov representations
%%=============================================================================

%%=============================================================================
%%					start of subsec: Stability and the numerical criterion

\subsection{Stability and the numerical criterion}
	\label{ssec:stability_and_the_numerical_criterion}

The key notion introduced here is that of a ``stable point'' for a closed group $\Gamma\subset \SL(V)$ acting on $\bP(V)$.
It allows us to develop the standard proper discontinuity criteria in an economical way.
To streamline the arguments, we consider the limit set of $\Gamma$ using the KAK decomposition and view it \emph{inside} the maximal compact $K$.
This has the advantage that the limit set ``acts'' on various spaces on which $G$ or $K$ act.
This is readily translated to the usual limit sets via the equalities $K/M=KAN/MAN=G/P$.

\subsubsection{Setup}
	\label{sssec:setup_stability_and_the_numerical_criterion}
Let $\Gamma\subset G$ be a closed subgroup.
We introduce a notion of limit set for $\Gamma$ which is convenient for our later purposes.
Many variations exist, see e.g. \cite{Benoist1997_Proprietes-asymptotiques-des-groupes-lineaires}.
We then relate this notion of limit set to a criterion for proper discontinuity of the $\Gamma$-action.

Recall that $\mu\colon G\to \fraka^{+}$ denotes the Cartan projection.
Denote by $\bP\fraka^+$ the set of rays in the Weyl chamber, its elements will be denoted by $[\mu]$ and a nonzero lift to $\fraka^+$ will be denoted $\mu$.
Recall also that any $g\in G$ has a KAK decomposition $g=k_-(g)e^{\mu(g)}k_{+}(g)$.

\begin{definition}[Limit data]
	\label{def:limit_data}
	The limit data $\cL(\Gamma)$ of $\Gamma$ is the set of accumulation points in $K\times K\times \bP\fraka^+$ of sequences
	\[
		(k_-(\gamma_i),k_{+}(\gamma_i),[\mu(\gamma_i)]) \text{ where }\gamma_i\in \Gamma_i \text{ is s.t. }\norm{\mu(\gamma_i)}\to+\infty.
	\]
	Denote also by $\cL_+(\Gamma)$ the projection to $K\times \fraka^+$ along the first factor of $\cL(\Gamma)$.
	A \emph{limit ray} is a one-parameter semigroup of the form $k_+^{-1}e^{\mu t} k_+$ for $t\geq 0$, where $(k_+,[\mu])\in \cL_+(\Gamma)$.
\end{definition}

\begin{remark}[Basics on limit data]
	\label{rmk:basics_on_limit_data}
	It is clear from the definition that $\cL(\Gamma)$ is a closed set, hence compact because the ambient space is compact (recall that $\fraka^+$ is the closed Weyl chamber).
	The KAK decomposition is not unique in general, and we allow in the definition all possible representations.
	In particular, the limit data is invariant under the action of $M$ by $(k_-,k_+,[\mu])\mapsto (k_- m^{-1},mk_+,[\mu])$.

	Since $\Gamma$ is a group, we can take inverses and this provides an extra symmetry of the limit data.
	Namely, let $w_0\in K$ be an element whose adjoint action gives on $\fraka$ the action of the longest element of the Weyl group, so that
	\begin{align*}
		\left[k_-(\gamma)e^{\mu(\gamma)}k_+(\gamma)\right]^{-1} & = 
		k_+(\gamma)^{-1} e^{-\mu(\gamma)} k_-(\gamma)^{-1}\\
		& = \left[k_+(\gamma)^{-1} w_0^{-1} \right]\cdot
		\left[e^{\mu(\gamma)^{op}}\right] 
		\cdot
		\left[w_0k_-(\gamma)^{-1}\right]\\
	\end{align*}
	where $\mu^{op}:=Ad_{w_0}(-\mu)$.
	It follows that the limit data is symmetric under
	\begin{align}
		\label{eqn:taking_inverses_symmetry}
		(k_-,k_+,[\mu])\mapsto \left(k_+^{-1}w_0^{-1}, w_0k_-^{-1},[\mu^{op}]\right).
	\end{align}
\end{remark}

We now reverse the Hilbert--Mumford criterion to provide a definition of stability:
\begin{definition}[Stable point]
	\label{def:stable_point}
	Let $V$ be a $G$-representation equipped with a $K$-invariant norm.
	A point $[v]\in \bP(V)$ is $\Gamma$-stable, or simply \emph{stable} if $\Gamma$ is clear from the context, if for any $(k_+,[\mu])\in \cL_{+}(\Gamma)$, and for any lift $v\in V$, we have along the corresponding limit ray:
	\[
		\norm{k_+^{-1}\cdot e^{\mu\cdot t}\cdot k_+ v}\xrightarrow{t\to +\infty}+\infty.
	\]
\end{definition}
The element $k_+^{-1}$ in the above definition does not affect the norm of the vector.
We explain below in \autoref{sssec:intrinsic_aspects_of_stability} that the property of point being stable is independent of the choice of $K$ and $\fraka$.

\subsubsection{Stable filtration}
	\label{sssec:stable_filtration}
Fix $[\mu]\in \bP\fraka^+$ and a $G$-representation $V$.
The weight space decomposition of $V$ for the action of $\fraka$ induces a coarse decomposition
\[
	V = 
	V^{<0}([\mu])\oplus
	V^{0}([\mu])\oplus
	V^{>0}([\mu])
\]
where the superscript indicates the result of pairing the subspace with any lift of $\mu$ to $\fraka^+$.
We will also write $V^{\leq 0}([\mu])$ for the sum of the first two, and $V^{\geq 0}([\mu])$ for the sum of the last two.
It is clear that
\begin{align}
	\label{eqn:stability_criterion_filtration}
	v\notin V^{\leq 0}([\mu]) \iff \norm{e^{\mu\cdot t}v} \xrightarrow{t\to +\infty}+\infty.
\end{align}

We also have an upper semicontinuity property of $V^{\leq 0}([\mu])$, namely if $[\mu_i]\to [\mu]$ in $\bP\fraka^+$ and $[v_i]\in \bP \left(V^{\leq 0}([\mu_i])\right)$, with $[v_i]\to [v]$, then it holds that $[v]\in \bP\left(V^{\leq 0}([\mu])\right)$.
Note incidentally that the spaces $V^{\bullet}([\mu])$ can take only finitely many values as $[\mu]$ ranges in $\bP\fraka^+$.

Finally, let $w_0\in K$ be again the longest element in the Weyl group.
Because of the calculation (where $\alpha$ is a weight):
\[
	\ip{\alpha,\mu} = \ip{w_0\alpha,w_0 \mu} = -\ip{w_0\alpha,-w_0\mu}
	=-\ip{w_0\alpha,\mu^{op}}
\]
it follows that
\begin{align}
	\label{eqn:longest_element_decomposition_action}
		w_0\cdot  V^{<0}\left([\mu]\right) = V^{>0}\left([\mu^{op}]\right) \text{ and }
	w_0\cdot  V^{0}\left([\mu]\right)=
	V^{0}\left([\mu^{op}]\right).
\end{align}

\subsubsection{Dynamically related points}
	\label{sssec:dynamically_related_points}
To establish proper discontinuity of actions, the following notion due to Frances \cite[\S3.2]{Frances2005_Lorentzian-Kleinian-groups} is convenient.
Let $S$ be a topological space and $\Gamma\subset \Homeo(S)$ a group.
Two points $x,y\in S$ are \emph{dynamically related} by a sequence $\left\lbrace \gamma_i \right\rbrace\subset \Gamma$ if there exists a sequence $x_i\to x$ such that $\gamma_i(x_i)\to y$.
We will say that two points are $\Gamma$-dynamically related (for $\Gamma\subset \Homeo(S)$ closed) if they are related by a sequence contained in $\Gamma$, and which leaves every compact set in $\Gamma$.
Note that being dynamically related is a reflexive relation, since setting $y_i:=\gamma_i x_i \to y$ one sees that $\gamma_i^{-1}y_i=x_i\to x$ holds.

The main application of the definition of stability is to establish:
\begin{theorem}[Proper discontinuity criterion]
	\label{thm:proper_discontinuity_criterion}
	Keep the notation from above.
	\begin{enumerate}
		\item A point $[v]\in \bP(V)$ is stable if and only if for any $(k_+,[\mu])\in \cL_+(\Gamma)$ we have that
		\[
			k_+\cdot [v]\notin V^{\leq 0}([\mu]).
		\]
		Therefore, the set of stable points in $\bP(V)$ is open (but possibly empty).
		\item Two distinct stable points cannot be dynamically related.
		\item If $M\subset \bP(V)$ is a locally closed $\Gamma$-invariant set, and $M^s\subset M$ denotes its intersection with the stable points, then $\Gamma$ acts properly discontinuously on $M^s$ when it is equipped with the subspace topology.
	\end{enumerate}
\end{theorem}
\begin{proof}
	The criterion of stability in part (i) is just a restatement of \autoref{eqn:stability_criterion_filtration} with the inclusion of a $k_+$-term.
	To check that the set of stable points is open is equivalent to establishing that its complement is closed.
	Inside $K\times \bP\fraka^+\times \bP(V)$, consider the set of points of the form $(k,[\mu],[v])$ such that $(k,[\mu])\in \cL_+(\Gamma)$ and $k\cdot [v]\in V^{\leq 0}([\mu])$.
	This set it closed (and compact) because $\cL_+(\Gamma)$ is, and because of the upper semicontinuity property of the stable filtration discussed in \autoref{sssec:stable_filtration}.
	The complement of the set of stable points is precisely the image of the constructed set under projection to $\bP(V)$.

	Clearly part (iii), proper discontinuity, follows from (ii).
	So assume, by contradiction, that two dynamically related points exist (see \autoref{sssec:dynamically_related_points}), take corresponding sequences, and denote by $x_i,x,y$ a choice of lifts to $V$.
	Write the KAK decomposition of $\gamma_i=k_-(\gamma_i)e^{\mu(\gamma_i)}k_+(\gamma_i)$ and pass to a subsequence such that there is a limit $(k_-,k_+,[\mu])\in \cL(\Gamma)$.
	From the assumed limit properties, it follows that
	\[
		k_+(\gamma_i)x_i \xrightarrow{i\to \infty} k_+x
		\quad\text{ and }\quad
		k_{-}^{-1}(\gamma_i)y \xrightarrow{i\to \infty} k_-^{-1}y.
	\]
	Because $x$ is stable by assumption, and using the criterion from \autoref{eqn:stability_criterion_filtration}, it follows that $k_+ x\notin V^{\leq 0}([\mu])$, and so there exists $\ve>0$ such that $d(k_+[x],\bP(V^{\leq 0}[\mu]))>\ve$ for the induced distance on projective space.
	It follows that for sufficiently large $i$, we have that $d(k_+(\gamma_i)x_i,V^{\leq 0}([\mu_i]))>\ve/2$, otherwise we could pass along a subsequence to a limit and contradict the lower bound of $\ve$ on $k_+[x]$ just obtained (again, combined with the upper semicontinuity of $V^{\leq 0}$).

	Because $y$ is stable, and using the symmetry of $\cD(\Gamma)$ under taking inverses from \autoref{eqn:taking_inverses_symmetry}, it follows that $w_0 k_-^{-1}y\notin V^{\leq 0}([\mu^{op}])$.
	Now applying $w_0$ to both sides and recalling that it is an involution, combined with the formula for its action from \autoref{eqn:longest_element_decomposition_action}, it follows that $k_-^{-1}y\notin V^{\geq 0}([\mu])$.

	From the assumption that $[x],[y]$ are dynamically related, and using the KAK decomposition of $\gamma_i$, we know that
	\[
		[z_i]:=e^{\mu(\gamma_i)}k_+(\gamma_i)[x_i] \to k_-^{-1} [y]\notin V^{\geq 0}([\mu]).
	\]
	However, we established that there is a uniform $\ve/2>0$ lower bound distance between $k_+(\gamma_i)[x_i]$ and $V^{\leq 0}([\mu_i])$.
	It follows that the distance between $[z_i]$ and $V^{>0}([\mu_i])$ goes to zero, and by the upper semicontinuity of $V^{\geq 0}$, it follows that any of the accumulation points of $[z_i]$ must be in $V^{\geq 0}([\mu])$.
	This however contradicts the fact that $k_{-}^{-1}[y]\notin V^{\geq 0}([\mu])$.
\end{proof}

\subsubsection{Semistable points}
	\label{sssec:semistable_points}
It is natural to define $[v]\in \bP(V)$ to be \emph{semistable} if the condition in \autoref{def:stable_point} is replaced by:
\[
	\liminf \norm{k_+^{-1}\cdot e^{\mu\cdot t}\cdot k_+ v}>0 \text{ as }{t\to +\infty}.
\]
Then \autoref{eqn:stability_criterion_filtration} gets replaced by
\[
	v\notin V^{< 0}([\mu]) \iff \norm{e^{\mu\cdot t}v} \text{ stays bounded as }{t\to +\infty}.
\]
The proof of \autoref{thm:proper_discontinuity_criterion} then goes through verbatim, only changing some of the signs $\leq$ by $<$ to yield that a semistable point $[x]$ and a stable point $[y]$ cannot be dynamically related (recall that being dynamically related is also a reflexive relation).
Indeed the last part in the proof changes to the assertion that the distance between $[z_i]$ and $V^{\geq 0}([\mu_i])$ goes to zero, and so their accumulation points must still be in $V^{\geq 0}[\mu]$, which contradicts the assumption on $[y]$.
Note that strictly semistable points can be dynamically related, even in the example of a single element of $\SL_3(\bR)$ with eigenvalues $e^\mu,1,e^{-\mu}$ acting on $\bP(\bR^3)$.

For how to take these notions further in the context of actions on flag manifolds, see \cite[\S7.4]{KapovichLeebPorti2018_Dynamics-on-flag-manifolds:-domains-of-proper-discontinuity-and-cocompactness}.

\begin{remark}[On the numerical criterion]
	\label{rmk:on_the_numerical_criterion}
	The Hilbert--Mumford numerical criterion \cite[Thm~2.1]{MumfordFogartyKirwan_Geometric-invariant-theory} is the basis for our definition of stability and semistability.
	Let us remark that a crucial role in the proofs is played by the $KAK$ (or polar, or Iwasawa) decomposition for a group $G$, where $K=\bbG(k\doublebracket{t})$ and $G=\bbG(k\doubleparen{t})$ for a field $k$, formal variable $t$, and reductive $k$-algebraic group $\bbG$.
	One could replace further $k\doublebracket{t}$ by a discrete valuation ring and $k\doubleparen{t}$ by its field of fractions (so, a local field would work).	
\end{remark}

\subsubsection{Intrinsic aspects of stability}
	\label{sssec:intrinsic_aspects_of_stability}
The notion of stable point in \autoref{def:stable_point}, and semistable point in \autoref{sssec:semistable_points}, appears to depend on a choice of maximal compact $K$ and Cartan subalgebra $\fraka$.
We now clarify that this is not the case, and any other choices would lead to the same notion.

Let $\crX$ denote the space of maximal compact subgroups of $G$ -- when we fix one such, say $K$, we get an identification $\crX\isom \leftquot{K}{G}$ (we will take this left quotient convention for this paragraph only).
The space $\crX$ carries a unique up to scale $G$-invariant Riemannian metric, which makes it a $\CAT(0)$ space.
As such it has a natural visual boundary $\partial_{vis}\crX$, and $\bP\fraka^+$ is identified with the space of $K$-orbits on $\partial_{vis}\crX$, for any choice of maximal compact $K$.

Next, the $1$-parameter semigroups used in \autoref{def:stable_point} determine geodesic rays starting at a fixed basepoint of $\crX$, obtained as limits of geodesic rays connecting the basepoint with its images under $\Gamma$.
In other words, the semigroups used in \autoref{def:stable_point} are those determined by the closure of $\Gamma$ in the visual compactification of $\crX$.
By the $\CAT(0)$ property of $\crX$, any other basepoint would lead to geodesic rays which are a bounded distance away from those based at the original basepoint.
This would affect the norms by bounded multiplicative constants, so would not affect the definitions of stability (or semistability).

\subsubsection{Control of the limit data}
	\label{sssec:control_of_the_limit_data}
We now collect some consequences of Anosov-type conditions on the limit data of $\Gamma:=\rho(\pi_1(X))$, where $\rho$ is $\theta$-(log)-Anosov.
First, it is immediate from the definition \autoref{def:log_anosov_representation} that the projection of $\cL_+(\Gamma)$ to $\fraka^+$ is contained in the \emph{complement} of the faces $\ker\alpha,\forall \alpha\in \theta$.

Control on the $K$-component comes from \autoref{thm:existence_of_boundary_maps} and its consequence, \autoref{thm:domination_in_log_anosov_representations}.
Freeze $K,\fraka^+$ for the rest of this discussion.
Then recall that we can re-express the flag manifold $\cF_\theta:=G/P_{\theta}=K/M_{\theta}$ and we have the boundary map
\begin{align}
	\label{eqn:boundary_map_KM_flag_manifold}
	\begin{split}
	\xi &\colon \partial\wtilde{X}\to \cF_{\theta} =\rightquot{K}{M_{\theta}}
	 % \text{ and its variant composing with the inverse map:}\\
	% \xi^{i} &\colon \partial\wtilde{X}\to \leftquot{M_{\theta}}{K}
	\end{split}
\end{align}
Let us also denote by $\pi_{K,\theta}$ the composition
\[
	\pi_{K,\theta}\colon K \to \leftquot{M_\theta}{K} \xrightarrow{g\mapsto g^{-1}}\rightquot{K}{M_{\theta}} .
\]
The key consequence of the Anosov property is then:
\begin{theorem}[Control of $K$ component]
	\label{thm:control_of_k_component}
	The projection of $\cL_+(\Gamma)$ to $K$ is contained in $\pi_{K,\theta}^{-1}(\xi(\partial \wtilde{X}))$.
\end{theorem}
Denote from now on by $\cL K_+$ the projection of $\cL_+$ to the $K$-component.
\begin{proof}
	The assertion is a direct consequence of \cite[Thm.~5.3]{GGKW_AnosovReps_GT} and the discussion in \S5.1 of loc.cit. preceding it.
	The only changes are that $\Xi_{\theta}$ there refers to $k_-(g)$ in our notation for the $KAK$ decomposition of $g$, so there is no need in loc.cit. to apply the inverse map on $G$.
	The flag manifold $\cF_{\theta}$ is identified with $G/P_{\theta}$, but $K$ acts transitively so it is also $K/M_{\theta}$.
	Note also that the ``gap summation property'' in loc.cit. needs to be replaced by the quantitative divergence of \autoref{def:log_anosov_representation} and $\log\norm{\gamma}$ from \autoref{sssec:norm_on_fundamental_group} instead of the word length in the group.
\end{proof}

\fxnote{Include explicit statement about stability and bundles in representations, for the Anosov case}

%%					end of subsec: Stability and the numerical criterion
%%=============================================================================

%%=============================================================================
%%					start of subsec: Integral vectors in the limit set

\subsection{Integral vectors in the limit set}
	\label{ssec:integral_vectors_in_the_limit_set}

\subsubsection{Setup}
	\label{sssec:setup_integral_vectors_in_the_limit_set}
Let $\rho\colon \pi_1(X)\to G$ be a $\theta$-log-Anosov representation.
Recall that $G$ is assumed to be a real semisimple algebraic group, and for this section we assume that $G$ has a $\bQ$-structure as well, with $\Lambda\subset G$ an arithmetic lattice.
Suppose that the image of $\rho$ lands in $\Lambda$, and consider a linear representation $G\xrightarrow{\phi} \GL(V)$ for which $\Lambda$ preserves a lattice $V_{\bZ}\subset V_{\bR}$.

\subsubsection{Decreasing subbundle}
	\label{sssec:decreasing_subbundle}
Let $\bV_{i,\theta}\subset \bV$ be the $g_t$-invariant subbundles provided by \autoref{thm:domination_in_log_anosov_representations}.
We shall consider
\[
	\bV^{-}:=\bigoplus_{V_{\theta,i}\prec_{\theta}0} \bV_{\theta,i}
\]
i.e. all the subbundles for which the weights appearing in them are strictly $\theta$-dominated by the zero weight.
It follows from \autoref{cor:definitely_contracting_vectors} that for any $v\in \bV^{-}$, we have that $\norm{vg_t}\to 0$ as $t\to +\infty$.
In fact $\norm{v g_t}\leq Ce^{-\ve t}\norm{v}$ for some uniform $C,\ve>0$ (assuming that $v$ starts in the compact part and we evaluate the norm at times when the flow recurred to the compact part).

We shall consider the associated map
\[
	\xi^{-}_{V}\colon \partial\wtilde{X}\to \Gr\left(\rk \bV^{-},V\right)
\]
into the Grassmannian parametrizing $\bV^{-}$.

\begin{theorem}[Integral vectors must come from cusps]
	\label{thm:integral_vectors_must_come_from_cusps}
	Suppose that $v\in V_{\bZ}\setminus \{0\}$ is such that there exists $p\in \partial \wtilde{X}$ such that $v\in \xi^{-}_{V}(p)$.

	Then $p$ is a parabolic boundary point.
	Furthermore, let $T$ be the corresponding quasi-unipotent monodromy transformation, and let $W_{\bullet}V$ be the weight filtration induced by (the logarithm of) $T$, as per \autoref{sssec:weight_filtration}.
	Then $v\in W_{-1}V$.
\end{theorem}
\begin{proof}
	Consider the unit vector $u\in T^1\wtilde{X}$ over the basepoint $x_0$ and such that the associated hyperbolic geodesic in $\wtilde{X}$ lands at $p$ on the boundary.
	Denote by $ug_t$ the geodesic flow at time $t$ for $u$.

	Suppose that $p$ is not a parabolic point, then there exist infinitely many times $t_1<t_2<\cdots$ such that the geodesic, projected to $X$, returns to a fixed compact set.
	By the construction of $\bV^{-}$, the norm $\norm{v}_{ug_t}$ goes to $0$ as $t\to +\infty$, in fact exponentially fast.
	On the other hand, given any compact set $K$ in $X$, there exists a uniform lower bound $\ve(K)>0$ on the norm of any integral vector in $\bV_\bZ$ on $K$.
	Since the image of $v$ under the monodromy transformations remains integral, this provides a contradiction.

	It follows that $p$ is indeed a parabolic point, and $\norm{v}_{ug_t}$ goes to zero as $t\to +\infty$.
	By \autoref{prop:estimates_for_strictly_adapted_metrics}, this is equivalent to $v$ belonging to the negative part of the weight filtration induced by the monodromy (the exponent $k$ in the lemma must be strictly negative, and $\Re \tau$ stays bounded).
\end{proof}

%%					end of subsec: Integral vectors in the limit set
%%=============================================================================

%%=============================================================================
%%					start of subsec: Minimality of the action on the limit set

\subsection{Minimality of the action on a limit set}
	\label{ssec:minimality_of_the_action_on_the_limit_set}

\subsubsection{Setup}
	\label{sssec:setup_minimality_of_the_action_on_the_limit_set}
% We specialize to the situation of interest in the next section.
Let $\rho\colon \pi_1(X)\to \Sp(V_{\bR})$ be an $\alpha_1$-Anosov representation, where $V_{\bR}$ is a $4$-dimensional symplectic space.
Set $\theta=\{\alpha_1\}$.
By \autoref{thm:existence_of_boundary_maps} there exists a continuous, equivariant, transverse and dynamics-preserving map
\[
	\xi \colon \partial\wtilde{X}\to \bP(V_{\bR}) = \cF_{\theta}
\]
Let us call the image curve $\Lambda_P:=\xi(\partial\wtilde{X})$, and the image of the representation $\Gamma:=\rho(\pi_1(X))$.
Because the action of $\pi_1(X)$ on $\partial\wtilde{X}$ is minimal, it follows that the action of $\Gamma$ on $\Lambda_P$ is minimal.

Recall also that we have a ``limit set'' $\Lambda_L\subset \LGr(V_{\bR})$, defined as the set of Lagrangians which contain a line from $\Lambda_P$.
Equivalently, for every $\xi(p)\in \Lambda_P$ we have the photon $\phi(p)\subset \LGr(V_{\bR})$, see \autoref{sssec:photons}, and $\Lambda_L=\cup_{p\in \wtilde{X}}\phi(p)$.

\subsubsection{Assumption on unipotent}
	\label{sssec:assumption_on_unipotent}
We make the following \emph{crucial standing assumption}: there exists at least one cusp in $X$ such that the monodromy around it is generated (up to finite index) by a rank $1$ unipotent matrix fixing a line and its symplectic orthogonal.
This assumption is satisfied for all hypergeometric examples from \autoref{sec:hypergeometric_local_systems}, since the monodromy around $1\in \bP^1(\bC)$ is always a unipotent satisfying the requirement.

We will call any parabolic boundary point with the above property ``good''.
By minimality of the action of $\pi_1(X)$ on $\partial\wtilde{X}$, the set of good parabolic points is dense on the boundary.

The goal of this section is to prove:
\begin{theorem}[Minimality of the action]
	\label{thm:minimality_of_the_action}
	Under the above assumptions, the action of $\Gamma$ on $\Lambda_L$ is minimal, i.e. every orbit is dense.
	Equivalently, any closed $\Gamma$-invariant set is either empty or all of $\Lambda_L$.
\end{theorem}

\subsubsection{Correspondence Principle}
	\label{sssec:correspondence_principle}
Let $P_{\theta^c}$ be the parabolic such that $\LGr(V_{\bR})=G/P_{\theta^c}$ (here $\theta^c:=\Delta\setminus \theta =\{\alpha_2\}$ is the complement of $\theta$ in the set of roots).
It is well-known, and immediate to verify, that there are natural bijections between closed invariant sets, and orbit closures, for the following group actions:
\begin{enumerate}
	\item $\Gamma$ acting on $\rightquot{G}{P_{\theta^c}}$.
	\item $\Gamma\times P_{\theta^c}$ acting on $G$ by a left-right action of the corresponding factors.
	\item $P_{\theta^c}$ acting on the right on $\leftquot{\Gamma}{G}$.
\end{enumerate}
The natural bijections are obtained by passing through the middle action and taking images or preimages.

Recall that the action of $\Gamma$ on $\LGr(V_\bR)=G/P_{\theta^c}$ has a closed invariant set $\Lambda_L$, on which the action is minimal by \autoref{thm:minimality_of_the_action}, and the complementary open set $\Omega_L$ on which the action is proper.
In particular, every orbit of $\Gamma$ on $\Omega_L$ is closed.
We conclude:

\begin{corollary}[Orbit dichotomy in homogeneous space]
	\label{cor:orbit_dichotomy_in_homogeneous_space}
	There exists a closed $P_{\theta^c}$-invariant set $\cR\subset \leftquot{\Gamma}{G}$, such that the $P_{\theta^c}$-orbit of any point in $\cR$ is dense in $\cR$, and the $P_{\theta^c}$-orbit of any point outside $\cR$ is closed in the complement of $\cR$.
\end{corollary}
We can think of $\cR$ as a ``recurrent set'' for the $P_{\theta^c}$-dynamics.

\subsubsection{Oseledets Lagrangians are dense}
	\label{sssec:oseledets_lagrangians_dense}
We return for a moment to the setting of \autoref{sssec:setup_minimality_of_the_action_on_the_limit_set} with a finite-volume hyperbolic Riemann surface $X$.
Recall that by the Oseledets Multiplicative Ergodic Theorem (see e.g. \cite[2.2.6]{sfilip_MET_lectures}) there exists a \emph{measurable} $\rho$-equivariant map $\xi_2\colon \partial\wtilde{X}\to \LGr(V_\bR)$ which is given by the second term of the Lyapunov filtration (implicit in this is that the Lyapunov spectrum is simple, which can be established by a proof identical to \cite[Thm.~4]{EskinMatheus2015_A-coding-free-simplicity-criterion-for-the-Lyapunov-exponents-of-Teichmuller-curves}).
It follows from \autoref{thm:minimality_of_the_action} that the image of this measurable map is \emph{dense} in $\LGr(V_{\bR})$.
One can also view the image as a measurable $\Gamma$-equivariant section of the projection $\Lambda_L\to \Lambda_P$, which again is dense.
This is reminiscent of Furstenberg's example of a minimal, but not uniquely ergodic skew product on $\bR^2/\bZ^2$.

\subsubsection{Unipotent dynamics and the fibration}
	\label{sssec:unipotent_dynamics_and_the_fibration}
To prove \autoref{thm:minimality_of_the_action}, we need some preliminaries on the dynamics of a rank $1$ unipotent as in the standing assumption.

Suppose $l\subset V_{\bR}$ is a line, and $l^{\perp}$ is its symplectic orthogonal, so we have a flag $l\subset l^{\perp}\subset V_{\bR}$.
Let $T\colon V_{\bR}\to V_{\bR}$ be a nontrivial rank $1$ unipotent transformation fixing $l^{\perp}$ pointwise and such that $T-\id$ maps $V/l^\perp$ isomorphically onto $l$.
We will identify the photon $\phi(l)$ of Lagrangians $L$ such that $l\subset L$ with $\bP(l^\perp/l)$, by assigning to $L$ the direction $l^\perp\cap L$.

Define, more generally, the following ``contraction map''
\begin{align}
	\label{eqn:contracting_Lagrangians}
	\begin{split}
	c_l\colon \LGr(V_{\bR})\setminus \phi(l)&\to \phi(l)\\
	L'&\mapsto (l^{\perp}\cap L) \mod l
	\end{split}
\end{align}
% \laternote{If there's time, work out what happens if I try to extend it to all of $\LGr(V_{\bR})$}

\begin{lemma}[Limits under unipotent iterates]
	\label{lem:limits_under_unipotent_iterates}
	\leavevmode
	\begin{enumerate}
		\item Suppose that $L'\in \LGr(V_{\bR})\setminus \phi(l)$ is any Lagrangian.
		Then as $n\to +\infty$ we have
		\[
			T^n L' \to c_l(L')\in \phi(l).
		\]
		\item Suppose that $l'\subset V_{\bR}$ is another line, such that $\phi(l')$ and $\phi(l)$ are disjoint.
		Then the maps
		\[
			T^{n}\vert_{\phi(l')}\text{ converge to the map }c_l\vert_{\phi(l')}\colon \phi(l')\to \phi(l).
		\]
		Furthermore, if we write $V_{\bR}=l\oplus \left(l^{\perp}\cap l^{',\perp}\right)\oplus l'$ and identify $\phi(l)$ and $\phi(l')$ with $\bP\left(l^{\perp}\cap l^{',\perp}\right)$ then $c_l$ as restricted above is the identity map.
	\end{enumerate}
\end{lemma}
\begin{proof}
	The verification is immediate, by computation in a symplectic basis in which $l$ is spanned by $e_1$, $l^{\perp}$ is spanned by $e_1,e_2,f_2$, and $f_1$ is such that $T(f_1)=f_1 + e_1$.
	Given $l'$ as in the second part of the statement, we can choose our basis such that $l'$ is spanned by $f_2$.

	Note that since $T$ acts as the identity on $l^{\perp}$, under iteration the intersection $T^{n}L'\cap l^{\perp}$ does not change, while the remaining vector in $L'$ converges to $L$.
\end{proof}

\subsubsection{Notation}
	\label{sssec:notation_unipotent_minimality}
For economy of notation, for a point $p\in \partial\wtilde{X}$, we will write $\phi(p)$ for the photon $\phi(\xi(p))$.
Similarly, we will write $c_p$, instead of $c_{\xi(p)}$, for the associated map defined in \autoref{eqn:contracting_Lagrangians}.
When $p$ is a good parabolic boundary point, we will write $T_p$ for the associated unipotent transformation.

Let from now on $S\subset \Lambda_L$ be a nonempty, closed, $\Gamma$-invariant set.
Our goal is to show that $S=\Lambda_L$, and we will do so by showing that for any good parabolic boundary point $p$, we have that $S\cap \phi(p)=\phi(p)$.
By density of such $p$ in $\partial\wtilde{X}$, the result will follow.

\begin{proposition}[Pushing and pulling fibers]
	\label{prop:pushing_and_pulling_fibers}
	Let $p$ be a good parabolic boundary point, and $p'\in \partial\wtilde{X}\setminus p$ a different point.
	\begin{enumerate}
		\item We have that
		\[
		 	c_p\left(S\cap \phi(p')\right) \subset S\cap \phi(p).
	 	\] 
	 	\item Suppose that $p'$ is a good parabolic point as well.
	 	Then
	 	\[
	 		c_{p}\left(S\cap \phi(p')\right) = S\cap \phi(p).
	 	\]
	 	\item Suppose again that $p'$ is arbitrary.
	 	Then
	 	\[
	 		c_{p}\left(S\cap \phi(p')\right) = S\cap \phi(p).
	 	\]
	\end{enumerate}
\end{proposition}
Clearly the last part implies the previous ones, but we will prove the statements in this order.
\begin{proof}
	For (i), because $S$ is closed and $T_p$-invariant, it follows from \autoref{lem:limits_under_unipotent_iterates} that any accumulation points of $T_p^n(S\cap \phi(p'))$ are contained in $S\cap \phi(p)$, so the result follows.

	Part (ii) follows by applying (i) to both parabolic points, going back-and-forth, and using the identification of the photons from \autoref{lem:limits_under_unipotent_iterates}.

	Finally for (iii) we need to prove the reverse inclusion from that in (i).
	But take a sequence of good parabolic points $p_i$ converging to $p'$.
	Then we have that $\limsup (S\cap \phi(p_i))\subset S\cap \phi(p')$, where $\limsup$ of sets refers to all the accumulation points.
	Applying $c_p$ to the above containment gives the result by (ii).
\end{proof}

The following is the key calculation which will give the result.
\begin{proposition}[Fiberwise unipotents]
	\label{prop:fiberwise_unipotents}
	Suppose that $p_1,p_2,p_3\in \partial\wtilde{X}$ are three distinct good parabolic boundary points.
	Set $V_{13}:=l_{p_1}^{\perp}\cap l_{p_3}^{\perp}$, so that we have a decomposition $V=l_1 \oplus V_{13}\oplus l_3$.
	Suppose that $l_{p_2}$ is \emph{not} contained in $l_{p_1}\oplus l_{p_3}$, and let $l_{123}\subset V_{13}$ be the projection of $l_{p_2}$ to $V_{13}$, nontrivial by assumption. 

	Then the composition
	\[
		\phi(p_1)\xrightarrow{c_{p_2}}\phi(p_2)
		\xrightarrow{c_{p_3}}\phi(p_3)
		\xrightarrow{c_{p_1}}\phi(p_1)
	\]
	is (the projectivization) of a nontrivial unipotent transformation, with fixed point $l_{123}$ when identifying $\phi(p_1)\isom \bP(V_{13})$.
\end{proposition}

\subsubsection{Finishing the proof}
	\label{sssec:finishing_the_proof}
The proof of \autoref{prop:fiberwise_unipotents} is contained below and is a calculation.
Here we explain why that establishes \autoref{thm:minimality_of_the_action}.
Fix $p_1,p_3$ two good, distinct parabolic boundary points.
A nonempty open set $U_{1,3}\subset \partial\wtilde{X}$ of $p_2$'s will satisfy the assumption of \autoref{prop:fiberwise_unipotents}, i.e. $l_{p_2}$ will not be contained in $l_{p_1}\oplus l_{p_3}$, otherwise this would contradict the Zariski-density of $\Gamma$.
Note also that the image of the projection of $\xi(U_{1,3})$ in $\bP(V_{1,3})$ will contain an open set, for the same reason.

In the open set $U_{1,3}$, a dense collection of good parabolic boundary points is available.
\autoref{prop:fiberwise_unipotents} gives then a family of nontrivial unipotent transformations of $\bP(V_{1,3})$, and the closure of their fixed points contains an open set.
Furthermore, by \autoref{prop:pushing_and_pulling_fibers}, the intersection $S\cap \phi(p_1)$ (after the identification of $\phi(p_1)\isom \bP(V_{1,3})$), is invariant under these transformations.
It follows that $S\cap \phi(p_1)=\phi(p_1)$, since all the fixed points of the unipotent transformations are clearly in $S$, and $S$ contains some open set with fixed points of unipotent transformations in the interior of that open set.
Large iterates of a single unipotent then expand the open set to cover all of $\phi(p_1)$.

\subsubsection{Proof of \autoref{prop:fiberwise_unipotents}}
	\label{sssec:proof_of_prop:fiberwise_unipotents}
To simplify notation, we will write $l_{i}$ instead of $l_{p_i}$.
Choose a symplectic basis such that $l_{1}$ is spanned by $e_1$ and $l_{3}$ is spanned by $f_1$, so that $V_{13}$ is spanned by $e_2,f_2$.
By acting with an internal symplectic automorphism of $V_{13}$, we can assume that
\[
	l_{2} \text{ is spanned by }\alpha \cdot e_1 + \beta \cdot f_1 + e_2\text{ with }\alpha\cdot \beta\neq 0.
\]
Indeed the $e_1$ and $f_1$ components cannot vanish, otherwise the transversality of $l_{2}$ with either of $l_{1},l_{3}$ would be contradicted, and the $V_{13}$ component is nontrivial and we can normalize it to be $e_2$.

It follows that
\[
	l_{2}^{\perp} \text{ is spanned by }e_2,e_1+\beta f_2, f_1-\alpha f_2
\]
We will take as a basis of $l_2^{\perp}/l_2$ the vectors
\[
	v_1 := e_1 + \beta f_2 \quad v_2 = f_1 - \alpha f_2
\]
Note also that
\[
	e_2 = (\alpha e_1 + \beta f_1 + e_2) - \alpha(e_1 + \beta f_2) - \beta(f_1-\alpha f_2)
\]
which will be useful for passing to representatives in $l_2^\perp/l_2$ in the $v_1,v_2$-basis.

We will also use for both $l_1^\perp/l_1$ and $l_3^{\perp}/l_3$ the basis $e_2,f_2$.
With these normalizations, we will compute the matrices of the maps $c_{p_2}$ and $c_{p_3}^{-1}$, and check the needed assertion.
Note that the matrix of $c_{p_1}$ is the identity in the selected bases.

To compute the image of $(x e_2 + y f_2)$ under $c_{p_2}$, we need to find its representative $\mod e_1$ such that it is symplectically orthogonal to $l_2$, and then pick out the $v_1,v_2$-coefficients modulo $l_2$.
We have
\begin{align*}
	x e_2 + y f_2 & \equiv x e_2 + \frac{y}{\beta}\left(e_1 + \beta f_2\right) \mod e_1\\
	& = x\big[\alpha e_1 + \beta f_1 + e_2 - \alpha(e_1+\beta f_2)-\beta(f_1-\alpha f_2)\big]\\
	& + \frac{y}{\beta}(e_1 + \beta f_2)\\
	& \equiv \left(\frac{y}{\beta}-\alpha x\right)(e_1 + \beta f_2)
	- \beta x (f_1 -\alpha f_2) \mod l_2
\end{align*}
In conclusion, the matrix of $c_{p_1}$ in the selected basis is
\[	
	c_{p_2}=\begin{bmatrix}
		-\alpha & \frac{1}{\beta}\\
		- \beta & 0
	\end{bmatrix}.
\]

We use the same type of calculation to find $c_{p_3}^{-1}$, applied to a vector $x e_2 + yf_2$:
\begin{align*}
	x e_2 + y f_2 & \equiv x e_2 - \frac{y}{\alpha}\left(f_1 -\alpha f_2\right) \mod f_1\\
	& = x\big[\alpha e_1 + \beta f_1 + e_2 - \alpha(e_1+\beta f_2)-\beta(f_1-\alpha f_2)\big]\\
	& - \frac{y}{\alpha}\left(f_1 -\alpha f_2\right)\\
	& \equiv -\alpha x(e_1 + \beta f_2)
	- \left(\beta x  + \frac{y}{\alpha}\right) (f_1 -\alpha f_2) \mod l_2
\end{align*}
and so
\[
	c_{p_3}^{-1} = \begin{bmatrix}
		-\alpha & 0\\
		- \beta & -\frac{1}{\alpha}
	\end{bmatrix}
	\quad\text{so}\quad
	c_{p_3}=
	\begin{bmatrix}
		-\frac{1}{\alpha} & 0 \\
		\beta & -\alpha
	\end{bmatrix}.
\]
Multiplying out yields
\[
	c_{p_3}\cdot c_{p_2} = \begin{bmatrix}
		1 & -\frac{1}{\alpha\beta}\\
		0 & 1
	\end{bmatrix}
\]
which is precisely the needed assertion. \hfill \qed

\subsubsection{Geometric interpretation}
	\label{sssec:geometric_interpretation}
We end by remarking that it is geometrically clear that the composed transformations will fix a point.
The whole calculation is to make sure that the composed map is not the identity.
Indeed, in $\bP(V_{\bR})$ we have three hyperplanes (corresponding to $l_{i}^{\perp}$) and points on each of them (corresponding to $l_i$).
The transformations $c_{\bullet}$, say going from $i$ to $j$, identify the pencils of lines in $l_i^{\perp}$ passing through $l_i$, with those passing through $l_j$ in $l_j^{\perp}$, by identifying them with the common line at which $l_i^{\perp}$ and $l_j^{\perp}$ intersect.
Now it is clear that the ``vertex'' where all three hyperplanes meet will be preserved by any of these mappings.

%%					end of subsec: Minimality of the action on the limit set
%%=============================================================================

%%%%%%%%%%%%%%%%%%%%%%%%%%%%%%%%%%%%%%%%%%%%%%%%%%%%%%%%%%%%%%%%%%%%%%%%%%%%%%%
%%% 				End of Section: Elements of Lie theory and Anosov representations
%%%%%%%%%%%%%%%%%%%%%%%%%%%%%%%%%%%%%%%%%%%%%%%%%%%%%%%%%%%%%%%%%%%%%%%%%%%%%%%

%%%%%%%%%%%%%%%%%%%%%%%%%%%%%%%%%%%%%%%%%%%%%%%%%%%%%%%%%%%%%%%%%%%%%%%%%%%%%%%
%%% 				Start of Section: Hodge theory and Anosov representations
%%%%%%%%%%%%%%%%%%%%%%%%%%%%%%%%%%%%%%%%%%%%%%%%%%%%%%%%%%%%%%%%%%%%%%%%%%%%%%%

\section{Hodge theory and Anosov representations}
	\label{sec:hodge_theory_and_anosov_representations}

\paragraph{Outline}
In \autoref{ssec:establishing_the_anosov_property} we use the analytic results on assumption A variations of Hodge structure from \autoref{sec:hodge_theory_and_growth} to conclude that their monodromy representations are log-Anosov.
Using the tools on log-Anosov representations developed in \autoref{sec:elements_of_lie_theory_and_anosov_representations}, we construct in \autoref{ssec:complex_domains_of_discontinuity} some further domains of discontinuity in complex flag manifolds.

%%=============================================================================
%%					start of subsec: Establishing the Anosov property

\subsection{Establishing the Anosov property}
	\label{ssec:establishing_the_anosov_property}

\subsubsection{Setup}
	\label{sssec:setup_establishing_the_anosov_property}
We use the notation and conventions of \autoref{sec:hodge_theory_and_growth} for Hodge theory and \autoref{sec:elements_of_lie_theory_and_anosov_representations} for Lie theory.
Fix for the rest of this section a local system $\bV\to X$ supporting a variation of Hodge structure satisfying assumption A of \autoref{def:assumption_a}.
For simplicity of exposition, we assume from the start that there is a symplectic version of the local system, i.e. a lift of the representation from $\SO_{2,3}(W_{\bR})$ to $\Sp_4(V_{\bR})$ (this is the case for all hypergeometric examples).

\subsubsection{Limit sets in flag manifolds}
	\label{sssec:limit_sets_in_flag_manifolds}
We will establish in \autoref{thm:anosov_property_from_assumption_A} below that the monodromy representation is $\alpha_1$-log-Anosov and so we set $\theta=\{\alpha_1\}$.
\autoref{thm:existence_of_boundary_maps} then gives a boundary map
\[
	\xi \colon \partial \wtilde{X}\to \cF_{\theta}\isom \rightquot{G}{P_{\theta}}
	\isom \rightquot{K}{M_{\theta}}
	\isom \bP(V_{\bR}).
\]
By \autoref{thm:control_of_k_component}, we have that $\cL K_+(\Gamma)$ is contained in $\pi_{K,\theta}^{-1}(\xi(\partial \wtilde{X}))$, and we will identify for convenience of notation the two sets:
\[
	\cL K_+(\Gamma):= \pi_{K,\theta}^{-1}(\xi(\partial \wtilde{X}))
\]
Note that $\cL K_+(\Gamma)\subset K$ is left $M_\theta$-invariant.
Project now this set to the other flag manifold $\cF_{\alpha_2}=K/M_{\alpha_2}=\LGr(V_\bR)$ via the map $K\xrightarrow{g\mapsto g^{-1}} K/M_{\alpha_2}=\LGr(V_{\bR})$.
Denote the image by $\Lambda_L(\Gamma)$.

Note that an alternative definition is to observe that the manifold of full flags $\cF_{\alpha_1,\alpha_2}$ fibers over each of $\cF_{\alpha_i}$.
The construction of $\Lambda_L(\Gamma)$ is then to take the preimage of $\xi(\partial \wtilde{X})\subset \cF_{\alpha_1}$ and project it to $\cF_{\alpha_2}$.

\subsubsection{Dominated splittings of cocycles}
	\label{sssec:dominated_splittings_of_cocycles}
Continuing with the assumption that $\rho$ is $\alpha_1$-log-Anosov, by \autoref{thm:domination_in_log_anosov_representations} we have domination structures associated to the cocycles induced by the local systems $\bV,\bW$.
It will also be useful to consider $\bU:=\Lambda^2 \bW\isom \Sym^2 \bV$, which is the cocycle associated to the adjoint representation of $\fraksp_{4}$.
Its domination structure is described by \autoref{fig:adjoint_rep_poset}.

It is straightforward to also work out the domination structure of $\bV$ and $\bW$.
On $\bV$ we have over $T^1X$ a $g_t$-invariant decomposition
\begin{align*}
	\bV& =\bV_{\mu_1}\succ_{\theta} \bV_{-\mu_2,\mu_2}\succ_{\theta} \bV_{-\mu_1}
	\intertext{where we have indicated the domination structure, and the dimension of each subbundle is given by the number of weights occuring in its index.
	For $\bW$ we similarly have:}
	\bW& =\bW_{\mu_1+\mu_2,\mu_1-\mu_2}\succ_{\theta} \bW_{0}\succ_{\theta}
	\bW_{\mu_2-\mu_1,-\mu_1-\mu_2}
\end{align*}
Note that the two extreme subspaces of $\bW$ are isotropic and $2$-dimensional, giving photons (see \autoref{sssec:photons}).
Their union over all points on the boundary of $\wtilde{X}$ constitutes by its definition the limit set $\Lambda_{L}(\Gamma)\subset \LGr(V_{\bR})\subset \bP(W_{\bR})$ (recall also that $\LGr(V_{\bR})$ is the set of isotropic vectors in $\bP(W_{\bR})$).

% More generally, we can make use of the stable filtration as constructed in \autoref{sssec:stable_filtration} of $\bW$ over $T^1X$.

\begin{theorem}[Anosov property]
	\label{thm:anosov_property_from_assumption_A}
	Let $\bV$ be a variation of Hodge structure satisfying assumption A.
	\leavevmode
	\begin{enumerate}
		\item The monodromy representation $\rho\colon \pi_1(X)\to \Sp_{4}(\bR)$ is $\alpha_1$-log-Anosov.
		\item The complement of the open set $\Omega_L\subset \LGr(V_{\bR})$ from \autoref{thm:real_uniformization_lagrangian} coincides with the limit set $\Lambda_L(\Gamma)\subset \LGr(V_{\bR})$.
	\end{enumerate}
\end{theorem}
\begin{proof}
	To establish the $\alpha_1$-log-Anosov property, in the sense of \autoref{def:log_anosov_representation}, we need to establish lower bounds on the growth of singular values.
	For this, we will use the second exterior power representation to $G:=\SO_{2,3}(W_\bR)$.
	Let $K\subset G$ be a maximal compact subgroup, corresponding to the Hodge decomposition at the basepoint $x_0\in \wtilde{X}$.
	Take the corresponding decomposition
	\[
		W_{\bR} = P^2(x_0) \oplus N^3(x_0)
	\]
	which is orthogonal for both the definite, and indefinite, inner products on $W_{\bR}$, where the definite inner product comes from the Hodge structure at $x_0$, and $\dim P^2=2, \dim N^3=3$, and $P^2$ is positive-definite while $N^3$ is negative-definite for the indefinite inner product.

	Let us choose a Cartan subalgebra and Weyl chamber compatible with the maximal compact $K$.
	This yields an orthonormal bases $w_1,w_2\in P^2(x_0)$ and $n_0,n_1,n_2\in N^3(x_0)$, such that the root spaces for $\fraka$, using the notation of \autoref{fig:sp4roots}, are spanned by:
	\begin{align*}
		\begin{split}
			W_{\mu_1+\mu_2} & = w_1 + n_1\\
			W_{-\mu_1-\mu_2} & = w_1 - n_1
		\end{split}
		\begin{split}
			W_{\mu_1-\mu_2} & = w_2 + n_2\\
			W_{\mu_2-\mu_1} & = w_2 - n_2
		\end{split}
		\quad\quad
		\begin{split}
			W_0 = n_0
		\end{split}
	\end{align*}
	In particular, note that
	\[
		2e^{\overrightarrow{\mu}}w_2 = e^{\mu_1-\mu_2}(w_2 + n_2) + e^{-(\mu_1-\mu_2)}(w_2-n_2)
	\]
	and for $\overrightarrow{\mu}\in \fraka^+$, i.e. $\mu_1\geq \mu_2\geq0$ it is the first term that dominates (we have denoted by $\overrightarrow{\mu}$ the element of the Lie algebra to distinguish in the expression above the Lie group exponential from the scalar exponential function on the right-hand side).
	
	Take now $\rho(\gamma)\in \Gamma\subset G$, and its KAK decomposition
	\[
		\rho(\gamma) = k_- e^{\mu} k_+ \quad k_\pm \in K
	\]
	where $K$ is the maximal compact corresponding to $x_0$.
	Our task is to estimate $\alpha_1(\mu)=\mu_1-\mu_2$ from below.

	We have by \autoref{eqn:norm_of_vector_flat_bundle}:
	\[
		\norm{w}_{\gamma x_0} = \norm{\rho(\gamma)^{-1}w}_{x_0} = 
		\norm{e^{\mu}k' w}_{x_0}
	\]
	for a certain $k'=k_0^{-1}k_-^{-1}\in K$, where $k_0$ is an element of $K$ inducing the longest Weyl group element action on $\fraka$, and where we used that the opposition involution in the symplectic group is trivial.

	Now there exists a unit vector $w'\in P^2(x_0)$ such that $k' w' = w_2$.
	\autoref{lem:exponential_growth_pos_def} applies to any unit vector $w\in P^2(x_0)$, with uniform constants $C_1,C_2,\ve>0$.
	It follows that
	\begin{align*}
		\norm{w'}^2_{\gamma x_0}  & \geq \frac{1}{C_1}e^{\ve\cdot \dist(x_0,\gamma x_0)} - C_2 && \text{ by \autoref{lem:exponential_growth_pos_def}}\\
		\norm{w'}^2_{\gamma x_0}  & \leq \frac{1}{4}e^{2(\mu_1-\mu_2)}+C_3 && \text{ by the choice of roots and }w'
	\end{align*}
	from which the desired lower bound
	\[
		\mu_1 - \mu_2 \geq \frac{1}{C_4}\dist(x_0,\gamma x_0) - C_5 \text{ follows.}
	\]
	To identify the complement of the domain of discontinuity $\Omega_L\subset \LGr(V_\bR)\subset \bP(W_{\bR})$, suppose $[w]\notin \Omega_L$ is isotropic.
	Then it follows from the proof of \autoref{thm:real_uniformization_lagrangian} that $\norm{w}^2$ does not achieve its minimum on $\wtilde{X}$.
	But if $[w]$ were also disjoint from $\Lambda_L(\Gamma)$, then $\norm{w}^2$ would grow in every direction of the geodesic flow, hence would be a proper function, so it would achieve its minimum, yielding a contradiction.

	Conversely, if $[w]\in \Lambda_L(\Gamma)$, then there is a direction of the geodesic flow on which $\norm{w}^2$ stays bounded.
	This implies that this function can't be proper, so $[w]$ can't be in $\Omega_L$ as it would contradict \autoref{lem:exponential_growth_isotropic_version}.
\end{proof}

%%					end of subsec: Establishing the Anosov property
%%=============================================================================

%%=============================================================================
%%					start of subsec: Complex domains of discontinuity

\subsection{Complex domains of discontinuity}
	\label{ssec:complex_domains_of_discontinuity}

\subsubsection{Setup}
	\label{sssec:setup_complex_domains_of_discontinuity}
We continue with the same setup as in the previous section, but now investigate the action on complex flag manifolds.
Recall that $\Gamma:=\rho(\pi_1(X))\subset G= \Sp_4(V_\bR)$.

\begin{theorem}[Proper discontinuity on indefinite Lagrangians]
	\label{thm:proper_discontinuity_on_indefinite_lagrangians}
	The monodromy $\Gamma$ acts properly discontinuously on $\LGr^{1,1}(V_{\bC})$.
\end{theorem}
Under the assumption of $\alpha_1$-divergence on the group, such results are standard, see e.g. \cite[Thm.~7.4]{GuichardWienhard2012_Anosov-representations:-domains-of-discontinuity-and-applications}, see also \cite[\S1.5]{Benoist_Actions-propres-sur-les-espaces-homogenes-reductifs}.
\begin{proof}
	This follows from the $\alpha_1$-divergence of the group $\Gamma$ (see \autoref{def:theta_divergence}), which is implied the the $\alpha_1$-log-Anosov property.
	Indeed, in order to apply \autoref{thm:proper_discontinuity_criterion}, we need to check that no vector $[w]\in \LGr^{1,1}(V_{\bC})\subset\bP(W_{\bC})$ can lie in $k^{-1}W^{\leq 0}([\mu])$.
	But the $\alpha_1$-divergence property implies that $W^{\leq 0}_{\bC}([\mu])$ is contained in the span of $W_{-\mu_1-\mu_2}\oplus W_{\mu_2-\mu_1}\oplus W_{0}$.
	Suppose that $[w]$ did belong to (the complexification of) such a space, or to any real translate of it.
	Then so would its complex conjugate $[\ov{w}]$, and together $w,\ov{w}$ would span a negative-definite real $2$-dimensional subspace.
	But this contradicts the signature of $W^{\leq 0}([\mu])$, which is $(0,1,2)$, i.e. one negative direction and two totally degenerate directions.
\end{proof}

\begin{remark}[Proper discontinuity on $\bS^{1,3}(W_\bR)$]
	\label{rmk:proper_discontinuity_on_bs_1_w_br}
	It is also true that an $\alpha_1$-divergent group acts properly discontinuously on $\bS^{1,3}(W_\bR)$, the sphere of positive-definite unit vectors.
	The proof is very similar to that of \autoref{thm:proper_discontinuity_on_indefinite_lagrangians} given above, except that the linear-algebraic observation is that a positive-definite vector $w\in W_{\bR}$ cannot belong to $V^{\leq 0}([\mu])$ for the same signature reasons.
\end{remark}

We now consider the situation in $\bP(V_\bC)$, with the standing assumptions on $\bV$.

\begin{theorem}[Proper discontinuity on complex projective space]
	\label{thm:proper_discontinuity_on_complex_projective_space}
	Let $\Omega_{P}(\Gamma)\subset \bP(V_{\bC})$ be the complement of the closed set
	\[
		\Lambda_P(\Gamma):=\bigcup_{p\in \partial\wtilde{X}}\xi(p)_{\bC}^{\perp}
	\]
	where $\xi(p)^{\perp}_{\bC}$ is the complexified symplectic orthogonal, and $\xi\colon \partial\wtilde{X}\to \bP(V_\bR)$ the boundary map.
	
	Then $\Omega_{P}(\Gamma)$ is nonempty and the monodromy $\Gamma$ acts properly discontinuously on it.
\end{theorem}
\begin{proof}
	A combination of the criterion for proper discontinuity \autoref{thm:proper_discontinuity_criterion}, and the control on the limit set provided by \autoref{thm:control_of_k_component}, implies that $\Omega_P(\Gamma)$ is contained in the set of stable points so the action is properly discontinuous.

	It remains to check that the set is nonempty.
	In fact, let $F^2(x)\subset V$ be an element of the Hodge filtration for any $x\in \wtilde{X}$.
	Recall that $F^2(x)=\cV^{3,0}(x)\oplus \cV^{2,1}(x)$.
	We claim that any $v\in F^2(x)$ satisfying $I(v,\ov{v})=0$ (for the complexified symplectic pairing) will be in $\Omega_P(\Gamma)$.
	Note that such $v$ exist because the hermitian form on $F^2$ has signature $(1,1)$ and we are just taking null vectors for the hermitian form.

	Indeed, set $L:=span(v,\ov{v})$.
	Because $F^2$ is a complex Lagrangian, so $I(v,v)=0$ for any $v\in F^2$, and by the assumption on $v$, it follows that $L$ is a Lagrangian subspace, which is moreover defined over $\bR$.
	Suppose, by contradiction, that $v\in \xi(p)^{\perp}_{\bC}$ for some $p\in \partial \wtilde{X}$, or equivalently $I(v,\xi(p))=0$.
	Because $\xi(p)$ is real, it follows that $I(\ov{v},\xi(p))=0$.
	So the real Lagrangian $L$ satisfies $L_{\bR}\subset \xi(p)^{\perp}_{\bR}$ and $L_{\bC}\cap F^2(x)\neq 0$.

	Let $w_{L}=\Lambda^2 L\in W_{\bR}$ be the coordinates of $L$ in $W_{\bR}$.
	The property $L_{\bR}\subset \xi(p)^{\perp}_{\bR}$, which is equivalent to $\xi(p)\subset L_{\bR}$, implies that $[w_L]\in \Lambda_L(\Gamma)\subset \LGr(V_{\bR})$.
	However, the property $L_{\bC}\cap F^2(x)\neq 0$ implies that $[w_L]$ belongs to the electron $\beta(F^2(x))\subset \LGr(V_{\bR})$.
	By \autoref{thm:real_uniformization_lagrangian} the electron $\beta(F^2(x))$ is entirely contained in $\Omega_L$, thus giving a contradiction.
\end{proof}

Let us record the following consequence of the proof:
\begin{corollary}[One-sided containment of limit set]
	\label{cor:one_sided_containment_of_limit_set}
	Suppose $F^2(x)\subset V_{\bC}$ is an element of the Hodge filtration for $x\in \wtilde{X}$.
	Then the intersection of $\Lambda_P(\Gamma)$ and $\bP(F^2(x))$ is disjoint from the set of hermitian null vectors of $F^2(x)$.
\end{corollary}
Note that $\bP(F^2(x))\isom \bP^1(\bC)$ and the hermitian null vectors correspond to an ``equator''.
The two components of the complement correspond to the hermitian-positive vectors, containing $\cV^{3,0}(x)$, and hermitian-negative vectors, containing $\cV^{2,1}(x)$.
It is natural to ask: in which of the two components does the intersect with the limit set live?
It turns out that it lies entirely in the component of hermitian-positive vectors, although a general proof of this fact will be presented elsewhere.

In the case of the hypergeometric examples treated in \autoref{sec:hypergeometric_local_systems}, this assertion can be verified by considering the degeneration of the Hodge decomposition in a neighborhood of a singular point.
At least one point in the limit set is known (the invariant vector of the monodromy) and it can be checked from the asymptotics of the period map that indeed its symplectic orthogonal will intersect $F^2$ near $\cV^{3,0}$.

The proof in the general case is indirect.
Namely, if $\Lambda_P(\Gamma)$ always intersected $\bP(F^2(x))$ in the component containing the hermitian negative vectors, then in particular it would never intersect $\cV^{3,0}$ (which is positive-definite).
This would imply then that one can establish a formula for the \emph{top} Lyapunov exponent of $\bV$ (in the usual manner, see e.g. \cite{EskinKontsevichMoller2018_Lower-bounds-for-Lyapunov-exponents-of-flat-bundles-on-curves}).
This can then be used to show that the monodromy representation is maximal, in the sense of \autoref{ssec:maximal_hypergeometric_monodromy}, which combined with assumption A would imply that $\bV$ is in fact the third symmetric power of a weight $1$, rank $2$, uniformizing representation of the base Riemann surface $X$.
In this case, the limit set is completely explicit and a direct calculation verifies the assertion (see e.g. \cite[Proof of Lemma~5.8]{CollierTholozanToulisse2019_The-geometry-of-maximal-representations-of-surface-groups-into} for analogous calculations).

% \laternote{Expand comments below on Period Surface}

\subsubsection{Real algebraic curves, recollections}
	\label{sssec:real_algebraic_curves_recollections}
To motivate the discussion in the subsequent paragraph \autoref{sssec:a_space_with_many_pieces}, we recall some classical facts on Riemann surfaces.
Let $C$ be a proper algebraic curve over $\bR$.
Two possibilities arise: either the real points $C(\bR)$ disconnect the complex points $C(\bC)$, in which case the real curve $C$ is called ``separating'', or not.
It turns out that $C$ is separating if and only if it admits a map $C\xrightarrow{f} \bP^1$ such that $f^{-1}(\bP^1(\bR))\subset C(\bR)$ (i.e the only real values of $f$ come from real points on $C$).
One direction is straightforward, the converse existence result was established by Ahlfors \cite{Ahlfors1950_Open-Riemann-surfaces-and-extremal-problems-on-compact-subregions}, but see Gabard \cite{Gabard2006_Sur-la-representation-conforme-des-surfaces-de-Riemann-a-bord-et-une-caracterisation-des-courbes} for a more recent proof, in the spirit of the discussion below.

More importantly, when $C$ is separating we can write $C(\bC)=C^- \coprod C(\bR)\coprod C^+$, say choosing one of the complementary components to be labeled with $+$ (the discussion can avoid choices, including of $\sqrt{-1}$, if preferred).
The noncompact Riemann surface $C^+$ has a uniformization by the upper half-plane $\bH^+$, and there is a subgroup $\Gamma\subset \PSL_2(\bR)$ and an isomorphism $\Gamma\backslash \bH^+\isom C^+$.
The limit set of $\Gamma$ in $\bP^1(\bR)$ is a Cantor set $\Lambda$, and removing it from $\bP^1(\bC)$ gives an isomorphism $\Gamma\backslash (\bP^1(\bC)\setminus \Lambda)\isom C(\bC)$.
The stratification of $C(\bC)$ into the different components is now mimicked by the uniformization via $\bP^1(\bC)\setminus \Lambda$.

Note that $\Gamma$ is \emph{not} a lattice in $\PSL_2(\bR)$ and the quotient $\Gamma\backslash\bH^+$ has infinite hyperbolic volume.
Note also that the construction yields a uniformization of the real locus as $C(\bR)\isom \bP^1(\bR)\setminus \Lambda$.

\subsubsection{A space with many pieces}
	\label{sssec:a_space_with_many_pieces}
The uniformization discussion of real algebraic curves above has the following counterpart in our context.
Let $\LGr(V)\subset \bP(W)$ be the Lagrangian Grassmannian, which will be considered over $\bR$ and over $\bC$.
\autoref{thm:anosov_property_from_assumption_A} provides a limit set $\Lambda_{L,\bR}\subset \LGr(V_{\bR})$, and we can analogously define $\Lambda_{L,\bC}\subset \LGr(V_{\bC})$, where we now include, for every $p\in \partial\wtilde{X}$, the complexified photon $\phi_{\bC}(\xi(p))$ of all Lagrangians containing $\xi(p)$.

Taking complements provides domains of discontinuity $\Omega_{L,\bR}\subset \Omega_{L,\bC}$.
We can then form the quotients
\[
	S(\bR) = \leftquot{\Gamma}{\Omega_{L,\bR}} \quad
	S(\bC) = \leftquot{\Gamma}{\Omega_{L,\bC}}
\]
Each of $S(k)$ is a $3$-dimensional $k$-analytic manifold.
Furthermore, the natural complex conjugation will preserve the domains of discontinuity, commute with the action of $\Gamma$, and the fixed locus is precisely the real domain of discontinuity.

Recall also that one can stratify $\LGr(V_{\bC})$ according to the orbits of $\Sp(V_\bR)$, and this stratification will descend to $S(\bC)$.
The most convenient way to classify the strata is to start by setting
\[
	\cR_{i}:=\{L\in \LGr(V_{\bC})\colon \dim_{\bR}(L\cap \ov{L})=i\}
\]
with $i=0,1,2$.
Clearly each $\cR_i$ is $\Sp(V_\bR)$-invariant.
Next, each $\cR_i$ fibers over $\Gr(i;V_{\bR})$, by assigning to $L$ the intersection $L\cap \ov{L}$.
The fibers are disjoint union of pseudo-Siegel spaces $\cS^{p,q}$ with $p+q=2-i$, where $\cS^{p,q}$ parametrizes complex Lagrangians in $\bC^{2(p+q)}$ of signature $(p,q)$ for the pseudo-hermitian metric.
Let us call this decomposition $\cR_{i}=\coprod_{p,q}\cR_{i,p,q}$.
It is immediate to check that the action of $\Sp(V_\bR)$ on each $\cR_{i,p,q}$ is transitive and complex conjugates takes $\cR_{i,p,q}$ to $\cR_{i,q,p}$.

For instance, the open strata are $\cR_{0,2,0},\cR_{0,0,2}$ and $\cR_{0,1,1}$.
The first two correspond to classical Siegel space and its complex conjugate, while the last one is $\LGr^{1,1}(V_{\bC})$ from \autoref{thm:proper_discontinuity_on_indefinite_lagrangians}.
% (this is the reason why the space $S(\bC)$ has ``Frankenstein'' features)
The ``seams'' are given by the lower dimensional strata, with the deepest one being $\cR_{2,0,0}=\LGr(V_{\bR})$, in the quotient leading to $S(\bR)$.

\begin{question}[Properties of the quotient]
	\label{question:properties_of_the_quotient_Frankenstein}
	Here are some natural questions about this construction.
	\begin{enumerate}
		\item Are there compatible compactifications of $S(k)$ for $k=\bR,\bC$, that turn them into manifolds?
		This is likely, and \autoref{thmintro:rational_directions_on_limit_curve} on the rational directions in the limit set is probably going to be useful.
		\item There are natural volume forms on $S(k)$.
		Is the volume of $S(\bC)$ finite?
		Note that the finiteness of volume of $S(\bR)$ can be seen from the uniformization in \autoref{thmintro:uniformization_of_real_lagrangians} and the identification with the circle bundle over the Riemann surface.
		\item Is there a fibration by complex curves of $S(\bC)$ to a complex surface?
		Note that inside the stratum coming from $\LGr^{1,1}(V_{\bC})$ there are plenty of $\bP^1(\bC)$'s (projected from the Griffiths period domain) but there can't be any $\bP^1(\bC)$'s fully contained in a Siegel component.
		\item Are there any nonconstant meromorphic functions on $S(\bC)$?
		\item There is a tautological rank $2$ holomorphic bundle of Lagrangians on $S(\bC)$.
		Does this bundle, or its determinant line bundle, admit any holomorphic sections?
		\item Instead of $\LGr(V_{\bC})$ consider $\Gr^+(2;W_{\bR})$, the (compact) Grassmannian of all oriented $2$-planes in $W_{\bR}$.
		The highest-dimensional open strata are naturally identified, but the lower-dimensional ones are in general different.
		The closed orbit, i.e. lowest dimensional stratum, is $\LGr(V_{\bR})$ in the original case, and $\bP(V_{\bR})$ in $\Gr(2;W_{\bR})$.
		What are the topological differences between the quotients of the domains of discontinuity?
	\end{enumerate}
\end{question}
An investigation of domains of discontinuity of Anosov representations in complex flag manifolds has been pursued by Dumas--Sanders \cite{DumasSanders_Uniformization-of-compact-complex-manifolds-by-Anosov-representations}, though the situation is somewhat different.

%%					end of subsec: Complex domains of discontinuity
%%=============================================================================

%%=============================================================================
%%					start of subsec: The adjoint representation

\subsection{The adjoint representation}
	\label{ssec:the_adjoint_representation}

\subsubsection{Setup}
	\label{sssec:setup_the_adjoint_representation}
In this section we will study the properties of the monodromy, and the VHS, in the adjoint representation, still under assumption A.
Specifically, set $\bU:=\Lambda^2 \bW$, or equivalently $\bU\isom \Sym^2 \bV$.
The local system $\bU$ is the one obtained by taking the adjoint representation of $\SO_{2,3}$ or $\Sp_4$.
It will be most convenient to work with the interpretation $\bU:=\Lambda^2\bW$.
Then $\bU$ admits a rank $10$ VHS with Hodge numbers $(1,1,2,2,2,1,1)$, which we will view as having weight $8$.

The middle Hodge bundle $\cU^{4,4}$ decomposes into two line bundles $\cW^{3,1}\wedge \cW^{1,3}$ and the one that will be relevant to us:
\[
	\cL:= \cW^{4,0}\wedge \cW^{0,4}.
\]
Set $U_{\bR}:=\Lambda^{2}W_{\bR}$ to be the corresponding vector space.
Then $\cL$ gives a map from the universal cover $\wtilde{X}$ to $\bP(U_{\bR})$, whose image is in the Grassmannian of $2$-planes in $W_{\bR}$.

\subsubsection{Domination structure in adjoint representation}
	\label{sssec:domination_structure_in_adjoint_representation}
By \autoref{thm:domination_in_log_anosov_representations}, since the monodromy representation is $\alpha_1$-log-Anosov, we have a log-dominated decomposition of $\bU$ over $T^1X$, with a poset as described in \autoref{fig:adjoint_rep_poset}.
To introduce notation, denote the dominated decomposition of the original $\bW$ as
\[
	\bW = \bW^{<0} \oplus \bW^{0} \oplus \bW^{>0}
\]
where each of $\bW^{>0},\bW^{<0}$ are isotropic $2$-dimensional subspaces, $\bW^{0}$ is negative-definite $1$-dimensional, and the decomposition is orthogonal for the indefinite metric on $\bW$.
We will also use the notation $\cW^{\leq 0}:=\bW^{<0}\oplus \bW^{0}$.
This notation is slightly different from that in \autoref{sssec:dominated_splittings_of_cocycles}.
On the universal cover, in the standard flat trivialization, the bundle $\bW^{\leq 0}$ gives a subspace $\bW^{\leq 0}(p)\subset W_{\bR}$ which depends only on the boundary point $p\in \partial \wtilde{X}$ corresponding to the projection $T^1\wtilde{X}\to \partial \wtilde{X}$ obtained by following the geodesic to the boundary.

Inside the dominated decomposition of $\bU$, we will need two subbundles:
\begin{align*}
	\bU^{\ll 0}&:= \Lambda^2 \bW^{<0} && \text{ which is a line}\\
	\bU^{\ll top}&:= \left(\bW^{<0}\right)\wedge \bW && \text{ which is codimension-one in }\bU\\
\end{align*}
In terms of \autoref{fig:adjoint_rep_poset}, the first bundle refer to the leftmost, least growing subbundle, while the second refers to the direct sum of all but the rightmost, i.e. fastest growing, subbundle.
On the universal cover this gives a family of subspaces
\begin{align}
	\label{eqn:xi_1_xi_n_adjoint}
	\begin{split}
	\xi_1\colon \partial\wtilde{X}\to \bP(U_{\bR}) \text{ and }
	\xi_{9}\colon \partial\wtilde{X}\to \bP(U_{\bR}^{\dual})
	\end{split}
\end{align}
given by taking the corresponding line and hyperplane.
Note that $\xi_1(p)\subset \xi_9(p)$ and the hyperplane is the orthogonal complement of the (isotropic) line.

\subsubsection{Relation to Hodge decomposition}
	\label{sssec:relation_to_hodge_decomposition}
Recall that $\bW=\oplus \cW^{p,q}(x)$ has a Hodge decomposition, depending on $x\in X$.
Lift it to the universal cover $\wtilde{X}$ to obtain a family of subspaces $\cW^{p,q}(x)\subset W_{\bR}$ for $x\in \wtilde{X}$.
Denote by
\[
	N^e(x):=\left(\cW^{4,0}(x)\oplus \cW^{0,4}(x)\right)\cap W_{\bR}
\]
the real $2$-dimensional negative-definite subspace spanned by the two extreme Hodge bundles.
Let also
\[
	P^e(x):= \left(\cW^{3,1}(x)\oplus\cW^{2,2} \oplus \cW^{1,3}(x)\right)\cap W_{\bR}
\]
be the orthogonal complement of $N^e$ for the indefinite metric.
It is $3$-dimensional of signature $(1,2)$.

A crucial consequence of \autoref{thm:real_uniformization_lagrangian} and \autoref{thm:anosov_property_from_assumption_A} is that $N^e(x)$ is disjoint from $\bW^{\leq 0}(p)$, for any $x\in \wtilde{X}$ and $p\in \partial\wtilde{X}$.
Indeed, suppose that $N^e(x)\cap \bW^{\leq 0}(p)\neq 0$.
Then taking orthogonal complements for the indefinite metric, the same should hold, i.e. $P^e(x)\cap \bW^{<0}(p)\neq 0$.
Take $w\in W_{\bR}$ a nonzero vector in the intersection.
Then the function $f_w:=\norm{w^{0,4}}^2$ would vanish at $x$ by construction, but this would contradict \autoref{thm:real_uniformization_lagrangian} and \autoref{thm:anosov_property_from_assumption_A} since $[w]\in \Lambda_L$	also by construction.

\begin{theorem}[Boundary values in adjoint representation]
	\label{thm:boundary_values_in_adjoint_representation}
	Consider the bundle $\cL$ from \autoref{sssec:setup_the_adjoint_representation}, viewed as a map
	\[
		\cL\colon \wtilde{X}\to \Gr_{0,2}(W_{\bR})\subset \Gr(2;W_{\bR})\subset \bP(U_{\bR})
	\]
	into the Grassmannian of negative-definite $2$-planes in $W_{\bR}$, sitting inside the larger Grassmannian of all $2$-planes in $W_{\bR}$, embedded in $\bP(U_{\bR})$.

	Suppose now that $x_i\in \wtilde{X}$ is any sequence with $x_i\to p\in \partial\wtilde{X}$, with $x_i$ at a bounded above distance from a geodesic connecting $p$ to the basepoint $x_0$, and such that $x_i$ project to the compact part of $X$.
	Then
	\[
		\cL(x_i)\to \xi_1(p) \in \Gr_{(0,0,2)}(W_{\bR})\isom \bP(V_\bR)
	\]
	where $\xi_1(p)$ is defined in \autoref{eqn:xi_1_xi_n_adjoint} and $\Gr_{(0,0,2)}(W_{\bR})$ denotes the Grassmannian of isotropic $2$-planes in $W_{\bR}$.
\end{theorem}
\begin{remark}[On the proof]
	\label{rmk:on_the_proof}
	\leavevmode
	\begin{enumerate}
		\item In fact the result also holds for sequences $x_i$ approaching $p$ not necessarily staying in the compact part (but still a bounded distance from a geodesic).
		To establish this more general result, the lower bound on the (intrinsic) distance between $\cL(x)$ and the limit set needs to be verified for $x$ in the cusp.
		This can be done using the estimates in \autoref{sec:unipotent_dynamics} (which are purely topological) and the asymptotics of the VHS coming from Schmid's $\SL_2$-orbit theorem.
		\item The Grassmannian of isotropic $2$-planes in $W_{\bR}$ is the closed $\SO(W_{\bR})$-orbit inside the Grassmannian of all $2$-planes in $W_{\bR}$.
		It is also isomorphic to $\bP(V_{\bR})$, which is the target for the limit curve from \autoref{thm:anosov_property_from_assumption_A}.
		It is immediate to verify that the notation $\xi_1$ refers in both situations to the same thing.
		The entire discussion could be lifted to the double cover where all real subspaces are oriented.
		\item We could also use the identification of the Grassmannian of negative-definite $2$-planes in $W_{\bR}$ with $\LGr^{1,1}(V_{\bC})$, and then the map $\cL(x)$ is identified with the holomorphic map coming from the term $F^2$ of the Hodge filtration.
		I do not expect that there are limit boundary values in this case, because the minimal orbit in $\LGr(V_{\bC})$ is $\LGr(V_{\bR})$, and the limit set there is $2$-dimensional.
	\end{enumerate}
\end{remark}
\begin{proof}[Proof of \autoref{thm:boundary_values_in_adjoint_representation}]
	We work in the universal cover $\wtilde{X}$, where all bundles are viewed as subspaces of a fixed vector space.
	Recall that we introduced in \autoref{sssec:relation_to_hodge_decomposition} the bundle $N^e(x)=\cW^{4,0}\oplus \cW^{0,4}$, which is defined over $\bR$, and it follows from the constructions that $\cL(x)=\Lambda^2 (N^e(x))$.
	We also established that the $2$-dimensional bundle $N^e(x)$ never intersects the $3$-dimensional one $\bW^{\leq 0}(p)$, for any $x\in \wtilde{X}$ and $p\in \partial\wtilde{X}$.
	It follows that $\cL(x)$ then never intersects $\xi_{9}(p)=\bW^{\leq 0}(p)\wedge \bW$, for any $x\in \wtilde{X}$ and $p\in \partial\wtilde{X}$.

	Now, we can descend this property from the universal cover to $T^1X$ and it follows that for the intrinsic Hodge metric on $\bU$, for $u\in T^1 X$ in a compact set $K$, there exists $\ve(K)>0$ such that $\dist(\cL(u),\xi_9(u))\geq \ve(K)$.
	Note that $\cL(u)$ depends only on the projection of $u$ to $X$, while $\xi_9(u)$ changes as $u$ changes in the fiber over a fixed $x\in X$.

	We can now apply the domination estimates on the bundles.
	Pick $x\in \wtilde{X}$ projecting to a fixed compact part of $X$, and $u_x\in T^1\wtilde{X}$ projecting to $x$, such that $u_x g_T=u_0$ where $u_0 \in T^1_{x_0}\wtilde{X}$, i.e. flowing from $x$, in the direction $u_x$, for time $T>0$, brings us to the basepoint $x_0$.
	Let also $p\in \partial{\wtilde{X}}$ be the boundary point obtained by flowing from $u_x$ to time $-\infty$, and similarly $p'\in \partial\wtilde{X}$ obtained by flowing to $+\infty$ (so $x_0$ is on the geodesic between $x$ and $p'$).
	For the statement of the theorem, it suffices to show that
	\[
		\dist_{x_0}(\cL(x),\xi_1(p)) \leq C e^{-\ve'\cdot dist(x_0,x)}= C e^{-\ve'\cdot T}
	\]
	for some constants $C,\ve'>0$, where $\dist_{x_0}$ refers to the distance in projective space using the Hodge metric at $x_0$.
	However, this follows from the domination estimates, applied to a nonzero vector $l\in \cL(x)$, decomposed as $l=l_1+l_9$ with $l_1\in \xi_1(p)$ and $l_9\in \xi_9(p')$.
	By uniform transversality of $\cL$ and $\xi_9$ we have that $\norm{l_1}_x\geq \ve \norm{l_9}_x$.
	Now apply the geodesic flow, which acts as the identity in the trivialization on the universal cover using the Gauss--Manin connection.
	The domination property gives $\norm{l_1}_{u_x g_T}\geq \frac{\ve}{C}e^{-\ve' \cdot T}\norm{l_9}_{u_x g_T}$.
	Recalling that $u_xg_T = u_0$ and projects to $x_0$, the stated distance estimate follows.
\end{proof}

%%					end of subsec: The adjoint representation
%%=============================================================================

%%=============================================================================
%%					start of subsec: Maximal hypergeometric monodromy

\subsection{Maximal hypergeometric monodromy}
	\label{ssec:maximal_hypergeometric_monodromy}

This section is devoted to a brief discussion of a natural variant of ``assumption A'', which leads to maximal representations in the sense of \cite{BurgerIozziWienhard2010_Surface-group-representations-with-maximal-Toledo-invariant}.
We will be brief and only illustrate the salient points, and a detailed study (including the higher rank situation) will appear in a separate text.
Additionally, we classify the rank $5$ hypergeometric equations with Hodge numbers $(1,1,1,1,1)$ and maximal monodromy.

Consider the following analogue of assumption A from \autoref{def:assumption_a}, which we shall call ``assumption B''.
Let $\bW$ be the rank $5$ variation of Hodge structure with Hodge numbers $(1,1,1,1,1)$.
Suppose also that there exists a rank $4$ VHS $\bV$, such that $\bW=\Lambda^2_{\circ}\bW$ as before.
We keep the notation of \autoref{def:assumption_a}.

\begin{definition}[Assumption B]
	\label{def:assumption_b}
	The VHS $\bW$ satisfies \emph{assumption B} if the second fundamental form
	\[
		\sigma_{2,2}\colon \cW^{2,2}_{ext}\to \cW^{1,3}_{ext}\otimes \cK_{\ov{X}}(\log D)
	\]
	is an isomorphism.
	Analogously $\bV$ satisfies assumption B if the second fundamental form
	\[
		\sigma_{3,0}\colon \cV^{3,0}_{ext}\to \cV^{2,1}_{ext}\otimes \cK_{\ov{X}}(\log D)
	\]
	is an isomorphism.
\end{definition}

\begin{remark}[On assumption B]
	\label{rmk:on_assumption_b}
	\leavevmode
	\begin{enumerate}
		\item It is immediate that if $\bV$ satisfies both assumption A and assumption B, then it is the third symmetric power of the uniformizing representation of the base Riemann surface.
		Indeed, under these assumptions we find $\cV^{p,q}\isom \cK_{\ov{X}}(\log D)^{\otimes (p-q)/2}$ and all the second fundamental forms are isomorphisms.
		Therefore, by the Simpson correspondence, the claim follows.
		\item Since $\cW^{2,2}$ is a trivial holomorphic line bundle, assumption B is equivalent to the isomorphism $\cW^{3,1}_{ext}\isom \cK_{\ov{X}}(\log D)$.
		The Toledo number of the monodromy representation is the same as the degree of $\cW^{3,1}$, so assumption $B$ is equivalent to the maximality of the monodromy (see \cite[Prop.~2.4]{CollierTholozanToulisse2019_The-geometry-of-maximal-representations-of-surface-groups-into} for the justifications in the more general setting of Higgs bundles).
	\end{enumerate}
\end{remark}

\subsubsection{Boundary maps, proper discontinuity}
	\label{sssec:boundary_maps_proper_discontinuity_maximal_case}
Let $\rho_V\colon \pi_1(X)\to \Sp(V_{\bR})$ the monodromy representation of a VHS satisfying assumption B, and let $\rho_W$ be the corresponding representation to $\SO(W_{\bR})$.
Since the Lagrangian Grassmannian $\LGr(V_{\bR})$ is contained in $\bP\left(W_{\bR}\right)$, only $\rho_W$ is needed for the discussion below, which we can write also as $\Ein^{1,2}(W_{\bR})$ to denote the quadric of null vectors.

It follows from \cite[Thm.~6.1]{BurgerIozziLabourie2005_Maximal-representations-of-surface-groups:-symplectic-Anosov-structures} that there exists a dynamics-preserving, $\rho_W$-equivariant map
\[
	\xi\colon \partial\wtilde{X}\to \Ein^{1,2}(W_{\bR})\isom \LGr(V_{\bR})
\]
in the sense of \autoref{thm:existence_of_boundary_maps}.
Furthermore, there is a domain of discontinuity $\Omega_P\subset \bP\left(V_{\bR}\right)$ (where in the absence of a symplectic lift of $\rho_W$, one can use the flag manifold of isotropic $2$-planes for $\SO(W_{\bR})$ in place of $\bP(V_{\bR})$).

One of the main results of Collier, Tholozan, and Toulisse \cite[Thm.~1]{CollierTholozanToulisse2019_The-geometry-of-maximal-representations-of-surface-groups-into} provides an identification of the quotient $\leftquot{\Gamma}{\Omega_P}$ with the circle bundle associated to $\cW^{1,3}$ (i.e. the canonical bundle).
This is analogous to \autoref{thm:real_uniformization_lagrangian}.

\subsubsection{Lyapunov exponent}
	\label{sssec:lyapunov_exponent_maximal_case}
In the case of maximal representations, using the aforementioned uniformization result from \cite{CollierTholozanToulisse2019_The-geometry-of-maximal-representations-of-surface-groups-into}, one can again obtain an identity involving Lyapunov exponents.
However, in this case the ``bad locus'' is empty for vectors in $\bP(V_{\bR})$, so we find:
\begin{align*}
	\lambda_1(\bV) &= \frac{\deg \cV^{0,3}}{\chi(X)}\\
	\lambda_1(\bW) + \lambda_2(\bW) &= \frac{\deg \cW^{1,3}+\deg \cW^{0,4}}{\chi(X)}.
\end{align*}
Note that the numbers in the second line are equal to twice those in the first.

\subsubsection{Classification of hypergeometric examples}
	\label{sssec:classification_of_hypergeometric_examples}
In the case of assumption A the hypergeometric system was $\bV$, whereas now we shall look for hypergeometric systems $\bW$ satisfying assumption B.
Following the same analysis as in \autoref{prop:assumption_a_in_the_local_case}, it is possible to classify the local exponents at a singularity of a hypergeometric equation which ensure that assumption B is satisfied.
Crucially, note that the singularity at $1\in \bP^1(\bC)$, always satisfies assumption B for rank 5 hypergeometrics, and never satisfies assumption B for rank 4 hypergeometrics.

It is straightforward to list the diagrams as in \autoref{fig:Hodge_numbers} corresponding to the possible degenerations in the rank 5 case, and we remark that they are in fact in bijection with the diagrams in the rank 4 case (by the correspondence between $\bV$ and $\bW$).
Keeping the terminology of case 1, case 2, and case MUM, assumption B is satisfied if we are in case 2 or case MUM (whereas assumption A needed case 1 or case MUM).
Note also that in the $\bW$-representation, case 1 corresponds to a rank $1$ unipotent with an invariant isotropic $2$-plane, while case 2 corresponds to a rank $1$ unipotent with an invariant isotropic line.

Here is now the list of local exponents which satisfy assumption B:
\begin{align*}
	\left(
	\mu,
	\frac{1}{2},
	\frac{1}{2},
	\frac{1}{2},
	1-\mu
	\right)
	&& \mu\in\left(0,\frac{1}{2}\right]\\
	\left(
	\frac{N-k}{2N},\frac{N-1}{2N},
	\frac{1}{2},
	\frac{N+1}{2N},\frac{N+k}{2N}
	\right) && 1 < k < N, \text{ orbifold order }2N\\
	\left(
	0,0,0,
	\frac{N}{2N+1},
	\frac{N+1}{2N+1}
	\right)
	&& \text{ orbifold order }2N+1\\
	\left(
	0, \frac{k}{N},
	\frac{k+1}{N},
	\frac{N-(k+1)}{N},
	\frac{N-k}{N}	
	\right)
	&& 1< 2(k+1) < N \text{ orbifold order }N
\end{align*}
In the first row $\mu\in \bR$ and in the last three rows $k,N\in \bN$.
Let us note that in the course of doing the above classification, following the analysis analogous to that in \autoref{prop:assumption_a_in_the_local_case}, one must exclude local exponents such as $(0,0,0,0,\tfrac 12)$ (and more generally those that have a $0$ and a $\frac{1}{2}$) since this does not allow for the correct interlacing pattern with the exponents at the other singularity to lead to Hodge numbers $(1,1,1,1,1)$ (recall also that the exponents must be symmetric under $x\mapsto 1-x \mod 1$).

A table of allowed possibilities, and required restrictions, on pairings between the above local choices is included in \autoref{table:maximal_hypgeom}.
Note that it is not allowed to have two singularities with one exponent coinciding (the hypergeometric monodromy would be reducible), so the only allowed pairings are to select one from the first two possibilities above, and another from the subsequent two.

%%					end of subsec: Maximal hypergeometric monodromy
%%=============================================================================

%%%%%%%%%%%%%%%%%%%%%%%%%%%%%%%%%%%%%%%%%%%%%%%%%%%%%%%%%%%%%%%%%%%%%%%%%%%%%%%
%%% 				End of Section: Hodge theory and Anosov representations
%%%%%%%%%%%%%%%%%%%%%%%%%%%%%%%%%%%%%%%%%%%%%%%%%%%%%%%%%%%%%%%%%%%%%%%%%%%%%%%

\appendix

%%%%%%%%%%%%%%%%%%%%%%%%%%%%%%%%%%%%%%%%%%%%%%%%%%%%%%%%%%%%%%%%%%%%%%%%%%%%%%%
% Start of Section: Unipotent dynamics
%%%%%%%%%%%%%%%%%%%%%%%%%%%%%%%%%%%%%%%%%%%%%%%%%%%%%%%%%%%%%%%%%%%%%%%%%%%%%%%

\section{Unipotent dynamics}
	\label{sec:unipotent_dynamics}

We develop a criterion in \autoref{thm:log_domination_in_the_cusps} showing that when the metric on the cocycle is adapted, the expected exponential growth occurs even when the geodesic goes into the cusp.
Adapted metrics are introduced in \autoref{def:adapted_metric}.
While no Hodge theory is needed for this section, the basic tools, such as the monodromy weight filtration, as well as adapted metrics, originate in Schmid's nilpotent and $\SL_2$-orbit theorems.
A useful notion is that of ``log-proximal'' unipotent transformation, generalizing that of proximal semisimple element.

%%=============================================================================
%%					start of subsec: Setup

\subsection{Setup}
	\label{ssec:setup_unipotent_matrices}

Assume that $\rho\colon \pi_1(X)\to G$ is a $\theta$-log-dominated representation and $\phi\colon G\to \GL(V)$ is a representation of $G$.
Let $\bV\to T^1X$ be the associated local system, which is $\Theta$-dominated in the compact part (in the sense of \autoref{sssec:domination_in_the_compact_part}) by the construction in \autoref{thm:domination_in_log_anosov_representations}.
In this section we will prove that the domination condition holds at all positive times, when the starting point is in a compact set, and if the metric is adapted in the sense of \autoref{def:adapted_metric} below.

\begin{theorem}[log-domination in the cusps]
	\label{thm:log_domination_in_the_cusps}
	If the cocycle $\bV\to T^1X$ is equipped with an adapted metric, then the conclusion of \autoref{thm:domination_in_log_anosov_representations} holds for starting points in a compact set, and all times $t\geq 0$.
\end{theorem}

\subsubsection{Reduction to proximal case}
	\label{sssec:reduction_to_proximal_case}
In fact, the result above is a consequence of \autoref{thm:adapted_metrics_give_domination_in_the_cusp} below, which treats a particular case, when domination is established for a rank $n$ bundle $V$ with a $1$-dimensional bundle dominating an $(n-1)$-dimensional bundle.
Let us explain why this suffices.

First, \autoref{thm:log_domination_in_the_cusps} clearly reduces to the case of irreducible $V$.
Next, it suffices to establish the singular value gap intrinsically in the group, as per \autoref{def:log_anosov_representation}, but using points in the cusp along a geodesic, and the adapted metric in the cusp.
Finally, this singular value gap can be established by considering, for each $\alpha\in \theta$, a corresponding irreducible representation as in \cite[Prop.~8.3]{BochiPotrieSambarino2019_Anosov-representations-and-dominated-splittings} on which there will be a totally dominating $1$-dimensional bundle.

\subsubsection{Notation for log-proximal case}
	\label{sssec:notation_for_log_proximal_case}
We assume from now on that $\rho\colon \pi_1(X)\to \GL(V)$ is a log-dominated representation with a decomposition into bundles of dimension $1,n-2$ and $1$, dominating each other in the order they are written.
Taking the direct sum of the last two, and the last one, gives subbundles that only depend on the future trajectory, i.e. are invariant under the stable foliation, and so give maps
\begin{align*}
	\xi_{n-1}&\colon \partial\wtilde{X}\to \bP\left(V^{\dual}\right)\\
	\xi_{1}&\colon \partial\wtilde{X}\to \bP\left(V\right)\\
\end{align*}
where $\bP(V^{\dual})$ parametrizes hyperplanes in $V$.
The transversality properties are that $\xi_1(p)\subset \xi_{n-1}(p)$ and for $p'\neq p$ we have $V = \xi_1(p)\oplus \left[\xi_{n-1}(p)\cap\xi_{n-1}(p')\right]\oplus \xi_1(p')$.
We will use this notation starting with \autoref{ssec:adapted_metrics_and_domination} below.

\subsubsection{Unipotent preliminaries}
	\label{sssec:unipotent_preliminaries}
We now start to collect some useful facts about the dynamics of unipotent matrices.
Fix for the rest of this paragraph a unipotent matrix $T\in G\subset \GL(V)$ where $G$ denotes the real points of a semisimple real-algebraic group, and $N:=\log T \in \frakg \subset \frakgl(V)$.

\subsubsection{$\fraksl_2$-triples and Jacobson--Morozov}
	\label{sssec:sl_2_triples_Jacobson_Morozov}
For proofs of the following facts, see \cite[\S9.2]{nilpotent_book}.
First, there exist $Y,N_+\in \frakg$ such that, denoting $N_-:=N$, we have that $Y,N_-,N_+$ form an $\fraksl_2$-triple:
\begin{align}
	\label{eqn:sl2_relations}
	[Y,N_\pm] = (\pm 2)N_\pm \qquad [N_+,N_-] = Y
\end{align}
Furthermore, for any other $Y',N_+'$ which extend $N$ to an $\fraksl_2$-triple, there exists $g\in Z_{G}(N)$, i.e. an element of $G$ which commutes with $N$, which conjugates the old triple to the new one.
For example, in $\fraksl_2$ we could take
\[
	Y =
	\begin{bmatrix}
		-1 & 0\\
		0 & 1
	\end{bmatrix}
	N_- =
	\begin{bmatrix}
		0 & 1\\
		0 & 0
	\end{bmatrix}
	N_+ =
	\begin{bmatrix}
		0 & 0\\
		1 & 0
	\end{bmatrix}
\]
Note the signs of the entries of $Y$, which are relevant for the choices below.

\subsubsection{Weight filtration}
	\label{sssec:weight_filtration}
For the proofs of the following results, see \cite[Lemma 6.4]{Schmid_VHS}.
Associated to $N$ there is a canonical \emph{weight filtration}:
\[
	\left\lbrace 0 \right\rbrace \subsetneq W_{-d}(N) \subset W_{-(d-1)}(N) \subset \cdots \subset W_{d-1} (N) \subsetneq W_d (N) = V
\]
where $d$ is determined from the condition that $N^d\neq 0 $ but $N^{d+1}=0$.
The integer $d$ will be called \emph{the order} of $N$ (or $T$).
The defining properties of the weight filtration are that $N(W_i)\subset W_{i-2}$ and the induced map $N^i:\rightquot{W_i }{W_{i-1}}\to \rightquot{ W_{-i}}{W_{-i-1}}$ is an isomorphism for all $i=1\ldots d$.
When $N$ is clear from the context, we write $W_i$ or $W_i V$ instead of $W_i(N)$.

To relate with the discussion in \autoref{sssec:sl_2_triples_Jacobson_Morozov}, any element $Y$ in an $\fraksl_2$-triple provides a splitting of the weight filtration, using the eigenspaces of $Y$.

\begin{definition}[Log-proximal unipotent transformation]
	\label{def:log_proximal_unipotent_transformation}
	The map $T$ is \emph{log-proximal} if in the weight filtration associated to $N:=\log T$, we have $\codim W_{d-1}(N)=1$, or equivalently $\dim W_{-(d-1)}(N)=1$.
	The pair $(W_{-(d-1)},W_{d-1})$ will be called the log-proximal filtration of $T$.
\end{definition}
This concept is analogous to proximality for semisimple transformations, i.e. having the largest eigenvalue of multiplicity one.
We will call $W_{-(d-1)}$ the \emph{attracting line} of $T$, and $W_{d-1}$ its \emph{repelling hyperplane}.
Note that $T$ and $T^{-1}$ have the same log-proximal filtration.

Assume from now on that $T$ is log-proximal and fix an $\fraksl_2$-triple $Y,N_+,N$.
This gives a direct sum decomposition
\[
	V = \bigoplus_{i=-d}^d V_i \quad
	\text{ with }Y\vert_{V_i}=i\id \text{ and }N(V_i)\subset V_{i-2}.
\]
Moreover we know that $N^d\colon V_d \toisom V_{-d}$ is an isomorphism.

Fix a metric on $V$ such that the above direct sum decomposition is orthogonal, and $N$ is an isometric embedding $V_i/\ker(N\vert_{V_i})\to V_{i-2}$.
In particular $N^d$ is an isometry.

\subsubsection{Hyperbolic space calculations}
	\label{sssec:hyperbolic_space_calculations}
We will need some elementary estimates for distances in the hyperbolic plane.
The identification of the half-plane model with the unit disk, for coordinates $z$ in the unit disk and $\tau$ in the upper half-plane, is via the transformations:
\begin{align}
	\label{eqn:half_plane_to_disk_and_back}
	\tau = \frac{1}{\sqrt{-1}}\frac{z+1}{z-1} \quad
	z = \frac{\tau - \sqrt{-1}}{\tau+\sqrt{-1}}
\end{align}
which are inverses of each other, identify $z=0$ with $\tau=\sqrt{-1}$, $z=1$ with $\tau=\infty$, and the axes $\Im z = 0$ with $\Re \tau = 0$.
It will be useful to note that on the boundary:
\begin{align}
	\label{eqn:disk_half_plane_mapping_boundary}
	z=e^{\sqrt{-1}\theta} \mapsto \tau= \frac{-\sin \theta}{1-\cos \theta} = -\frac{1+O(\theta^2)}{\theta + O(\theta^3)} = -\frac{1}{\theta}(1+O(\theta^2))
\end{align}
for $\theta$ in a small but fixed neighborhood of $0$.

Recall that in the disk model, the hyperbolic distance between the origin and $z$ is given by
\[
	\dist(0,z) = \log \left( \frac{1+|z|}{1-|z|} \right) = -\log (1-|z|^2) + O(1)
\]
Computing now the same quantities in terms of $\tau$ gives:
\begin{align*}
	z\cdot \ov{z} & = \frac{|\tau|^2+1 - 2\Im \tau}{|\tau|^2+1 + 2\Im \tau}\\
	-\log \left( 1-|z|^2 \right) & = \log (|\tau|^2+1 + 2\Im \tau) - \log (4\Im \tau)
\end{align*}
from which we will use the approximation
\begin{align}
	\label{eqn:hyperbolic_distance_tau_approximation}
	\dist(\sqrt{-1},\tau) = 2\log(\max(|\Re \tau|, \Im \tau)) - \log (\Im \tau) + O(1)
\end{align}
valid in the upper half-plane when $\Im \tau \geq 1$.
To deduce this last approximation, use that $\log(a+b)=\log \max(a,b) + O(1)$ for positive reals (which is clear by exponentiation).

%%					end of subsec: Setup
%%=============================================================================

%%=============================================================================
%%					start of subsec: Adapted metrics

\subsection{Adapted metrics}
	\label{ssec:adapted_metrics}

We keep the notation on Riemann surfaces and representations introduced in \autoref{ssec:setup_notation_and_conventions_for_surfaces} and onwards.

\subsubsection{Conventions for the metric on the vector bundle}
	\label{sssec:conventions_for_the_metric_on_the_vector_bundle}
To give a metric on the vector bundle $\bV_\rho$ is the same as to give a family of metrics $h(x)$ on the vector space $V$, depending on $x\in \widetilde{X}$, with the equivariance property
\[
	\norm{v}_{h(x)} = \norm{\rho(\gamma) v}_{h(\gamma x)} \text{ or equivalently }\norm{v}_{h(\gamma x)} = \norm{\rho(\gamma)^{-1}v}_{h(x)}.
\]
When the representation $\rho$ takes values in $G\subset \GL(V)$, the metric on $V_\rho$ should be adapted to the $G$-action, so it is required to be of the form:
\begin{align}
	\label{eqn:G_metric_universal_cover}
	h\colon \widetilde{X}\to \rightquot{G}{K(x_0)} \quad h(\gamma x) = \rho(\gamma) h(x) \quad \forall \gamma \in \pi_1(X,x_0)
\end{align}
where $K(x_0)\subset G$ denotes the maximal compact subgroup of $G$ determined by the metric at $x_0$.
With the above notation, the norm of a vector is given by
\begin{align}
	\label{eqn:norm_of_vector_hx}
	\norm{v}_{h(x)} := \norm{h(x)^{-1}\cdot v}_{h(x_0)}
\end{align}

For the next discussion, let $\bH$ denote the upper half-plane, with coordinate $\tau$ and basepoint $\sqrt{-1}\in \bH$.
Fix also a reference metric at the basepoint and let $K$ denote the corresponding maximal compact.

\begin{definition}[Strictly adapted metric]
	\label{def:strictly_adapted_metric}
	Let $N\in \frakg \subset \frakgl(V)$ be a nilpotent transformation.
	A family of metrics $h\colon \bH \to G/K$ on $V$ is called \emph{strictly adapted to $N$} if there exists $Y\in \frakg$ such that $(Y,N)$ can be extended to an $\fraksl_2$-triple, and $h$ is defined by
	\[
		h(\tau) = e^{(\Re \tau) N} \cdot e^{-\frac{1}{2}\log(\Im \tau)\cdot Y}\cdot h(\sqrt{-1})
	\]
\end{definition}
\begin{remark}
	\leavevmode
	\begin{enumerate}
		\item A strictly adapted metric satisfies $h(\tau+1)=e^{N}h(\tau)$ so it descends to a metric on the vector bundle over the punctured disk $\rightquot{\bH}{(\tau\mapsto \tau+1)}$, with monodromy $e^{N}$ around the puncture.
		\item An equivalent definition of strictly adapted metrics is that there exists a homomorphism $\SL_2(\bR)\to G$, such that at the Lie algebra level we have $\left[\begin{smallmatrix}
			0 & 1 \\
			0 & 0
		\end{smallmatrix}\right]\mapsto N$, and such that the metric $h\colon \SL_2(\bR)/\SO_2(\bR)\to G/K$ is equivariant for the group homomorphism.
		Indeed we can use the $NAK$ decomposition of $\SL_{2}(\bR)$ to write any element modulo $\SO_2(\bR)$ as
		\[
			g = \begin{bmatrix}
				y^{1/2} & y^{-1/2}x \\
				0 & y^{-1/2}
			\end{bmatrix}
			=
			\begin{bmatrix}
				1 & x \\
				0 & 1
			\end{bmatrix}\cdot 
			\begin{bmatrix}
				y^{1/2} & 0 \\
				0 & y^{-1/2}
			\end{bmatrix}
		\]
		leading to the formula in \autoref{def:strictly_adapted_metric}.
		Note the sign of $Y$, compatible with the conventions in \autoref{sssec:sl_2_triples_Jacobson_Morozov}, and that with $g$ as in the last displayed equation, we have $g(\sqrt{-1})=x+\sqrt{-1}y$.
		\item A notion equivalent to adapted metrics was introduced independently by Canary--Zhang--Zimmer \cite{CanaryZhangZimmer2021_Cusped-Hitchin-representations-and-Anosov-representations-of-geometrically-finite} under the name ``canonical family of norms''.
	\end{enumerate}
\end{remark}

\begin{definition}[Adapted metric]
	\label{def:adapted_metric}
	A metric $h$ on $\bV_\rho\to X$, when $\bV_\rho$ has unipotent monodromy around the cusps, is \emph{adapted} if the following holds around each cusp.
	Fix a standard cusp neighborhood as in \autoref{sssec:standard_cusp_neighborhood} and lift the metric $h$ to a map $h\colon \left\lbrace \Im \tau \geq A \right\rbrace \to G/K $.
	Let $T$ denote the monodromy of $\bV_\rho$ around the cusp and $N:=\log T\in \frakg$.
	Then there exists metric $h'$ on $\bH$ which is strictly adapted to $N$, and a constant $C>0$ such that $\dist(h(\tau),h'(\tau))\leq C$ whenever $\Im \tau \geq A$, with distance measured in $G/K$.
	Equivalently, the metrics $h,h'$ are uniformly comparable, so up to uniform constants any inequality valid for $h$ will also hold for $h'$.
\end{definition}

\begin{theorem}[Hodge metrics are adapted]
	\label{thm:hodge_metrics_are_adapted}
	Suppose that the vector bundle $\bV_\rho$ admits a variation of Hodge structures and is equipped with the corresponding Hodge metric.
	Then this metric is adapted in the sense of \autoref{def:adapted_metric}.
\end{theorem}

\begin{proof}
	This follows from Schmid's $\SL_2$-orbit theorem \cite[5.13]{Schmid_VHS}, which in fact gives a stronger statement.

	In fact, he shows that there exists a strictly adapted metric $h'$ (\autoref{def:strictly_adapted_metric}) such that if $h$ denotes the Hodge metric in the horoball in the upper half-space model, we have
	\[
		\dist(h(\tau),h'(\tau))\leq \frac{C}{(\Im \tau)^{1/2}}
	\]
	for some $C>0$ and $\Im\tau>C$, so the metrics are in fact asymptotic.
\end{proof}

\begin{remark}[Topological nature of growth]
	\label{rmk:topological_nature_of_growth}
	It is interesting to remark that the nature of growth of sections in the cusp depends only on the conjugacy class of the unipotent matrix, and not on the particular Hodge numbers.
	In particular, it is always possible to take a direct sum of symmetric powers of the standard degeneration for $\SL_2$, and possibly also some trivial nondegenerating factors, to model the metric behavior.

	Indeed, suppose that $\bV_1,\bV_2$ are two variations of Hodge structure over the punctured unit disc $\Delta^{\times}$, of the same rank and ``with the same monodromy'', i.e. there exists an isomorphism of local systems $\iota\colon \bV_1\isom \bV_2$.
	As a section of $\Hom(\bV_1,\bV_2)$, $\iota$ is invariant by the monodromy and hence has uniformly bounded Hodge norm (on a smaller punctured disc), as a consequence of Schmid's $\SL_2$-orbit theorem (see e.g. \cite[Cor.~6.7]{Schmid_VHS}).
	Therefore, given a vector $v\in\bV_1$, the Hodge norms of $v$ and $\iota(v)\in \bV_2$ are comparable with uniform constants, and all coarse-geometric properties near the cusp can be investigated in either of the two variations.
\end{remark}

For future use, we collect some inequalities regarding the growth rate of vectors relative to strictly adapted metrics.
Because adapted and strictly adapted metrics are comparable in the cusp, the inequalities below are also valid for adapted metrics, after perhaps changing the constants.

Recall that the elements $(Y,N)$ which are part of an $\fraksl_2$-triple are fixed, and we have a decomposition $V= V_{-d} \oplus \cdots \oplus V_{d}$, and strictly adapted metrics $h(\tau)$ for $\tau\in \bH$.
We only consider parameters with $\Im \tau \geq A$, for some fixed uniform $A$.

\begin{proposition}[Estimates for strictly adapted metrics]
	\label{prop:estimates_for_strictly_adapted_metrics}
	\leavevmode
	\begin{enumerate}
		\item Suppose that $v\in V_k$ is a unit vector and set $l=\floor{\frac{k+d}{2}}$.
		Then we have the estimate
		\[
			1 \leqapprox
			\frac{\norm{v}_{h(\tau)}}{(\Im\tau)^{\frac k2}
			\left[1+ \left(\frac{|\Re \tau|}{(\Im \tau)^{1/2}}\right)^{l}\right] }
			\leqapprox 1
		\]
		\item In particular, if $w\in W_{d-1}$ us a unit vector then
		\begin{align}
			\label{eqn:W_d_1_growth_estimate_new}
			\norm{w}_{h(\tau)} & \leqapprox \Im (\tau)^{\frac{d-1}{2}}
			\left[1 + 
			\left(
			\frac{|\Re\tau|}{\left(\Im \tau\right)^{1/2}}
			\right)^{d-1}\right]
		\end{align}
	\end{enumerate}
\end{proposition}
\noindent We will also make use of the case $v\in V_d$, in which case $k=l=d$ (these are the vectors which grow fastest).

	% Let $w\in W_{d-1}(N)=V_{d-1}\oplus\cdots V_{-d}$ and $v\in V_d$ be unit vectors for the metric $h(\sqrt{-1})$.
	% Then:
	% \begin{align}
	% 	\label{eqn:W_d_1_growth_estimate}
	% 	\norm{w}_{h(\tau)} & \leqapprox \Im (\tau)^{\frac{d-1}{2}} + \Im(\tau)^{-\frac{d-1}{2}}\cdot |\Re \tau|^{d-1}
	% \intertext{and similarly}
	% 	\label{eqn:W_d_growth_estimate}
	% 	\Im (\tau)^{\frac{d}{2}} + \Im(\tau)^{-\frac{d}{2}} \cdot |\Re\tau|^{d}
	% 	 & \leqapprox
	% 	\norm{v}_{h(\tau)} \leqapprox \Im (\tau)^{\frac{d}{2}} + \Im(\tau)^{-\frac{d}{2}} \cdot |\Re\tau|^{d}
	% \end{align}

% Note that for $w$ we only need an upper bound, whereas for $v$ both the lower and upper bounds will be useful.
% The proof will use several times the following estimate: a sum of terms in a geometric progression of positive reals $a+ar + \cdots + ar^d$ can be estimated below by $a+ar^d$, and above by $O(1)(a+ar^d)$ (where $d$ will be fixed throughout).
% Similarly, if we have a sum of $d+1$ terms, each of which is $O(ar^i)$ for $i=0\ldots d$, then the whole sum is $O(1)(a+ar^d)$.
% In the case of interest, we will have $a=\Im(\tau)^{d/2}$ and $r=\frac{\Re(\tau)}{\Im(\tau)}$

\begin{proof}
% [Proof of \autoref{prop:estimates_for_strictly_adapted_metrics}]
	Recall that by equivariance:
	\[
		\norm{u}_{h(\tau)}=\norm{h(\tau)^{-1}u}_{h(\sqrt{-1})}=
		\norm{e^{\frac{1}{2}\log \Im(\tau)Y}\cdot  e^{-\Re(\tau)N} \cdot u}_{h(\sqrt{-1})} 
	\]
	The second estimate follows from the first by bounding each of the coordinates of $w$ in the decomposition $\oplus V_i$.
	We now compute.
	First
	\begin{align*}
		e^{-\Re \tau N}v & = v + (-\Re \tau)(Nv)+\cdots + \frac{(-\Re \tau)^l}{l!}\left(N^l v\right)
		\intertext{and next}
		e^{\tfrac 12 \left(\log \Im \tau\right) Y}e^{-\Re \tau N}v & = 
		(\Im \tau)^{\tfrac k2}v + \left(\Im\tau\right)^{\tfrac {k-1}{2}} (-\Re \tau)(Nv)+\cdots \\
		& \cdots + 
		(\Im \tau)^{\tfrac{k-l}{2}}
		\cdot
		\frac{(-\Re \tau)^l}{l!}
		\cdot \left(N^l v\right)
	\end{align*}
	Now the summation above is over orthogonal vectors, each of norm $O(1)$.
	Furthermore the scaling factors are in geometric progression, so we can estimate the whole sum, up to a constant factor, by the extreme two terms.
\end{proof}

\subsubsection{Growth of vectors under unipotent iterates}
	\label{sssec:growth_of_vectors_under_unipotent_iterates}
To analyze the behavior of $T^k$ for large $k$, fix a unit vector $v$ with components $v=\oplus v_i$, and compute directly:
\begin{align}
	\label{eqn:power_of_T_explicit}
	\begin{split}
	T^k & = e^{kN} = 1 + kN + \cdots + \frac{k^d }{d!}N^d\\
	\norm{T^kv} & = \frac{|k|^d}{d!}\norm{v_d} + O(|k|^{d-1})
	\end{split}
\end{align}
where the implied constants depend only on $d$, i.e. there exists $c(d)$ such that $\left\vert \norm{T^kv}-\frac{|k|^d}{d!}\norm{v_d} \right\vert\leq c(d)|k|^{d-1}$.
The estimate is valid for all $k\in \bZ\setminus \left\lbrace 0 \right\rbrace$, and could have been alternatively deduced from \autoref{prop:estimates_for_strictly_adapted_metrics}, taking $\tau = \sqrt{-1}+k$.

\begin{proposition}[Control of planes under unipotent domination]
	\label{prop:control_of_planes_under_unipotent_domination}
	Suppose that $\xi_1 \oplus \xi_{n-1}=V$ is a direct sum decomposition of $V$ into a line and hyperplane.
	% with $\dist(\xi_1, W_{d-1}(N))\geq \ve>0$.
	Suppose that there exist constants $t,k>0$ such that
	\[
		\log \norm{T^kv} - \log \norm {T^kw} \geq t
	\]
	for all unit vectors $v\in \xi_1,w\in \xi_{n-1}$.
	Then we have
	\[
		\dist(\xi_{n-1}, W_{d-1}(N)) \leqapprox \max \left( e^{-\log |k|}, e^{-t} \right)
	\]
	% where $c$ is a some universal constant depending only on the dimensions.
\end{proposition}
\begin{proof}
	We will write the assumption in the form
	\[
		\norm{T^k v} \geq e^{t} \norm{T^k w}
	\]
	assume that $k>0$ for simplicity of notation, and use \autoref{eqn:power_of_T_explicit} to estimate:
	\begin{align*}
		\norm{T^k v}& \leq \frac{k^d}{d!} + O(k^{d-1})\\
		\norm{T^k w}& \geq \frac{k^d}{d!}\norm{w_d} + O(k^{d-1})\\
	\end{align*}
	where $w = \oplus w_i$ is the decomposition with respect to $V=\oplus V_i$.
	Putting the assumptions together with the estimates gives the chain of manipulations:
	\begin{align*}
		\frac{k^d}{d!} + O(k^{d-1})& \geq e^t\left( \frac{k^d}{d!}\norm{w_d} + O(k^{d-1}) \right) \\
		\frac{k}{d!} + O(1)& \geq e^t\left( \frac{k}{d!}\norm{w_d} + O(1) \right) \\
		k + O(1)& \geq e^t\left( k\norm{w_d} + O(1) \right) \\
		ke^{-t} + O(e^{-t})& \geq k\norm{w_d} + O(1) \\
		k\norm{w_d} & \leq ke^{-t} + O(e^{-t}) + O(1) \\
		 \norm{w_d} & \leq e^{-t} + O(e^{-t}e^{-\log k}) + O(e^{-\log k}) \\
		 \norm{w_d} & \leq c(d)\max(e^{-\log k}, e^{-t})
	\end{align*}
	This last estimate is valid for any unit vector $w\in \xi_{n-1}$, hence the desired estimate for $\dist(\xi_{n-1},W_{d-1})$ follows.
\end{proof}
Note that in the proof we did not use anything about $\xi_1$, so in fact the assumption in \autoref{prop:control_of_planes_under_unipotent_domination} could have been replaced by
\[
	\norm{T^k w} \leq e^{-t}\left( \frac{k^d}{d!} + O(k^{d-1}) \right)
\]
for all unit vectors $w\in \xi_{n-1}$.

%%					end of subsec: Adapted metrics
%%=============================================================================

%%=============================================================================
%%					start of subsec: Adapted metrics and domination

\subsection{Adapted metrics and domination}
	\label{ssec:adapted_metrics_and_domination}

We will now use the above results to control the dynamics of the bundle in the cusp.

\begin{proposition}[Stable spaces near the cusp]
	\label{prop:stable_spaces_near_the_cusp}
	Let $p_0\in \partial\widetilde{X}$ be a parabolic boundary point, let $T\in G \subset \GL(V)$ be the corresponding monodromy, and $N:=\log T$ with weight filtration $W_\bullet(N)$.
	Then $p_0$ has a neighborhood $U_0\subset \partial \widetilde{X}$ and constants $C,\ve>0$ such that
	\[
		\dist(W_{d-1}(N), \xi_{n-1}(p)) \leq C\dist(p_0,p)^{\ve} \quad \forall p\in U_0
	\]
	where the distance between subspaces is measured relative to the metric on $V$ at $x_0$, and between points on the boundary using the visual metric from $x_0$.

	In particular $\xi_{n-1}(p_0)=W_{d-1}(N)$.
\end{proposition}
\begin{proof}
	We fixed a horoball neighborhood of the projection of $p_0$ to $X$, which we lift to $\widetilde{X}$ as a horoball neighborhood of $p_0$ and denote by $V_0$.
	Any geodesic on $\widetilde{X}$ starting at $x_0$ either enters $V_0$ or not, and let $t_0>0$ denote the upper bound on time such that every geodesic entering the horoball does so before time $t_0$.
	Further shrinking $V_0$ to a smaller horoball $V_1$ to be determined below, we take $U_0$ to be the neighborhood of $p_0$ such that if $p\in U_0$, the geodesic from $x_0$ to $p$ must enter $V_1$ (see also \autoref{fig:figures_looking_into_cusp}).
	The requirements on the smaller horoball $V_1$ are the following:
	\begin{itemize}
		\item A geodesic starting at $x_0$ and entering $V_1$ must spend at least time $t_1$ in $V_0$, where $t_1$ is such that $\exp({\ve t_1})$ dominates all constants appearing below, which only depend on $V_0$.
		\item There exist uniform constants such that for any $p\in U_0$ and its opposite $p':=op_{x_0}(p)$, the angle between $\xi_1(p')$ and $\xi_{n-1}(p)$ is uniformly bounded below when measured for the metric at $x_0$.
	\end{itemize}
	Consider now a geodesic starting at $x_0$ and going to $p\in U_0$ with $p\neq p_0$.
	Let $x',x''\in V_0$ be the points of entry, resp. exit, of the geodesic from the horoball.
	The size of vectors at $x_0$ and $x'$ are comparable, up to constants depending on the time $t_0$ which it takes to enter the horoball $V_0$ starting at $x_0$.
	Further, let $t$ be the distance between $x'$ and $x''$, which by assumption satisfies $t>t_1$.
	We can work in the upper half-plane model and assume $x'=\sqrt{-1}$ and $x''=\sqrt{-1}+(k + O(1))$ with $\log k = t + O(1)$ for some integer $k$.
	We then have that $\log \norm{v}_{h(x'')} = \log \norm{v}_{h(\sqrt{-1}+k)} + O(1)$, so we can do all calculations assuming $x''=\sqrt{-1}+k$.

	Because we have the domination property for the geodesic connecting $x_0$ and $x''$, we deduce that for $h(x_0)$-unit vectors $w\in \xi_1(p'),v\in \xi_{n-1}(p)$ (with $p'=op_{x_0}(p)$)
	\[
		\log \norm{w}_{h(x'')} - \log \norm{v}_{h(x'')} \geq \ve \cdot t -C
	\]
	for constants determined by the domination property.
	Next, replacing the metric at $x_0$ by that at $x'$ only gives a $O(1)$ loss, and we know that $h(x'')=T^kh(x')$, so we conclude that
	\[
		\log \norm{T^{-k}w}_{h(x')} - \log \norm{T^{-k}v}_{h(x')} \geq \ve \cdot t -C'
	\]
	We conclude using \autoref{prop:control_of_planes_under_unipotent_domination} that
	\begin{align}
		\label{eqn:distance_cusp_using_Rex}
		\dist(\xi_{n-1}(p),W_{d-1}(N)) \leq c''\frac{1}{|k|^{\ve}}		
	\end{align}
	which combined with the estimate in the hyperbolic plane in \autoref{eqn:disk_half_plane_mapping_boundary} shows that
	\[
		\dist(\xi_{n-1}(p),W_{d-1}(N)) \leq c''' \theta^{\ve}
	\]
	where $\theta$ is the angle between $p_0$ and $p$ measured from $x_0$.
\end{proof}

We can now show that adapted metrics reflect the dynamics at times when the geodesic flow visits the cusp.

\begin{theorem}[Adapted metrics give domination in the cusp]
	\label{thm:adapted_metrics_give_domination_in_the_cusp}
	Suppose that $\rho$ is a log-dominated representation and $\bV_\rho$ is equipped with an adapted metric.
	Then the domination condition in \autoref{def:log_anosov_representation} holds for starting points in the compact part, and all times.
\end{theorem}

\begin{proof}
% [Proof of \autoref{thm:adapted_metrics_give_domination_in_the_cusp}]
	Recall that $X_K\subset X$ denotes a fixed choice of compact part.
	Any geodesic $\gamma$ from $x'\in X_{K}$ to $x\in X\setminus X_{K}$ can be approximated by the concatenation of two geodesics: $\gamma'$ from $x'$ to the basepoint $x_0$, and $\gamma''$ from $x_0$ to $x$, such that $length(\gamma)=length(\gamma')+length(\gamma'')+O(1)$ where the implied constant only depends on the surface, and furthermore $\gamma$ and $\gamma'\cup\gamma''$ are homotopic.
	Moreover, we can assume that for a uniform constant $t_0>0$ depending on the choice of compact part $X_{K}$, the geodesic $\gamma''$ spends in the beginning at most time $t_0$ in the compact part, and then the rest in a fixed neighborhood of a single cusp.
	By modifying the constants in the domination condition, it suffices to prove the claim only for such geodesics $\gamma''$ which start at $x_0$ and aim directly into the cusp.

	For any $t_1>t_0$, by modifying again the constants in the domination condition, it suffices to prove the claim under the further assumption that the geodesic $\gamma''$, when continued indefinitely into the future, is in the same cusp for at least the time interval $[t_0,t_1]$ (with its position at time $0$ being $x_0$).

\begin{figure}[!htbp]
	\centering
	\includegraphics[width=0.3\linewidth]{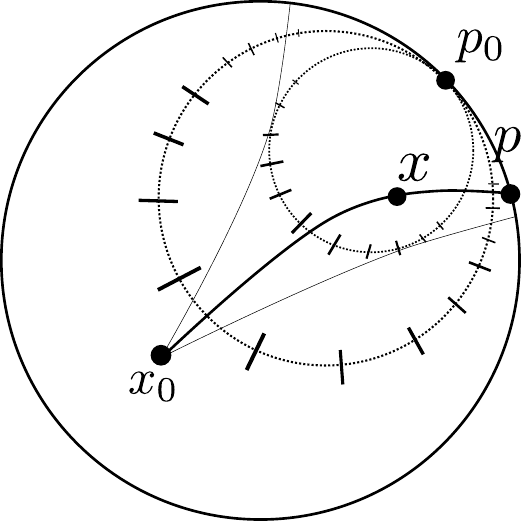}
	\hskip 1em
	\includegraphics[width=0.6\linewidth]{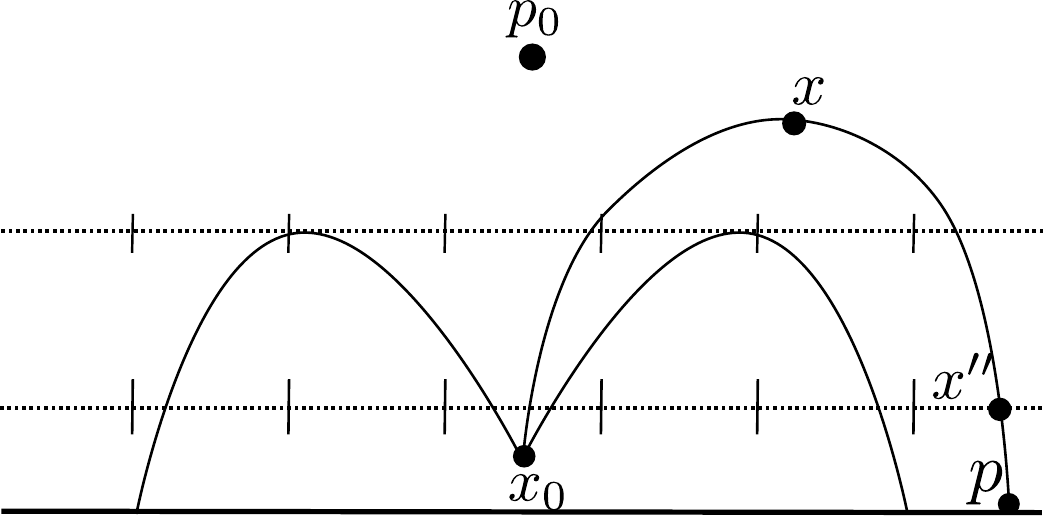}
	\caption{Looking into the cusp from $x_0$.
	Two horocycles depicted, we are only interested in geodesics that cross the second one, i.e. go sufficiently deep into the cusp.
	Left: Disc model; Right: Upper half-space model.}
	\label{fig:figures_looking_into_cusp}
\end{figure}
Let $x''$ be the point where the geodesic from $x_0$ to $p$ exits the horoball.
Setting $r:=|\Re x''|$, we know that
\[
		\Im x \leq r + O(1) \quad\text{ and }\quad  |\Re x| \leq r
\]
for any $x$ which is in the cusp, and on the geodesic between $x_0$ and $p$.
From \autoref{prop:stable_spaces_near_the_cusp} we know that
\begin{align}
	\label{eqn:distance_xi_W_r_proof}
	\dist(\xi_{n-1}(p),W_{d-1}(N)) \leqapprox r^{-\ve}	
\end{align}
using \autoref{eqn:distance_cusp_using_Rex}.

For $p'=op_{x_0}(p)$, let $w\in \xi_1(p')$ be a unit vector for the $h(x_0)$ metric.
By taking a sufficiently small neighborhood of $p$ as in the proof of \autoref{prop:stable_spaces_near_the_cusp}, we can assume that the angle between $\xi_1(p')$ and $W_{d-1}(N)$ is uniformly bounded below (for the $h(x_0)$ metric) and hence using the estimate for strictly adapted metrics in \autoref{prop:estimates_for_strictly_adapted_metrics}, for vectors in $V_d$, we find
\[
	\norm{v}_{h(x)} \geqapprox 
	\Im (x)^{\frac{d}{2}} 
	\left[1 + 
	\left(\frac{|\Re x|}{(\Im x)^{1/2}}\right)^{d}\right].
\]
Suppose now that $w\in \xi_{n-1}(p)$ is a unit vector for $h(x_0)$; decompose it as $w=w_d + w_{d-1}$ with $w_d\in V_d, w_{d-1}\in W_{d-1}(N)$, both of norm bounded by $1$ for $h(x_0)$.
Using \autoref{eqn:distance_xi_W_r_proof} we find
\begin{align*}
	\norm{w}_{h(x)}& \leq \norm{w_d}_{h(x)} + \norm{w_{d-1}}_{h(x)}\\
	& \leqapprox r^{-\ve}\norm{v}_{h(x)} + \norm{w_{d-1}}_{h(x)}
	% & \leqapprox r^{-\ve}\left( \Im (x)^{\frac{d}{2}} + \Im(x)^{-\frac{d}{2}}|\Re(x)|^{d} \right) + \\
	% & + \left( \Im (x)^{\frac{d-1}{2}} + \Im(x)^{-\frac{d-1}{2}}\cdot |\Re x|^{d-1} \right)
\end{align*}
The inequality we want to show is (in exponentiated form)
\[
	\norm{v}_{h(x)} \geqapprox \norm{w}_{h(x)}\cdot e^{\ve' \dist(x_0,x)}
\]
Recall from \autoref{eqn:hyperbolic_distance_tau_approximation} that $\dist(x_0,x) =  2\log(\max(|\Re x|, \Im x)) - \log(\Im x) + O(1)$, or in more symmetric form
\[
	e^{\dist(x_0,x)} \leqapprox \max\left(\frac{|\Re x|}{(\Im x)^{1/2}}, \left(\Im x\right)^{1/2}\right)^{2}.
\]
% so it suffices to show
% \[
% 	\norm{v}_{h(x)} \geqapprox \norm{w}_{h(x)}\cdot \max(|\Re(x)|, \Im(x))^{\ve'}
% \]
The quantity $\norm{w}_{h(x)}$ was estimated above using two terms, and it is clear that the term $r^{-\ve}\norm{v}_{h(x)}$ satisfies the desired inequality since $r +O(1) \geq \max(\Re |x|,\Im x)$, and we can work in the regime $|\Re x|\geq 1$ to avoid degenerate situations.

It suffices therefore to show, estimating $\norm{w_{d-1}}_{h(x)}$ using \autoref{prop:estimates_for_strictly_adapted_metrics}:
\begin{multline*}
	\Im (x)^{\frac{d}{2}} 
	\left[1 + 
			\left(\frac{|\Re x|}{(\Im x)^{1/2}}\right)^{d}\right]
	\geqapprox \\
	\geqapprox \max\left(\frac{|\Re x|}{(\Im x)^{1/2}}, \left(\Im x\right)^{1/2}\right)^{\ve}
	\cdot
	\Im (x)^{\frac{d-1}{2}} 
	\cdot
	\left[
	1 + 
	\left(\frac{|\Re x|}{(\Im x)^{1/2}}\right)^{d-1}\right]
	% \Im (x)^{\frac{d}{2}} + \Im(x)^{-\frac{d}{2}}|\Re(x)|^{d} \geqapprox \\
	% \geqapprox \max(|\Re(x)|, \Im(x))^{\ve'} \cdot \left( \Im (x)^{\frac{d-1}{2}} + \Im(x)^{-\frac{d-1}{2}}\cdot |\Re x|^{d-1} \right)
\end{multline*}
for some $\ve>0$.
Canceling the $(\Im x)^{d/2}$ factor and moving the $\Im(x)^{-1/2}$ to the left side reduces the inequality to
\begin{multline*}
	\Im (x)^{\frac{1}{2}} 
	\cdot 
	\left[1 + 
			\left(\frac{|\Re x|}{(\Im x)^{1/2}}\right)^{d}\right]
	\geqapprox \\
	\geqapprox \max\left(\frac{|\Re x|}{(\Im x)^{1/2}}, \left(\Im x\right)^{1/2}\right)^{\ve}
	\cdot
	\left[
	1 + 
	\left(\frac{|\Re x|}{(\Im x)^{\frac{1}{2}}}\right)^{d-1}\right]
\end{multline*}
It is now clear that $\ve=1$ works for this inequality, by considering the two cases depending on which of the two quantities inside the $\max$ is maximal.
Note that $\Im x\geq A\geq 1$ by assumption.
\end{proof}

% which after multiplication by $\Im(x)^{\frac{d}{2}}$ becomes
% \begin{multline*}
% 	\Im (x)^{d} + |\Re(x)|^{d} \geqapprox \\
% 	\geqapprox \max(|\Re(x)|, \Im(x))^{\ve'} \cdot \Im(x)^{-\frac{1}{2}} \left( \Im (x)^{d} + |\Re x|^{d-1} \right)
% \end{multline*}
% for some choice of $\ve'>0$.
% There are now two cases.

% If $\Im(x)^d\geq |\Re(x)|^{d-1}$ then $\Im(x)\geq |\Re(x)|^{\frac{d-1}{d}}$ so $\Im(x)^{-\frac{1}{2}}\leq |\Re(x)|^{-\frac{d-1}{2d}}$ and the inequality holds for any $\ve'<\frac{d-1}{2d}$.
% If $\Im(x)^d\leq |\Re(x)|^{d-1}$, then $\Im(x)\leq |\Re(x)|^{\frac{d-1}{d}}\leq |\Re(x)|$ (again assuming $|\Re(x)|\geq 1$) and any $\ve'<1$ works.

% This concludes the proof.

%%					end of subsec: Adapted metrics and domination
%%=============================================================================

%%%%%%%%%%%%%%%%%%%%%%%%%%%%%%%%%%%%%%%%%%%%%%%%%%%%%%%%%%%%%%%%%%%%%%%%%%%%%%%
% End of Section: Unipotent dynamics
%%%%%%%%%%%%%%%%%%%%%%%%%%%%%%%%%%%%%%%%%%%%%%%%%%%%%%%%%%%%%%%%%%%%%%%%%%%%%%%

%====================================================
%====================================================
%		Bibliography
%====================================================
%====================================================
\bibliographystyle{sfilip}
\bibliography{14_families_3CY}
%====================================================
\end{document}